\documentclass[twoside,11pt]{article}
\usepackage{amsmath}
\usepackage{amsthm}
\usepackage{amssymb}
\usepackage{amscd}
\usepackage[british]{babel}
\usepackage{graphicx}
\usepackage{times}
\usepackage{comment}
\usepackage{url}
\usepackage{psfrag}
\usepackage{color}
\usepackage{colortbl}
\usepackage{comment}
\usepackage{multicol}
\usepackage{multirow}
\usepackage{pdflscape}
\usepackage{verbatim}
\usepackage[figuresright]{rotating}

\input{macros.sty}

\title{Rigorous computer assisted application of KAM theory: \\ a modern approach}

\date{31st December 2015}

\author{
Jordi-Llu\'{\i}s Figueras\thanks{figueras@math.uu.se}\\
{\footnotesize Departament of Mathematics} \\
{\footnotesize Uppsala University} \\
{\footnotesize Box 480, 751 06 Uppsala (Sweden).}
\and
Alex Haro\thanks{alex@maia.ub.es}\\
{\footnotesize Departament de Matem\`atica Aplicada i An\`alisi} \\
{\footnotesize Universitat de Barcelona} \\
{\footnotesize Gran Via 585, 08007 Barcelona (Spain).}
\and
Alejandro Luque\thanks{luque@icmat.es}\\
{\footnotesize Instituto de Ciencias Matem\'aticas} \\
{\footnotesize Consejo Superior de Investigaciones Cient\'{\i}ficas} \\
{\footnotesize C/ Nicol\'as Cabrera 13-15, 28049 Madrid (Spain).}}

\begin{document}
\maketitle

\thispagestyle{empty}

\begin{abstract}
In this paper we present and illustrate a general methodology to apply KAM theory in particular
problems, based on an {\em a posteriori} approach.  We focus on the existence of real-analytic quasi-periodic Lagrangian
invariant tori for symplectic maps. The purpose is to verify the hypotheses of a KAM theorem in an a posteriori format: 
given a parameterization of an approximately invariant torus, we have to check non-resonance (Diophantine) conditions,
non-degeneracy conditions and certain inequalities to hold.
To check such inequalities we require to control the analytic norm of some
functions that depend on the map, the ambient structure and
the parameterization.  To this end, we propose an efficient computer assisted
methodology, using fast Fourier transform, having the same 
asymptotic cost of using the parameterization method for obtaining numerical approximations of invariant tori.
We illustrate our methodology by proving the existence of invariant curves
for the standard map (up to $\eps=0.9716$),
meandering curves for the non-twist standard map
and 2-dimensional tori for the Froeschl\'e map.
\end{abstract}



\newpage

\tableofcontents

\section{Introduction}\label{sec:intro}

KAM theory is concerned with the existence and stability of quasi-periodic
motions in different contexts of dynamical systems such as symplectic
maps, Hamiltonian systems, reversible systems or volume-preserving systems,
just to mention a few. The foundations of the theory started with the
celebrated works of A.N.  Kolmogo\-rov~\cite{Kolmogorov54}, V.I.
Arnold~\cite{Arnold63a}, and J.K. Moser~\cite{Moser62}, so that the acronym KAM
is used in their honor.  These pioneer papers sowed the seed of a subject of
remarkable importance in dynamical systems and, nowadays, KAM theory is a vast
area of research that involves a large collection of methods.  Actually, there
are many excellent surveys covering different points of view in the theory
(e.g.~\cite{Arnold63b, Bost86, BroerHS96, Llave01, Gallavotti83, Moser67}).

Classic KAM methods typically deal with a perturbative setting in
such a way that the problem is written as a perturbation of an
integrable system (in the sense that it has a continuous
family of invariant tori). They are based on the use of canonical transformations to
simplify the expression of the problem. To this end, one
takes advantage of the existence of action-angle-like coordinates for the unperturbed
system. This is a source of different
shortcomings and limitations in the study of particular problems,
mainly related to the fact that many systems are
non-perturbative. 
For example,
in some cases it is possible to identify an integrable approximation of a
given system but the remaining part cannot be considered as an
arbitrarily small perturbation.
Moreover, given a particular perturbative problem,
in general it is very complicated to establish
action-angle variables for the unperturbed system. Such
action-angle variables can be defined implicitly, become singular
or introduce problems of regularity.

In spite of the previous difficulties,
classic KAM methods have been successfully applied 
in several problems.
The interested reader is referred to Section 1.4 in~\cite{CellettiC07}
for a brief history and references of the application of
{\KAM} theory, and to~\cite{CellettiC97,Locatelli98,LocatelliG00}
for computer assisted proofs in problems of celestial mechanics.
A prominent example is the persistence
of the golden invariant  curve of the standard map
(c.f. Section~\ref{ssec:standard}
and notation therein).
From the numerical point of view, the persistence of this golden curve has been
considered for example in~\cite{Chirikov79,Greene79,MacKay93}
observing that the breakdown takes place around $\eps_c \simeq 0.97163540324$.
Upper bounds for $\eps_c$ were provided in~\cite{MacKayP85,Mather84} by Converse {\KAM} Theory.
A quite sharp non-existence result was reported in~\cite{Jungreis91}, proving
that the standard map has no rotational invariant circles for several parameter values including $\eps = 0.9718$.
KAM theory provides lower bounds for the critical value $\eps_c$.
This was already considered
by Herman~\cite{Herman86}, obtaining that the golden curve persists for $\eps \leq 0.029$.
Later, a computer assisted proof was given in~\cite{CellettiC88}
proving existence of the invariant curve for $\eps < 0.68$
using Lindstedt series.
This lower bound was improved in~\cite{LlaveR90,LlaveR91} extending the
result up to $\eps \leq 0.91$.

An alternative to the classic approach is the use of the parameterization
method.  Instead of performing canonical transformations, 
the strategy consists in solving the invariance equation directly
by correcting an approximately invariant object. Such correction
is obtained iteratively by 
considering the linearized equation around
the previous approximately invariant torus.
The parameterization method is suitable for studying
existence of invariant tori without using neither action-angle variables nor
being in a perturbative setting.  We point out that the geometry of the
problem plays an important role in the study of these equations. Such geometric
approach, also referred to as KAM theory \emph{without action-angle variables}
or \emph{a posteriori} KAM theory, was suggested by R. de la Llave
in~\cite{Llave01} (following long-time developed ideas, e.g.~\cite{CellettiC97,
JorbaLZ99, 
Moser66b,
Russmann76b,SalamonZ89,Zehnder76}) and a
complete proof was presented in \cite{GonzalezJLV05}.  This approach has been
later extended to other contexts, such as the study of
lower dimensional (isotropic) invariant tori that are
hyperbolic~\cite{FontichLS09} or elliptic~\cite{LuqueV11}, the case of non-twist
invariant tori in degenerate systems~\cite{GonzalezHL13} or, more recently,
dissipative systems~\cite{CallejaCLa,CanadellH15a}.
A remarkable advantage of
the parameterization method is that the steps of the proof allow us to obtain
very fast and efficient numerical methods for the approximation of
quasi-periodic invariant tori (e.g.~\cite{CallejaL10, 
FoxM14,
HuguetLS12}).  We refer the reader to the recent survey~\cite{HaroCFLM} for a
detailed discussion on the numerical implementation of the method and examples.

The goal of this
paper is to present and illustrate a general methodology to apply the KAM
theorem in specific problems. We focus on the
existence of analytic quasi-periodic Lagrangian invariant tori of symplectic
maps.  We resort to a revisited version of the \emph{a posteriori} KAM result
presented in~\cite{GonzalezJLV05}. As usual in KAM theory, the main hypotheses
consist in checking non-resonance (Diophantine) conditions, non-degeneracy
conditions and also asking certain inequalities to hold.  To check such
inequalities we require to control the analytic norm of some functions that
depend on the known objects (the map, the ambient structure and the initial
parameterization).  To this end, we propose a rigorous computer assisted methodology
based on the use of fast Fourier transform. 
An important consequence of our methodology is that the application of the KAM
theorem is performed in a very fast way.  Indeed, with the same asymptotic cost
of using the parameterization method to obtain numerical approximations of
invariant tori.

It is worth mentioning that computer assisted analysis has played an important
role to achieve remarkable results in the literature.  Among them, we
highlight: 
the proof of the Feigenbaum conjecture \cite{KochSW96,Lanford82};
Rigorous interval methods in quantum mechanics \cite{FeffermanS96};
the proof that the Lorenz attractor exists \cite{Tucker99}; 
the proof of the double bubbling conjecture \cite{HassS00},
and the existence of singular solutions in fluid dynamics \cite{CastroCFGG13}.
Computer assisted methods in analysis
rely on the fact that one can define a rigorous interval arithmetic on 
computers. These intervals have computer representable floating point
numbers as end-points and all basic operations as addition, subtraction, 
multiplication, division and composition of
standard functions
(e.g $\exp$, $\cos$, $\sin$, $\log$) satisfy the \emph{isotonicity inclusion principle} (the image of any two nested intervals is nested) and
the \emph{range enclosure principle} (the range of any function is enclosed with the image of the domain under the action of the natural interval
extension). We refer the reader to \cite{Tucker11} for more details.

Finally, we describe the organization of the paper and
we briefly summarize the content of each section:
\begin{itemize}
\item 
In Section~\ref{sec:theKAMteo} we introduce some elementary geometric objects
(Section~\ref{ssec:geom}), set up notation and norms
(Section~\ref{sec-anal-prelims}), and present a detailed statement of the KAM
theorem (Section~\ref{ssec:teoKAMestate}). Then we present a full proof of this
result (Section~\ref{ssec:proof:KAM}).  This is necessary in order to link
the different expressions that appear in the constants of the theorem with
their corresponding geometric object or equation. We give explicit and
sharp estimates for all constants quantifying the hypotheses
of the theorem.
\item 
In Section~\ref{sec:DFT} we control the difference between an analytic function
$f$ on the torus and its discrete Fourier approximation $\tilde f$, and we
present several technical results that allow us to control (with explicit
constants) the analytic norm of $\tilde f-f$. 
More concretely,
if $f$ is an analytic and bounded function on the complex strip
of size $\hat \rho>0$, then
\[
\snorm{\tilde f-f}_\rho \leq C_{\NF}(\rho,\hat \rho) \snorm{f}_{\hat \rho},
\]
for every $0\leq \rho < \hat \rho$,
where $C_{\NF}(\rho,\hat \rho)$ is an explicit constant that depends also
on the dimension and the number of Fourier coefficients.
This result is motivated by the previous
work~\cite{Epstein04} and related ideas have been used in~\cite{Minton13,SchenkelWW00}.
Our aim is that this section can be read independently, in spite of
some
notation introduced in Section~\ref{sec-anal-prelims} regarding Fourier series and norms.
\item  
In Section~\ref{sec:CAPlemmas} we consider additional technical results that
allow us to apply the KAM theorem in 
an effective and efficient way.
On the one hand, we present an approach to obtain a positive measure
set of Diophantine vectors close to a given vector, possibly obtained
numerically
(Section~\ref{ssec:diophantine}).  On the other hand, we present an
improvement of the classic R\"ussmann estimates
(Section~\ref{ssec:superrussmann}).
To take into account the effect of small divisors, we compute
the first elements explicitly and
then we control the remaining tail analytically.

\item 
In Section~\ref{sec:validation:algorithm} we present the main
methodology to apply the KAM result: the validation algorithm.
The validation
procedure is performed on a sampling of an approximately invariant
torus obtained numerically.
We also require a finite
amount of input data
characterizing
the geometric information of the problem. Suitable values 
for these parameters can be obtained following an heuristic method
explained in Appendix~\ref{app:optimization}.

\item
In Section~\ref{sec:examples},
we apply the validation algorithm to several examples, thus highlighting different features of our approach.

In order to illustrate the reliability of the method, we consider the standard map (Section~\ref{ssec:standard}) and 
prove that the golden invariant curve exists up to $\eps=0.9716$, thus establishing 
a new lower bound to the so-called Greene critical 
value~\cite{Greene79} $\eps_c\simeq 0.97163540324$. 
Moreover, we have also proved the existence of other rotational invariant curves with different rotation numbers.

An important feature of the method is that it can be applied to invariant curves that are not graphs over the angular coordinate. 
We consider the non-twist standard map and prove the existence of so-called meandering invariant curves
(Section~\ref{ssec:nontwist:standard}). 

We finally consider a higher dimensional example. We prove the existence
of $2$-dimensional invariant tori for the Froeschl\'e map, a $4$-dimensional symplectic map consisting in two coupled standard maps
(Section~\ref{ssec:froeschle}). 

\end{itemize}

\section{A KAM theorem for Lagrangian invariant tori of exact symplectic maps}
\label{sec:theKAMteo}

In this section we present an \emph{a posteriori} {\KAM} theorem for Lagrangian
invariant tori of exact symplectic maps. 
The result was first
obtained in~\cite{GonzalezJLV05} and it is a version of Kolmorogov
theorem~\cite{Kolmogorov54} using neither action-angle coordinates nor a
perturbative setting. The specific statement given here, with explicit and
sharp estimates, is a slightly modified version of the KAM Theorem discussed in
Chapter 4 of~\cite{HaroCFLM}.  Roughly speaking, the \emph{a posteriori}
result reads as follows: \emph{if we have a good enough approximation of an
invariant torus with frequency $\omega$, then, under certain non-degeneracy and
non-resonance conditions, there exists a true invariant torus nearby.}

After setting the problem and geometrical background in
Section~\ref{ssec:geom}, we introduce some basic notation regarding Banach
spaces, norms and cohomological equations in Section~\ref{sec-anal-prelims}.
In Section~\ref{ssec:teoKAMestate} we present the statement of a KAM theorem
for existence (and persistence) of Lagrangian invariant torus having
Diophantine frequencies.  In the proof of the main result, given in
Section~\ref{ssec:proof:KAM}, we pay special attention to compute explicitly
all constants appearing during the
process and to obtain optimal and sharp estimates.

\subsection{Geometric setting and invariant tori}\label{ssec:geom}

We denote $\TT^n=\RR^n/\ZZ^n$ the $n$-dimensional torus 
with covering space.
The ambient manifold is a $2n$-dimensional annulus $\A
\subset \TT^n \times \RR^{n}$ with covering space
$\tilde \A \subset \RR^{2n}$. The coordinates on $\A$ are denoted by
$z=(z_1,\ldots,z_{2n})=(x,y)$, with $x=(x_1,\ldots,x_n)$ and
$y=(y_1,\ldots,y_{n})$.  A function $u:\RR^n\to \RR$ is $1$-periodic if
$u(\theta+e)= u(\theta)$ for all $\theta\in\RR^n$ and $e\in\ZZ^n$.  
We abuse notation and denote it as $u:\TT^n\to \RR$.
Similarly, a function $g:\cA\to \RR$
is $1$-periodic in $x$ if $g(x+e,y)= g(x,y)$ for all $x\in\RR^n$ and
$e\in\ZZ^n$.
We abuse notation and denote it as $g:\A\to\RR$.

In the following we assume that $\A$ is endowed with an exact symplectic form
$\sform={\rm d} \aform$ for a certain $1$-form $\aform$.  For any point
$z\in \A$, we write the matrix representation of the 1-form $\aform_z$ and the
2-form $\sform_z$ as 
\begin{equation}\label{def-Omega}
a(z)=(a_1(z)~\ldots~ a_{2n}(z))^\top, \qquad \text{and}\qquad \Omega(z)= \Dif
a(z)^\top  - \Dif a(z),
\end{equation}
respectively. 
Notice that $\det\Omega(z)\neq 0$.

\begin{remark}\label{rem:standardforms}
The prototype example of symplectic structure is 
the \emph{standard symplectic structure} on $\TT^n\times \RR^{n}$: $\sform_0=
\sum_{i= 1}^n  {\rm d} z_{n+i}\wedge {\rm d}z_i$.  An action form for
$\sform_0$ is ${\aform_0}= \sum_{i= 1}^n z_{n+i}\ {\rm d} z_i$.  The matrix
representations of $\aform_0$ and $\sform_0$ are, respectively,
\[
a_0(z)=  \begin{pmatrix} O_n & I_n \\ O_n & O_n \end{pmatrix}z, \qquad
\Omega_0= \begin{pmatrix} O_n & -I_n \\ I_n & O_n \end{pmatrix}.
\]
\end{remark}

A map $F:\A\rightarrow\A$ is \emph{symplectic} if $F^*\sform = \sform$.  A
symplectic map $F:\A\rightarrow\A$ is \emph{exact} if there is a smooth
function $S:\A\rightarrow\RR$, called \emph{primitive function of $F$}, such
that  $F^*\aform-\aform = {\rm d} S$. In coordinates, the symplectic and the
exact symplectic properties of a map $F$ are equivalent to
\begin{equation}\label{symplectic_condition}
   \Dif F(z)^\top\ \Omega(F(z)) \ \Dif F(z) = \Omega(z), \quad
   \forall z\in\A ,
   \end{equation}
and
   \begin{equation}\label{eq:exact}
   \Dif  S(z) = a(F(z))^\top \Dif F(z) - a(z)^\top , \quad
   \forall z\in\A ,
\end{equation}
respectively. A map $F:\A\rightarrow\A$ is \emph{homotopic to the
identity} if $F(x,y)-(x,0)$ is $1$-periodic in $x$.

Given an embedding $K:\TT^n\to \A$, 
called parameterization from now on,
we say that
$K(\TT^n)$ is an \emph{$F$-invariant torus} with frequency $\omega \in \RR^n$ if
\begin{equation}\label{eq:inv:rot}
F (K(\theta)) = K (\theta+\omega).
\end{equation}
For convenience, we denote $\Rot (\theta)=\theta + \omega$ the rigid rotation
of frequency $\omega$.
In case that
$\omega$ is rationally independent (i.e., $k \cdot \omega \notin \ZZ$ for all
$k \in \ZZ^d\backslash\{0\}$) then the rotation $\Shift{\omega}$ is 
\emph{ergodic} and the invariant torus $K(\TT^n)$ is 
\emph{quasi-periodic}. Finally, the parameterization $K:\TT^n\to \A$ is
\emph{homotopic to the zero section} if $K(\theta)-(\theta,0)$ is $1$-periodic
in $\theta$.

\begin{remark}\label{rem:topology}
If $K$ is homotopic to the zero section, then $K(\TT^n)$ is called primary
tori. In the classic KAM perturbative setting, these objects correspond to
continuation of the planar tori that are present in the unperturbed problem.
The methodology presented in this paper can be adapted to deal with invariant
tori having other relative homotopies in a straightforward way.
\end{remark}

By taking derivatives at both sides of Equation~\eqref{eq:inv:rot}
we observe that the tangent bundle is invariant. Indeed, 
\begin{equation}\label{eq:inv:rot:der}
\Dif F(K(\theta)) \Dif K(\theta)= \Dif K(\theta+\omega).
\end{equation}
Given a parameterization $K$ as described above, we consider the
pullback $K^*\sform$.
Its matrix representation at a point $K(\theta)$ is
\begin{equation} 
\label{def-OK}
\sub{\Omega}{K}(\theta) =  \Dif K(\theta)^\top\
\Omega(K(\theta)) \ \Dif K(\theta).
\end{equation}
It is well known (c.f.~\cite{Moser66c}) that if $K(\TT^n)$ is 
a quasi-periodic $F$-invariant torus, then $\sub{\Omega}{K}(\theta)=O_n$ for every $\theta \in \TT^n$.
In combination with Equation~\eqref{eq:inv:rot:der}, this means that we have a Lagrangian invariant subbundle.

Roughly speaking, the parameterization method for proving existence of
quasi-periodic $F$-invariant tori consists in
studying the linearized invariance equation around an approximate solution.
The invariance of $\Dif K(\theta)$ in Equation~\eqref{eq:inv:rot:der} suggests that these vectors will help us to
obtain a suitable frame to write the map $\Dif F$.  The fact that every
Lagrangian subspace has a Lagrangian complementary is the starting point of the
general construction discussed in~\cite{HaroCFLM} which is followed in this
paper (for previous constructions we refer
to~\cite{GonzalezJLV05,FontichLS09,LuqueV11,GonzalezHL13}). This construction
goes as follows: given a map $N_0: \TT^n \rightarrow \RR^{2n \times n}$ such
that 
\begin{equation}\label{eq:invLKONK}
\det (\Dif K(\theta)^\top \Omega(K(\theta)) N_0(\theta)) \neq 0
\end{equation}
it turns out that the
\emph{frame map}
$P: \TT^n \rightarrow \RR^{2n \times 2n}$, given by
\begin{equation}\label{eq:def-P}
P(\theta)=
\begin{pmatrix}
\Dif K(\theta) & N(\theta)
\end{pmatrix},
\end{equation}
with
\begin{align}
N(\theta)={} & \Dif K(\theta) A(\theta)+N_0(\theta) B(\theta), \label{eq:def-N}
\\ 
B(\theta)={} & -(\Dif K(\theta)^\top \Omega(K(\theta)) N_0(\theta))^{-1},
\text{and}
\label{eq:def-B} \\ 
A(\theta)={} & -\frac{1}{2} (B(\theta)^\top
N_0(\theta)^\top \Omega(K(\theta)) N_0(\theta) B(\theta)), \label{eq:def-A}
\end{align}
is a symplectic frame ($N$ is the Lagrangian complement of $\Dif
K$). 
Since the dynamics on the torus is ergodic, 
it follows that the symplectic frame in Equation~\eqref{eq:def-P} reduces the linearized dynamics
$\Dif F\comp K$ to a  block-triangular matrix
\begin{equation}\label{reducibility}
 P(\theta+\omega)^{-1} \Dif F(K(\theta)) P(\theta)
 = \Lambda(\theta), \qquad \Lambda(\theta)
 = \begin{pmatrix}
 I_n &  T(\theta) \\ O_n  & I_n
 \end{pmatrix},
\end{equation}
where the \emph{torsion matrix} $T: \TT^n \to \RR^{n\times n}$ is
given by
\begin{equation}\label{def-T}
T(\theta) = N(\theta+\omega)^\top \
\Omega(K(\theta+\omega)) \ 
\Dif F(K(\theta)) \ N(\theta).
\end{equation}

\begin{remark}
Of course, if we endow the annulus with 
additional structure (e.g. a Riemannian metric) we can obtain $N_0$ in a
natural way according to this structure. A summary of different approaches used
in the literature can be found in Chapter 4 of~\cite{HaroCFLM}.
\end{remark}

\subsection{Analytic functions and norms}\label{sec-anal-prelims}

In this paper we work with Banach spaces of real analytic functions in
complex neighborhoods of real domains. A complex strip of  $\TT^n$ of width
$\rho>0$ is defined as
\[
\TT^{n}_{\rho}= 
\left\{\theta\in \CC^n / \ZZ^n \ : \ \norm{\im{\,\theta_{i}}} < \rho, 
\,i=1,\dots, n\right\}.
\]
A function defined on $\TT^n$ is \emph{real
analytic} if it can be analytically extended to a complex strip $\TT_{\rho}^n$.

We consider analytic functions $u:\TT^n_\rho\to\CC$ such that they can be
continuously extended up to the boundary of $\TT^n_{\rho}$.  We endow these
functions with the norm
\begin{equation}\label{eq:norm1}
\snorm{u}_{\rho}  = 
\sup_{\theta \in \TT^n_{\rho}}  \norm{u(\theta)}.
\end{equation}
Moreover, we write the Fourier expansion
\[
u(\theta)=\sum_{k \in \ZZ^n} u_k \ee^{2\pi \ii k \cdot \theta}, \qquad u_k =
\int_{[0,1]^n} u(\theta) \ee^{-2 \pi \ii k \cdot \theta} \dif \theta,
\]
and we denote the average of $u$ as $\avg{u}= u_0 = \int_{[0,1]^n} u(\theta)
\dif \theta$. 
Then, we consider the Fourier norm
\begin{equation}\label{eq:normF}
\snorm{u}_{F,\rho}=\sum_{k \in \ZZ^n} | u_k| \ee^{2\pi |k|_1 \rho},
\end{equation}
where $|k|_1 = \sum_{i= 1}^n |k_i|$.  We observe that $\snorm{u}_{\rho} \leq
\snorm{u}_{F,\rho}$ for every $\rho>0$.  

A complex strip of $\A$ is a complex connected open neighborhood $\B \subset
(\CC^n / \ZZ^n) \times \CC^n$ of $\A$ that projects surjectively on
$\TT^n$.  A function defined on $\A$ is \emph{real
analytic} if it can be analytically extended to a complex strip 
$\B$.
Given an analytic function
$u: \B \rightarrow \CC$ we introduce the norm
\begin{equation}\label{eq:normB}
\snorm{u}_{\B}  = 
\sup_{z\in\B} \,
\norm{u(z)}.
\end{equation}

The previous definitions extend naturally to matrices. If $A$ is an $n_1 \times
n_2$ matrix of analytic functions on $\TT_{\rho}^n$ (resp. on $\B$), we extend
the norms in Equations~\eqref{eq:norm1} and~\eqref{eq:normF} (resp.
Equation~\eqref{eq:normB}) as follows
\begin{equation}\label{eq:norm2}
\snorm{A}_{\rho} = \max_{i=1,\ldots,n_1} \sum_{j=1}^{n_2}
\snorm{A_{i,j}}_{\rho}, \qquad
\snorm{A}_{F,\rho} = \max_{i=1,\ldots,n_1} \sum_{j=1}^{n_2}
\snorm{A_{i,j}}_{F,\rho},
\qquad
\mbox{(resp. $\snorm{A}_{\B}$)}.
\end{equation}

Notice that, if $F :\A \rightarrow \A$, $\Omega$ is the matrix representation
of $\sform$ and $a$ is the matrix representation of $\aform$, then
\[
\snorm{\Dif F}_{\B} = \max_{i=1,\ldots, 2n} \sum_{j=1}^{2n}
\Snorm{\frac{\partial F_i}{\partial z_j}}_{\B}, \qquad \snorm{\Dif^2 F}_{\B} 
   = \max_{i=1,\ldots, 2n} \sum_{j,k=1}^{2n}
\Snorm{\frac{\partial F_i}{\partial z_j\partial z_k}}_{\B},
\]
\[
\snorm{\Omega}_{\B} = \max_{i=1,\ldots, 2n} \sum_{j=1}^{2n}
\Snorm{\Omega_{i,j}}_{\B}, \qquad
\snorm{\Dif \Omega}_{\B} = \max_{i=1,\ldots, 2n} \sum_{j,k=1}^{2n}
\Snorm{\frac{\partial \Omega_{i,j}}{\partial z_k}}_{\B},
\]
\[
\snorm{\Dif a}_{\B} = \max_{i=1,\ldots, 2n} \sum_{j=1}^{2n}
\Snorm{\frac{\partial a_i}{\partial z_j}}_{\B}, \qquad
\snorm{\Dif^2 a}_{\B} = \max_{i=1,\ldots, 2n} \sum_{j,k=1}^{2n}
\Snorm{\frac{\partial^2 a_i}{\partial z_j\partial z_k}}_{\B},
\]

Finally, we introduce some useful notation regarding the so-called
\emph{cohomological equations} 
that play an important role in {\KAM} theory.
Given $\omega\in\RR^n$, we define the cohomology operator $\L$
on functions $u:\TT^n\to \RR$ as follows:
\begin{equation}\label{eq:calL}
\L\, u = u  - u\comp \Shift{\omega}.
\end{equation}
Then, the core of {\KAM} theory is the
cohomological equation
\begin{equation}\label{eq:standard}
\L u = v -\avg{v}, 
\end{equation}
for a given periodic function $v$.

Let us assume that $v$ is a continuous function and $\Shift{\omega}$ is ergodic.
If there exists a continuous zero-average solution of Equation~\eqref{eq:standard},
then it is unique and will be denoted by
$u = \R v$.
Note that the formal solution of Equation~\eqref{eq:standard} is
immediate. Actually, if $v$ has the Fourier
expansion $v(\theta)=\sum_{k \in \ZZ^n} \hat v_k \mathrm{e}^{2\pi
\mathrm{i} k \cdot \theta }$ and the dynamics is ergodic, then
\begin{equation}\label{eq:small:formal}
\R v(\theta) = \sum_{k \in \ZZ^n \backslash \{0\} } \hat u_k \mathrm{e}^{2\pi
\mathrm{i} k \cdot  \theta}, \qquad \hat u_k = \frac{\hat
v_k}{1-\mathrm{e}^{2\pi
\mathrm{i}  k \cdot \omega}}.
\end{equation}
In particular, this implies that $\R v =0$ if $v=0$.
The solutions of Equation~\eqref{eq:standard} differ by a constant
(the average). 

We point out that ergodicity is not enough to ensure regularity of the
solutions of cohomological equations.  This is related to the effect of the
small divisors $1-\mathrm{e}^{2\pi \mathrm{i}  k \cdot \omega}$ in
Equation~\eqref{eq:small:formal}.  
To deal with regularity, we
require stronger non-resonant conditions on the vector of frequencies. 
In this
paper, we consider the following classic condition:
\begin{definition}\label{def-Diophantine}
Given $\gamma>0$ and $\tau\geq n$, we say that $\omega\in\RR^n$ is a
$(\gamma, \tau)$-Diophantine vector of frequencies
if
\begin{equation}\label{eq:Diophantine}
\norm{k \cdot \omega-m}\ge \,\gamma\, |k|_1^{-\, \tau}, 
\qquad  \forall k\in\ZZ^n\backslash\{0\}, \, m\in\ZZ , 
\end{equation}
where $|k|_1 = \sum_{i= 1}^n |k_i|$.
\end{definition}

Finally, we recall the so-called R\"ussmann estimates to control the regularity
of the solutions of Equation~\eqref{eq:standard} (we refer the reader
to~\cite{Russmann75}).  If $v:\TT^n\to\RR$ is analytic, with $\snorm{v}_\rho<
\infty$ and $\omega$ satisfies~\eqref{eq:Diophantine}, then
\begin{equation}\label{eq:russ:first}
\snorm{\R v}_{\rho-\delta} \leq \tfrac{c_R}{\gamma \delta^\tau}\snorm{v}_\rho
\end{equation}
for
$0<\delta\leq \rho$.  In Lemma~\ref{lem-Russmann} we present 
an improvement
of the classic R\"ussmann constant $c_R$ with the help
of the computer.

The above definitions for $\L$ and $\R$ extend component-wise to
vector and matrix-valued functions. These extensions also
satisfy the R\"ussmann estimates. 

\subsection{Statement of the KAM theorem}\label{ssec:teoKAMestate}

At this point, we are ready to state sufficient conditions to guarantee the
existence of an $F$-invariant torus with fixed frequency close to an
approximately $F$-invariant torus. Theorems of this type are often
called \emph{a posteriori} results.
Notice that the hypotheses in Theorem \ref{theo:KAM} are tailored to be verified
with a finite amount of computations.

\begin{theorem}\label{theo:KAM}
Let us consider an exact symplectic structure $\sform = {\rm d} \aform$ on the
$n$-dimensional annulus $\A$, an exact symplectic map $F: \A \rightarrow \A$
homotopic to the identity and a frequency vector $\omega \in \RR^n$. Let us
assume that the following hypotheses hold:
\begin{itemize}
\item [$H_1$] The map $F$, the 1-form $\aform$ and the 2-form $\sform$ are real
analytic and can be analytically extended to
some complex strip $\B$ and continuously up to the boundary.  Moreover, there
are constants $\sub{c}{\Dif F}, \sub{c}{\Dif^2 F}, \sub{c}{\Omega},
\sub{c}{\Dif \Omega}, \sub{c}{\Dif a}$ and $\sub{c}{\Dif^2 a}$ such that
$\snorm{\Dif F}_{\B} \leq \sub{c}{\Dif F}$, $\snorm{\Dif^2 F}_{\B} \leq
\sub{c}{\Dif^2 F}$, $\snorm{\Omega}_{\B} \leq \sub{c}{\Omega}$, $\snorm{\Dif
\Omega}_{\B} \leq \sub{c}{\Dif \Omega}$, $\snorm{\Dif a}_{\B} \leq \sub{c}{\Dif
a}$, and $\snorm{\Dif^2 a}_{\B} \leq \sub{c}{\Dif^2 a}$.

\item[$H_2$]
There exists $K : \TT^n \rightarrow \A$, homotopic to the zero section, that
can be analytically extended to $\TT_\rho^n$ with $\rho>0$, and continuously up
to the boundary, with $K(\TT^n_\rho)\subset \B$. Moreover, there exist constants $\sub{\sigma}{\Dif K}$ and
$\sub{\sigma}{\Dif K^\top}$ such that
\[
\snorm{\Dif K}_\rho < \sub{\sigma}{\Dif K}, 
\qquad 
\snorm{\Dif K^\top}_\rho <
\sub{\sigma}{\Dif K^\top},
\qquad
\mathrm{dist}(K(\TT_\rho^n),\partial \B) > 0.
\]
Given two subsets $X,Y\in \CC^{2n}$, $\mathrm{dist}(X,Y)$ is defined as
$\inf\{|x-y|\,: ~x\in X,~y\in Y\}$,
where $|\cdot|$ is the maximum norm.

\item [$H_3$] There exists a map $N_0: \TT^n \rightarrow \RR^{2n \times n}$
that is real analytic and can be analytically extended to $\TT_\rho^n$, and
continuously up to the boundary. Moreover, there exist constants
$\sub{c}{N_0}$, $\sub{c}{N_0^\top}$, $\sub{c}{N_0^\top (\Omega \circ K) N_0}$,
and $\sub{\sigma}{B}$ such that
\[
\snorm{N_0}_{\rho} \leq \sub{c}{N_0}, \qquad
\snorm{N_0^\top}_{\rho} \leq \sub{c}{N_0^\top}, \qquad
\snorm{N_0^\top (\Omega \circ K) N_0}_\rho \leq 
\sub{c}{N_0^\top (\Omega \circ K) N_0}, \qquad
\snorm{B}_\rho < \sub{\sigma}{B},
\]
where $B(\theta) = -(\Dif K (\theta)^\top \Omega(K(\theta))
N_0(\theta))^{-1}$. 

\item [$H_4$] There exists $\sub{\sigma}{T}$ such that the matrix-valued map
\[
 T(\theta) = N(\theta+\omega)^\top \
 \Omega(K(\theta+\omega)) \ 
 \Dif F(K(\theta)) \ N(\theta)
\]
satisfies $|\avg{T}^{-1}| < \sub{\sigma}{T}$, where
\[
N(\theta)= \Dif K(\theta)
A(\theta)+N_0(\theta)B(\theta),
\]
with $A(\theta)= -\frac{1}{2} (B(\theta)^\top N_0(\theta)^\top
\Omega(K(\theta)) N_0(\theta) B(\theta))$.

\item [$H_5$] The frequency vector $\omega$ is $(\gamma,\tau)$-Diophantine
for certain $\gamma>0$ and $\tau\geq n$.
\end{itemize}

Under the above hypotheses, for each $0<\rho_\infty < \rho$ there exists a
constant $\mathfrak{C}_1$ (see Remark~\ref{rem:constants}) such that, if 
the following condition holds
\begin{equation}\label{eq:hyp:small0}
\frac{\mathfrak{C}_1 \snorm{E}_{\rho}}{\gamma^4 \rho^{4\tau}} < 1,
\end{equation}
where $E(\theta)=F(K(\theta))-K(\theta+\omega)$, then there exists an
$F$-invariant torus $K_{\infty} (\TT^n)$ with frequency
$\omega$. The map $ K_{\infty} $ is an embedding, homotopic to the zero
section, analytic in $\TT^n_{\rho_\infty}$, and satisfies
\[
\snorm{\Dif K_{\infty}}_{\rho_\infty} < \sub{\sigma}{\Dif K}, \qquad
\snorm{\Dif K^\top_{\infty}}_{\rho_\infty} < \sub{\sigma}{\Dif K^\top}, \qquad
\mathrm{dist}(K_{\infty}(\TT_{\rho_\infty}^n),\partial \B) >
0.
\]
Furthermore, the map $K_\infty$ is close to $K$: there exists a constant
$\mathfrak{C}_2$ (see Remark~\ref{rem:constants}) such that
\begin{equation}\label{eq:teo:close}
\snorm{K_{\infty}-K}_{\rho_\infty} <
\frac{\mathfrak{C}_2 \snorm{E}_\rho}{\gamma^2 \rho^{2\tau}} .
\end{equation}
\end{theorem}

\begin{remark}
We use the symbol $\sigma$ to denote those constants
that control objects that are corrected iteratively, and so, we have to ensure that the
control prevails along the proof.
\end{remark}

\begin{remark}\label{rem:constants}
Constants $\mathfrak{C}_1$ and $\mathfrak{C}_2$, which are given explicitly
later in Equations~\eqref{eq:Cstar} and~\eqref{eq:Cstar2}, depend explicitly on
the initial data. Concretely, they depend polynomially on $\sub{c}{\Dif F}$,
$\sub{c}{\Dif^2 F}$, $\sub{c}{\Omega}$, $\sub{c}{\Dif \Omega}$, $\sub{c}{\Dif
a}$ and $\sub{c}{\Dif^2 a}$. They also depend polynomially on
$(\sub{\sigma}{\Dif K}-\snorm{\Dif K}_\rho)^{-1}$, $(\sub{\sigma}{\Dif
K^\top}-\|\Dif K^\top\|_\rho)^{-1}$, $(\sub{\sigma}{B}-\snorm{B}_\rho)^{-1}$,
$(\sub{\sigma}{T}-|\avg{T}^{-1}|)^{-1}$ and
$\mathrm{dist}(K(\TT_\rho^n),\partial \B))^{-1}$, and on the strict estimations
$\sub{\sigma}{\Dif K}$, $\sub{\sigma}{\Dif K^\top}$, $\sub{\sigma}{B}$, and
$\sub{\sigma}{T}$, respectively. If we fix $c_R>0$ then $\mathfrak{C}_1$ and
$\mathfrak{C}_2$ depend polynomially on $n$, $c_R$, $\gamma$ and powers of
$\rho$.  These constants can be optimized by selecting a suitable value of
$\rho_\infty$ and adjusting the rate of converge of the iterative scheme.
\end{remark}

\subsection{Proof of the KAM theorem}\label{ssec:proof:KAM}

The proof follows from a standard {\KAM} scheme.
Although a detailed proof of a very similar statement is given in Chapter 4
of~\cite{HaroCFLM}, for the sake of completeness we present here a compact
self-contained 
exposition.
This allows us to describe the main geometric
objects so that the reader can relate them to a corresponding constant that
contributes to the computation of $\mathfrak{C}_1$ and $\mathfrak{C}_2$. 

The argument
consists in
refining $K(\theta)$ by means of a 
Newton method.
At every step, we
add to $K(\theta)$ a correction $\Delta K(\theta)$, given by an approximate
solution of the linearized equation
\begin{equation}\label{eq:newton:linear}
\Dif F(K(\theta)) \Delta K(\theta) - \Delta K(\theta+\omega)=-E(\theta).
\end{equation}
To face this equation we consider a suitable frame on the full tangent space.
The main ingredient is the fact that (under certain assumptions)
an approximately $F$-invariant torus is also approximately Lagrangian.
Hence, the linear dynamics around the torus
is approximately reducible. Specifically, it turns out that we have a behavior
similar to Equation~\eqref{reducibility} but with an error of
order $\snorm{E}_\rho$. This is 
enough to perform a quadratic scheme to correct the initial
approximation.

\begin{lemma}[The Iterative Lemma]\label{prop:newton1step}
Let us consider the same setting and hypotheses of
Theorem~\ref{theo:KAM}.
Then, there exist constants $\hat C_1$, $\hat C_2$, $\hat C_3$, $\hat C_4$, and $\hat C_5$
(depending explicitly on the constants defined in the hypotheses)
such that if
\begin{equation}\label{eq:hyp:small}
\frac{\hat{\mathfrak{C}}_1
\snorm{E}_{\rho}}{\gamma^2 \delta^{2\tau+1}} < 1
\end{equation}
holds for some $0<\delta<\rho$, where
\begin{equation}\label{eq:hyp:prop:exp}
\hat{\mathfrak{C}}_1 := \max \bigg\{
\hat C_1 \gamma \delta^\tau,
\frac{n \hat C_2}{\sub{\sigma}{\Dif K}-\snorm{\Dif K}_{\rho}},
\frac{2 n \hat C_2}{\sub{\sigma}{\Dif K^\top}-\snorm{\Dif K^\top}_{\rho}},
\frac{\hat C_3}{\sub{\sigma}{B} - \snorm{B}_\rho}, 
\frac{\hat C_4}{\sub{\sigma}{T} - |\avg{T}^{-1}|},
\frac{\hat C_2 \delta}{\mathrm{dist} (K(\TT^n_{\rho}),\partial \B)}
\bigg\}, 
\end{equation}
then we have an approximate $F$-invariant torus
of the same frequency $\omega$ given by $\bar K=K+\Delta K$,
that defines new objects $\bar B$ and $\bar T$ (obtained
replacing $K$ by $\bar K$)
satisfying
\begin{align}
&\snorm{\Dif \bar K}_{\rho-3\delta} < \sub{\sigma}{\Dif K},& {}
&\snorm{\Dif \bar K^\top}_{\rho-3 \delta} < \sub{\sigma}{\Dif K^\top},& {}
& \mathrm{dist}(\bar K(\TT_{\rho-2\delta}^n),\partial \B) > 0, 
\label{eq:newob1}  \\
&\snorm{\bar B}_{\rho-3\delta} < \sub{\sigma}{B},& {}
&|\avg{\bar T}^{-1}|< \sub{\sigma}{T},
\label{eq:newob2}
\end{align}
and
\begin{equation}\label{eq:iter1}
\snorm{\bar K-K}_{\rho-2\delta} <
\frac{\hat C_2}{\gamma^2 \delta^{2\tau }} \snorm{E}_\rho,
\quad
\snorm{\bar B -B}_{\rho-3\delta} < \frac{\hat
C_3}{\gamma^{2} \delta^{2\tau +1}} \snorm{E}_\rho,
\quad
|\avg{\bar T}^{-1} - \avg{T}^{-1}| < \frac{\hat
C_4}{\gamma^{2} \delta^{2\tau +1}} \snorm{E}_\rho.
\end{equation}

The new error of invariance is given by
\begin{equation}\label{eq:prop:new:error}
\bar E(\theta)=F(\bar K(\theta)) - \bar K(\theta+\omega), \qquad 
\snorm{\bar E}_{\rho-2\delta} < \frac{\hat C_5}{\gamma^{4}
\delta^{4\tau}} \snorm{E}^2_\rho.
\end{equation}
\end{lemma}

Before proving Lemma~\ref{prop:newton1step}, we present two auxiliary results. 

\begin{lemma}\label{lem:coho:solve1}
Let us consider vector-valued maps $\eta= (\etaL,\etaN) : \TT^n \rightarrow
\RR^{n} \times \RR^{n}$ and a matrix-valued map $T : \TT^n \rightarrow \RR^{n\times n}$.
Assume that $T$ satisfies the non-degeneracy condition $\det
\avg{T(\theta)} \neq 0$, $\forall
\theta \in \TT^n$. Then, the system of equations
\[
\begin{pmatrix}
I_n & T(\theta) \\
O_n & I_n
\end{pmatrix}
\begin{pmatrix}
\xiL(\theta) \\
\xiN(\theta)
\end{pmatrix}
-
\begin{pmatrix}
\xiL(\theta+\omega) \\
\xiN(\theta+\omega)
\end{pmatrix}
=
\begin{pmatrix}
\etaL(\theta) \\
\etaN(\theta)-\avg{\etaN}
\end{pmatrix}
\]
has a (formal) solution 
$\xi= (\xiL,\xiN) : \TT^n \rightarrow
\RR^{n} \times \RR^{n}$
given by
\begin{align}
\xiN(\theta)={} &\R(\etaN(\theta)) +
\xiN_0,\label{eq:xiy} \\
\xiL(\theta)={} &\R(\etaL(\theta) - T(\theta)
\xiN(\theta))+\xiL_0,\label{eq:xix}
\end{align}
for every $\xiL_0 \in \RR^n$, and
\begin{equation}\label{eq:averxiy}
\xiN_0= \avg{T}^{-1} \avg{\etaL-T \R(\etaN)}.
\end{equation}
Note that $\R$ gives the zero-average solution of the one bite cohomological
equation (see Equation~\eqref{eq:standard}).
\end{lemma}

\begin{proof}
The triangular form of this system allows us to face first the equation $\L
\xiN (\theta) = \etaN(\theta) - \avg{\etaN}$, where $\L$ is given
by Equation~\eqref{eq:calL}. 
The right hand side of this equation has already zero
average, so we obtain the solution in~\eqref{eq:xiy}, where 
$\xiN_0 = \avg{\etaN}\in \RR^n$. Then, the upper equation is $\L \etaL
(\theta)=\etaN(\theta) - T(\theta) \xiN(\theta)$ and the vector $\xiN_0$
selected in~\eqref{eq:averxiy} allows us to guarantee that $\avg{\etaL - T
\xiN}=0$. In this way, we obtain the solution in~\eqref{eq:xix}.
\end{proof}

\begin{lemma}\label{lem-quad-avg}
If $K(\theta)$ is an approximately $F$-invariant torus with error
$E(\theta)$, then
\[
\avg{\Dif K(\theta+\omega)^\top
\Omega(K(\theta+\omega)) E(\theta)} = 
\avg{\Dif E(\theta)^\top \Delta a(\theta) +
\Dif K(\theta+\omega)^\top \Delta^2 a(\theta)},
\]
where
\[
\begin{split}
\Delta a(\theta) = {} & a(F(K(\theta))) -
a(K(\theta+\omega)) = \int_0^1 \Dif a(K(\theta+\omega) + tE(\theta))
E(\theta) \dif t, \\
\Delta^2 a(\theta) = {} &  a(F(K(\theta)))-a(K(\theta+\omega)) -
\Dif a(K(\theta+\omega)) E(\theta)\\
= {} & \int_0^1 (1-t) \Dif^2 a(K(\theta+\omega) + tE(\theta))
E(\theta)^{\otimes 2} \dif t.
\end{split}
\]
\end{lemma}

\begin{proof}
From the definition of $\Omega$ in~\eqref{def-Omega}, and some easy computations,
\begin{equation*}
\begin{split}
\label{eq-LOE}
\Dif K (\theta & +\omega)^\top  \Omega(K(\theta+\omega)) E(\theta) \\
= & {} \Dif K(\theta+\omega)^\top \Dif a(K(\theta+\omega))^\top E(\theta) -
\Dif K(\theta+\omega)^\top \Dif a(K(\theta+\omega)) E(\theta) \\
= &{}  \left(\Dif (a(K(\theta+\omega)))\right)^\top E(\theta) + 
\Dif K(\theta+\omega)^\top \left( \Delta^2 a(\theta) - a(F(K(\theta))) + a(K(\theta+\omega))\right) \\ 
= & {}\left(\Dif (a(K(\theta+\omega))^\top E(\theta))\right)^\top - (\Dif E(\theta))^\top a(K(\theta+\omega)) + 
\Dif K(\theta+\omega)^\top \Delta^2 a(\theta) \\ 
   & {} -  \left(\Dif F(K(\theta)) \Dif K(\theta) - \Dif E(\theta)\right)^\top a(F(K(\theta))) +  \Dif K(\theta+\omega)^\top a(K(\theta+\omega))\\
= & {} \left(\Dif (a(K(\theta+\omega))^\top E(\theta))\right)^\top + (\Dif E(\theta))^\top \Delta a(\theta) + 
\Dif K(\theta+\omega)^\top \Delta^2 a(\theta)  \\
   & {}  - (\Dif (S(K(\theta))))^\top - 
\Dif K(\theta)^\top a(K(\theta)) +  \Dif K(\theta+\omega)^\top a(K(\theta+\omega)),
\end{split}
\end{equation*}
where in the last identity we use that $S$ is the primitive function of $F$,
see Equation~\eqref{eq:exact}. The result follows by taking averages and realizing that
$\Dif (a(K(\theta+\omega))^\top E(\theta))$, $\Dif (S(K(\theta)))$  and
$a(K(\theta+\omega))^\top \Dif K(\theta+\omega) - a(K(\theta))^\top \Dif
K(\theta)$ have zero average.
\end{proof}

\begin{proof}[Proof of Lemma~\ref{prop:newton1step}]
In the first part of the proof we see that, since $K(\TT^n)$ is approximately
$F$-invariant, the frame $P(\theta)$ is symplectic up to an error 
controlled by $E(\theta)$.

We start by controlling the objects
$N$, $B$ and $A$, given in Equations~\eqref{eq:def-N}, \eqref{eq:def-B},
and~\eqref{eq:def-A}, respectively. By hypothesis, we have $\snorm{\Dif K}_\rho
< \sub{\sigma}{\Dif K}$ and $\snorm{B}_\rho < \sub{\sigma}{B}$. Then, we obtain
\begin{equation}\label{eq:estima:A}
\snorm{A}_\rho = \snorm{A^\top}_\rho \leq \frac{1}{2}
\snorm{B^\top N_0^\top
(\Omega \comp K)
N_0 B}_\rho \leq \frac{1}{2} n \sub{c}{N_0^\top (\Omega \circ K) N_0} 
(\sub{\sigma}{B})^2 =: \sub{c}{A}.
\end{equation}
where the constant $\sub{c}{A}$ is introduced in order to simplify subsequent
computations. We use small letters ($\sub{c}{A}$, $\sub{c}{N}$, etc.) when the
constant is related to an estimation of a geometric object, using the subscript
to identify the corresponding object. We use capital letters ($C_1$, $C_2$,
etc.) for constants that appear in estimates that depend on the error
$\snorm{E}_\rho$ (divisors are considered separately).

We estimate the norm of $N$ as
\[
\snorm{N}_\rho \leq \snorm{\Dif K}_\rho
\snorm{A}_\rho + \snorm{N_0}_\rho
\snorm{B}_\rho \leq \sub{\sigma}{\Dif K} \sub{c}{A} + \sub{c}{N_0} \sub{\sigma}{B} =: \sub{c}{N}
\]
and
\[
\snorm{N^\top}_\rho \leq \sub{c}{A} \sub{\sigma}{\Dif K^\top} + n \sub{\sigma}{B} \sub{c}{N_0^\top} =: \sub{c}{N^\top}.
\]
The frame $P(\theta)$, given by Equation~\eqref{eq:def-P},
satisfies
\begin{align}
\snorm{P}_\rho \leq {} &
\snorm{\Dif K}_\rho+\snorm{N}_\rho \leq \sub{\sigma}{\Dif K}+\sub{c}{N} =: \sub{c}{P}
\label{eq:norm:cP}
\nonumber
\end{align}
and the torsion $T(\theta)$, given by
Equation~\eqref{def-T}, is controlled by
\[
\snorm{T}_\rho \leq \snorm{N^\top}_\rho
\snorm{\Omega}_{\B} \snorm{\Dif F}_{\B} \snorm{N}_\rho \leq  \sub{c}{N^\top}
\sub{c}{\Omega} \sub{c}{\Dif F} \sub{c}{N} =: \sub{c}{T}.
\]

Now we control the approximate Lagrangian character of $K(\TT^n)$.
Taking derivatives at both sides of $E(\theta)=F(K(\theta))-K(\theta+\omega)$ we have
\begin{equation}\label{eq:newton:der}
\Dif F(K(\theta)) \Dif K(\theta)
=\Dif K(\theta+\omega)+\Dif E(\theta).
\end{equation}
Then, a direct computation of $\L \sub{\Omega}{K} (\theta)$, using Equations~\eqref{eq:calL}
and~\eqref{eq:newton:der}, leads to
\begin{equation}\label{eq:newton:LOmega}
\begin{split}
\L \sub{\Omega}{K} (\theta)
= & 
 \Dif K(\theta+\omega)^{\top}
\Delta\Omega(\theta)\, 
\Dif K(\theta+\omega) + \Dif K(\theta+\omega)^{\top}\Omega(F(K(\theta))) \,\Dif
E(\theta)\\
& + \Dif E(\theta)^{\top}\Omega(F(K(\theta)))
\Dif F(K(\theta))\Dif
               K(\theta)\, , 
\end{split}
\end{equation}
where
\begin{equation}\label{eq:newton:DOmega}
\Delta\Omega(\theta) = \Omega(F(K(\theta)))-
\Omega(K(\theta+\omega))= \int_0^1 \Dif \Omega(K(\theta+\omega)+t
E(\theta)) E(\theta) \dif t\, .
\end{equation}
Using the Mean Value Theorem for integrals and properties of Banach algebras we
obtain $\snorm{\Delta\Omega}_{\rho} \leq \sub{c}{\Dif \Omega} \snorm{E}_\rho$,
and introducing this expression into Equation~\eqref{eq:newton:LOmega} we
control $\snorm{\L \sub{\Omega}{K}}_{\rho-\delta}$ as follows (we use Cauchy
estimates)
\[
\snorm{\L \sub{\Omega}{K}}_{\rho-\delta} \leq 
\bigg(\sub{\sigma}{\Dif K^\top} \sub{\sigma}{\Dif K} \sub{c}{\Dif \Omega} \delta + n \sub{\sigma}{\Dif K^\top} \sub{c}{\Omega} + 2n
\sub{c}{\Omega} \sub{c}{\Dif F} \sub{\sigma}{\Dif K} \bigg) 
\frac{\snorm{E}_\rho}{\delta} =: \frac{C_1}{\delta} \snorm{E}_\rho.
\]
Then, using the R\"ussmann estimates (see Equation~\eqref{eq:russ:first} or Lemma~\ref{lem-Russmann})
we end up with
\[
\snorm{\sub{\Omega}{K}}_{\rho-2\delta} \leq \frac{c_R C_1}{\gamma \delta^{\tau +1}}
\snorm{E}_\rho =:  \frac{C_2}{\gamma \delta^{\tau +1}}
\snorm{E}_\rho.
\]

Next, we introduce the error
in the symplectic character of the frame as follows
\begin{equation}\label{eq:newton:sym}
E_{\mathrm{sym}}(\theta)=P(\theta)^\top \Omega(K(\theta))
P(\theta) - \Omega_0
\end{equation}
and a straightforward computation shows that
\begin{equation}\label{eq:newton:sym2}
E_{\mathrm{sym}}(\theta)=
\begin{pmatrix}
\sub{\Omega}{K}(\theta) & \sub{\Omega}{K}(\theta)A(\theta) \\
A(\theta)^\top \sub{\Omega}{K}(\theta) & A(\theta)^\top
\sub{\Omega}{K}(\theta) A(\theta)
\end{pmatrix},
\end{equation}
which is controlled by
\begin{equation}\label{eq:cotanormEsym}
\snorm{E_{\mathrm{sym}}}_{\rho-2\delta} \leq \frac{(1+\sub{c}{A}) \max
\{1,\sub{c}{A}\}C_2 }{\gamma \delta^{\tau+1}} \snorm{E}_\rho =: \frac{C_3}{\gamma
\delta^{\tau+1}} \snorm{E}_\rho.
\end{equation}

Next, we show that the tangent map $\Dif F$ is approximately reducible
in the frame $P(\theta)$. To this end, we introduce
\begin{equation}\label{eq:newton:red}
E_{\mathrm{red}}(\theta)
=
-\Omega_0 P(\theta+\omega)^{\top}\Omega(K(\theta+\omega))\Dif
F(K(\theta))
 P(\theta)\, 
 -  \Lambda(\theta),
\end{equation}
where $\Lambda(\theta)$ is given by Equation~\eqref{reducibility}. 
We decompose $E_{\mathrm{red}}(\theta)$ into
four $(n\times n)$-block
components given by:
\begin{align}
E_{\mathrm{red}}^{1,1} (\theta) 
 = {} &  N(\theta+\omega)^{\top}\Omega(K(\theta+\omega))\Dif E(\theta)
+ A(\theta+\omega)^\top \sub{\Omega}{K}(\theta+\omega)\, , 
\label{eq:Lambda1} 
\\ 
E_{\mathrm{red}}^{1,2}(\theta)  = {} &
N(\theta+\omega)^{\top}\Omega(K(\theta+\omega))\Dif F(K(\theta))
N(\theta) - T(\theta) = O_n
\nonumber \\
E_{\mathrm{red}}^{2,1} (\theta) 
= {} & - \sub{\Omega}{K}(\theta+\omega) - \Dif K
(\theta+\omega)^\top \Omega(K(\theta+\omega)) \Dif E(\theta)\, , 
\label{eq:Lambda3} 
\\
E_{\mathrm{red}}^{2,2} (\theta) 
= {} & - \sub{\Omega}{K}(\theta) A(\theta) +
\Dif K (\theta+\omega)^\top \Delta \Omega(\theta) \Dif F(K(\theta))
N(\theta) \nonumber \\
& + \Dif E(\theta)^\top \Omega(F(K(\theta))) \Dif F(K(\theta)) N(\theta).
\label{eq:Lambda4} 
\end{align}
Then, we conclude that the error of reducibility satisfies
\[
\snorm{E_{\mathrm{red}}}_{\rho-2\delta} \leq \frac{\max \{C_4, C_5+
C_6\}}{\gamma \delta^{\tau +1}} \snorm{E}_\rho =:
\frac{C_7}{\gamma \delta^{\tau+1}} \snorm{E}_\rho
\]
where
\begin{align}
C_4 = {} & n \sub{c}{N^\top} \sub{c}{\Omega} \gamma \delta^\tau + \sub{c}{A} C_2, \nonumber \\
C_5 = {} & C_2 + n \sub{\sigma}{\Dif K^\top} \sub{c}{\Omega} \gamma
\delta^\tau, \nonumber \\
C_6 = {} & \sub{c}{A} C_2 + \sub{\sigma}{\Dif K^\top} \sub{c}{\Dif \Omega} \sub{c}{\Dif F} \sub{c}{N} \gamma
\delta^{\tau +1} + 2n \sub{c}{\Omega} \sub{c}{\Dif F} \sub{c}{N} \gamma
\delta^\tau. \label{eq:cons:C6}
\end{align}

Now, we study Equation~\eqref{eq:newton:linear}
using the symplectic frame in~\eqref{eq:def-P}:
we introduce
$\Delta K(\theta)=P(\theta) \xi(\theta)$ thus 
obtaining
\[
\Dif F(K(\theta)) P(\theta) \xi(\theta) -
P(\theta+\omega) \xi(\theta+\omega) = - E(\theta).
\]
We multiply
both sides by $-\Omega_0 P(\theta+\omega)^\top
\Omega(K(\theta+\omega))$ and we get
\begin{equation}\label{eq:lin:fram}
\begin{split}
\Lambda(\theta)
\xi(\theta)
+ E_{\mathrm{red}}(\theta) & \xi(\theta) -(I-\Omega_0
E_{\mathrm{sym}}(\theta+\omega)) \xi(\theta+\omega) = \\
 & \Omega_0
P(\theta+\omega)^\top \Omega(K(\theta+\omega)) E(\theta),
\end{split}
\end{equation}
where we used Equations~\eqref{eq:newton:sym}
and~\eqref{eq:newton:red}. In order to obtain an approximate solution
of Equation~\eqref{eq:lin:fram}, we 
consider Lemma~\ref{lem:coho:solve1} taking
\begin{equation}\label{eq:case:eta}
\eta(\theta)= 
\Omega_0
P(\theta+\omega)^\top \Omega(K(\theta+\omega)) E(\theta),
\end{equation}
and $T(\theta)$ given by Equation~\eqref{def-T}. 
We choose the solution satisfying $\xiL_0=0$. 
To control
the resulting Equations~\eqref{eq:xiy},~\eqref{eq:xix},
and~\eqref{eq:averxiy}, we first compute
\begin{align}
\snorm{\etaL}_\rho = {} & \snorm{N(\theta+\omega)^\top
\Omega(K(\theta+\omega)) E(\theta)}_\rho  \leq  \sub{c}{N^\top} \sub{c}{\Omega}
\snorm{E}_\rho, \label{eq:etaLnorm}\\
\snorm{\etaN}_\rho = {} & \snorm{\Dif K(\theta+\omega)^\top
\Omega(K(\theta+\omega)) E(\theta)}_\rho  \leq \sub{\sigma}{\Dif K^\top} \sub{c}{\Omega}
\snorm{E}_\rho. \label{eq:etaNnorm}
\end{align}
On the one hand, using Equations~\eqref{eq:russ:first} and~\eqref{eq:etaNnorm}, we obtain
\[
\snorm{\R (\etaN)}_{\rho-\delta} \leq \frac{c_R \sub{\sigma}{\Dif K^\top}
\sub{c}{\Omega}}{\gamma \delta^\tau} \snorm{E}_\rho =:
\frac{C_8}{\gamma \delta^\tau} \snorm{E}_\rho,
\]
and on the other hand, using 
Hypothesis $H_4$
and Equations~\eqref{eq:averxiy},~\eqref{eq:russ:first}
and~\eqref{eq:etaLnorm}, we have
\begin{equation}\label{eq:norm:xi}
\begin{split}
\snorm{\xiN}_{\rho- \delta} \leq {} & \frac{C_8 + \sub{\sigma}{T} (\sub{c}{N^\top} \sub{c}{\Omega} \gamma \delta^\tau +
\sub{c}{T} C_8)}{\gamma \delta^{\tau}} \snorm{E}_\rho =: \frac{C_{9}}{\gamma \delta^{\tau}} \snorm{E}_\rho
, \\
\snorm{\xiL}_{\rho-2\delta} \leq {} & \frac{c_R(\sub{c}{N^\top} \sub{c}{\Omega} \gamma \delta^\tau + \sub{c}{T} C_9)}{\gamma^2 \delta^{2\tau}} \snorm{E}_\rho =: \frac{C_{10}}{\gamma^2 \delta^{2\tau}} \snorm{E}_\rho
. 
\end{split}
\end{equation}

The new parameterization $\bar K = K + \Delta K$ and the related objects are controlled
using standard computations.  The first estimate
in~\eqref{eq:iter1} follows directly from $\Delta K= \Dif K \xiL + N \xiN$ and
estimates in~\eqref{eq:norm:xi}:
\[
\snorm{\bar K-K}_{\rho-2\delta} =
\snorm{\Delta K}_{\rho-2\delta} \leq \frac{\sub{\sigma}{\Dif K} C_{10} +
\sub{c}{N} C_9 \gamma \delta^\tau}{\gamma^2 \delta^{2\tau}} \snorm{E}_\rho =: \frac{\hat
C_2}{\gamma^2 \delta^{2\tau}} \snorm{E}_\rho.
\]
Combining this
expression with Cauchy estimates we obtain the first estimate
in~\eqref{eq:newob1}:
\begin{equation}\label{eq:prop:cond2}
\snorm{\Dif \bar K}_{\rho-3 \delta} \leq {} \snorm{\Dif K}_\rho +
\snorm{\Dif \Delta K}_{\rho-3\delta} 
\leq \snorm{\Dif K}_\rho + \frac{n \hat C_2}{\gamma^2
\delta^{2\tau+1}} \snorm{E}_\rho < \sub{\sigma}{\Dif K}.
\end{equation}
The last inequality in the previous computation is obtained by
including this condition
in Hypothesis~\eqref{eq:hyp:small}.
The control of the transposed object in~\eqref{eq:newob2} is analogous:
\begin{equation}\label{eq:prop:cond3}
\snorm{\Dif \bar K^\top}_{\rho-3 \delta} \leq {} 
\snorm{\Dif K^\top}_\rho + \frac{2 n \hat C_2}{\gamma^2
\delta^{2\tau+1}} \snorm{E}_\rho < \sub{\sigma}{\Dif K^\top}.
\end{equation}

To control $\bar B$ and $\avg{\bar T}^{-1}$ we use that for every pair of
matrices $X$ and $Y$
\begin{equation}\label{eq:mat:inv}
Y^{-1} = (I+X^{-1}(Y-X))^{-1} X^{-1}.
\end{equation}
If $\|X^{-1}\|\|Y^{-1}\|< 1$, the Neumann series implies
\begin{equation}\label{eq:mat:inv2}
\snorm{Y^{-1}-X^{-1}} \leq \frac{\snorm{X^{-1}}^2
\snorm{Y-X}}{1-\snorm{X^{-1}} \snorm{Y-X}}.
\end{equation}
First, we use
Equation~\eqref{eq:mat:inv} 
taking
$X=\Dif K^\top \Omega(K) N_0$ and $Y=\Dif \bar K^\top
\Omega(\bar K) N_0$. We obtain the second estimate in~\eqref{eq:iter1}
with
\[
\hat C_3 := 2 \sub{\sigma}{B}^2 C_{11}, \quad
C_{11}:= \sub{c}{N_0} \hat C_2(\sub{\sigma}{\Dif K^\top} \sub{c}{\Dif \Omega} \delta + 2 n
\sub{c}{\Omega}),
\]
where we assumed that (to be included
in~\eqref{eq:hyp:small})
\begin{equation}\label{eq:prop:cond4}
\frac{ 2 \sub{\sigma}{B} C_{11}}{\gamma^2 \delta^{2\tau +1}}
\snorm{E}_\rho < 1.
\end{equation}
This computation allows us to set that in order to
satisfy Equation~\eqref{eq:newob1} we have to include
\begin{equation}\label{eq:prop:cond5}
\snorm{\bar B}_{\rho-3\delta} \leq \snorm{B}_{\rho-3\delta} +
\snorm{\bar B - B}_{\rho-3 \delta}  \leq \snorm{B}_{\rho-3\delta} + \frac{
\hat C_3}{\gamma^2 \delta^{2\tau +1}}
\snorm{E}_\rho < \sub{\sigma}{B},
\end{equation}
into Condition~\eqref{eq:hyp:small}. 

The third expression in~\eqref{eq:iter1} also follows
using Equation~\eqref{eq:mat:inv} with $X=T$ and $Y=\bar T$.
Now we have to control
the new matrices $\bar N(\theta)$ and $\bar A(\theta)$, given by Equations~\eqref{eq:def-N} and~\eqref{eq:def-A} replacing
$K(\theta)$ by $\bar K(\theta)$.
Specifically, we obtain
\[
\snorm{\bar A-A}_{\rho-3\delta} \leq 
\left(\tfrac{n}{2} (\sub{\sigma}{B})^2 \sub{c}{\Dif \Omega} \hat C_2 \delta 
+ \tfrac{n+1}{2} \sub{c}{N_0^\top (\Omega \circ K) N_0} \hat C_3\right)
\frac{\snorm{E}_{\rho}}{\gamma^2 \delta^{2\tau +1}}
=: \frac{C_{12}}{\gamma^2 \delta^{2\tau +1}} {\snorm{E}_{\rho}}, 
\]
and observe that $\snorm{\bar A-A}_{\rho- 3\delta} = \snorm{\bar A^\top -A^\top }_{\rho- 3\delta}$. Moreover
\[
\snorm{\bar N-N}_{\rho-3\delta} \leq \frac{(\sub{\sigma}{\Dif K}
C_{12} + n \hat C_2 \sub{c}{A} + \sub{c}{N_0} \hat C_3) \snorm{E}_{\rho}}{\gamma^2
\delta^{2\tau +1}} =: \frac{C_{13}}{\gamma^2
\delta^{2\tau +1}} \snorm{E}_{\rho}
\]
and
\[
\snorm{\bar N^\top-N^\top}_{\rho-3\delta} \leq 
\frac{(\sub{\sigma}{\Dif K^\top}
C_{12} + 2 n \hat C_2 \sub{c}{A} + n\sub{c}{N_0^\top} \hat C_3) \snorm{E}_{\rho}}{\gamma^2
\delta^{2\tau +1}} =: \frac{C_{13}^*}{\gamma^2
\delta^{2\tau +1}} \snorm{E}_{\rho},
\]
that allow us to compute
\begin{equation}\label{eq:newtwist}
\snorm{\bar T-T}_{\rho-3\delta} \leq
\frac{C_{14}}{\gamma^2 \delta^{2\tau+1}} \snorm{E}_{\rho}, 
\end{equation}
with
\[
C_{14}:= \sub{c}{N^\top} \sub{c}{N} \hat C_2 (\sub{c}{\Omega} \sub{c}{\Dif^2 F} + \sub{c}{\Dif \Omega}
\sub{c}{\Dif F}) \delta + \sub{c}{\Omega} \sub{c}{\Dif F} (\sub{c}{N^\top} C_{13}+\sub{c}{N} C_{13}^*).
\]
Introducing Equation~\eqref{eq:newtwist} into Equation~\eqref{eq:mat:inv2} we
obtain the third estimate in~\eqref{eq:iter1} by defining the constant $\hat
C_4 := 2 \sub{\sigma}{T}^2 C_{14}$ and also the third estimate
in~\eqref{eq:newob2}.  Computations are analogous to those performed to control
the object $\bar B$.  Hence we have to include the condition
\begin{equation}\label{eq:prop:cond7}
\frac{\hat C_4}{\gamma^2 \delta^{2\tau +1}}
\snorm{E}_\rho <
\sub{\sigma}{T}-|\avg{T}^{-1}|
\end{equation}
in~\eqref{eq:hyp:small}. Note that the closure of $\bar
K(\TT^n_{\rho-2\delta})$ lies in $\B$, since
\begin{equation}\label{eq:est:dist}
\mathrm{dist} (\bar K(\TT^n_{\rho-2\delta}),\partial \B) \geq 
\mathrm{dist} (K(\TT^n_{\rho}),\partial \B) -
\snorm{\Delta K}_{\rho-2\delta}  \geq
\mathrm{dist}
(K(\TT^n_{\rho}),\partial \B) - \frac{\hat C_2}{\gamma^2
\delta^{2\tau}} \snorm{E}_\rho > 0. 
\end{equation}
The last inequality is also included in~\eqref{eq:hyp:small}.

Hence, the terms $E_{\mathrm{red}}(\theta) \xi(\theta)$
and $\Omega_0 E_{\mathrm{sym}}(\theta+\omega)\xi(\theta+\omega)$ in
Equation~\eqref{eq:lin:fram} are quadratic in $E(\theta)$. Then, 
using $\Delta K(\theta)=P(\theta)\xi(\theta)$,
Equation~\eqref{eq:lin:fram}, the definition of $\xi(\theta)$, and
also that
\[
(-\Omega_0 P(\theta+\omega)^\top
\Omega(K(\theta+\omega)))^{-1}=P(\theta+\omega)(I-\Omega_0
E_{\mathrm{sym}}(\theta+\omega))^{-1},
\]
it turns out that
\begin{equation}\label{eq:newton:lin:error}
\begin{split}
\Dif F(K(\theta)) \Delta K(\theta) & -
\Delta K (\theta+\omega) + E(\theta) = \\
&P(\theta+\omega)(I-\Omega_0
E_{\mathrm{sym}}(\theta+\omega))^{-1} E_{\mathrm{lin}}(\theta),
\end{split}
\end{equation}
where
\begin{equation}\label{eq:def-Elin}
E_{\mathrm{lin}}(\theta) =
E_{\mathrm{red}}(\theta)
\xi(\theta) + \Omega_0 E_{\mathrm{sym}}(\theta+\omega)
\xi(\theta+\omega) - 
\begin{pmatrix}
0\\
\avg{L(\theta+\omega)^\top
\Omega(K(\theta+\omega)) E(\theta)}
\end{pmatrix}.
\end{equation}
After performing one step of the
Newton method,
the error of invariance
associated to the parameterization
$\bar K = K+\Delta K$ is given by
\begin{equation}\label{eq:def-error2}
\begin{split}
\bar E (\theta)
= {} &
F(K(\theta)+\Delta K(\theta))-K(\theta) -
\Delta K(\theta+\omega) \\
= {} & 
P(\theta+\omega)(I-\Omega_0
E_{\mathrm{sym}}(\theta+\omega))^{-1} E_{\mathrm{lin}}(\theta)
+ \Delta^2 F(\theta),
\end{split}
\end{equation}
where we used Equation~\eqref{eq:newton:lin:error}, and
\[
\begin{split}
\Delta^2 F(\theta)= {} &
F(K(\theta)+\Delta K(\theta)) - F(K(\theta)) - \Dif F(K(\theta))
\Delta K(\theta) \\
= {} & \int_0^1 (1-t) \Dif^2 F(K(\theta) + t \Delta K(\theta))
\Delta K(\theta)^{\otimes 2} \dif t.
\end{split}
\]
The last step of the proof is to see, using the previously computed
expressions, that the new error 
$\bar E (\theta)$
is quadratic in $E(\theta)$.

We use Lemma~\ref{lem-quad-avg} to control the
modulus of the average:
\[
\left|\avg{L(\theta+\omega)^\top \Omega(K(\theta+\omega))
E(\theta)} \right| \leq \bigg(
\frac{2n \sub{c}{\Dif a}}{\delta} + \frac{\sub{c}{\Dif^2 a}}{2}
\bigg) \snorm{E}_\rho^2
\]
and from the expression of $E_{\mathrm{lin}}(\theta)$ in Equation~\eqref{eq:def-Elin} we obtain
\[
\snorm{E_{\mathrm{lin}}}_{\rho-2\delta} \leq \bigg( \frac{(C_3 + C_7)
 \max\{C_9 \gamma \delta^\tau,C_{10}\}}{\gamma^3 \delta^{3 \tau +1}} + \frac{2n \sub{c}{\Dif a}}{\delta} +
\frac{\sub{c}{\Dif^2 a}}{2}
\bigg) \snorm{E}_\rho^2=:
\frac{C_{15}}{\gamma^3 \delta^{3\tau +1}} \snorm{E}_\rho^2.
\]
Using
a Neumann series argument we obtain
\begin{equation}\label{eq:norm:Neumann}
\snorm{(I-\Omega_0 E_{\mathrm{sym}})^{-1}} \leq 
\frac{1}{1-\snorm{\Omega_0 E_{\mathrm{sym}}}}.
\end{equation}
Let us consider, as a hypothesis that we include
in~\eqref{eq:hyp:small}, that
\begin{equation}\label{eq:prop:cond1}
\frac{2C_3}{\gamma \delta^{\tau+1}} \snorm{E}_\rho 
=:\frac{\hat C_1}{\gamma \delta^{\tau+1}} \snorm{E}_\rho 
< 1.
\end{equation}
Using Equations~\eqref{eq:cotanormEsym},~\eqref{eq:norm:Neumann} and~\eqref{eq:prop:cond1},
we obtain $\snorm{(I - \Omega_0 E_{\mathrm{sym}})^{-1}} < 2$.
Then, the new
error of invariance, given by Equation~\eqref{eq:def-error2},
satisfies Condition~\eqref{eq:prop:new:error}:
\begin{equation}\label{eq:prop:e2}
\snorm{\bar E}_{\rho-2\delta} <
\bigg( 2 \sub{c}{P} C_{15} \gamma \delta^{\tau-1} + \frac{1}{2} \sub{c}{\Dif^2 F}
\hat C_2^2\bigg)
\frac{\snorm{E}_\rho^2}
{\gamma^4 \delta^{4\tau}}
=:
\frac{\hat C_{5}
\snorm{E}_\rho^2}
{\gamma^4 \delta^{4\tau}}.
\end{equation}

{We complete} the proof by merging
Equations~\eqref{eq:prop:cond2}, \eqref{eq:prop:cond3}, \eqref{eq:prop:cond5},
\eqref{eq:prop:cond7}, \eqref{eq:est:dist} and \eqref{eq:prop:cond1}, thus
obtaining 
the expression in~\eqref{eq:hyp:prop:exp}
that appears in the statement.
\end{proof}

\begin{proof}[Proof of Theorem~\ref{theo:KAM}]
Let us consider 
the approximate $F$-invariant torus 
$K_{0}:=K$ 
with initial error
$E_{0}:= E$. 
We also introduce $B_{0}:=B$ and $T_{0}:=T$
associated with the initial approximation.
By applying Lemma~\ref{prop:newton1step} recursively
we obtain new objects $K_{s}=K_{s-1}$, $E_{s}=E_{s-1}$, $B_{s}=B_{s-1}$, and
$T_{s}=T_{s-1}$. The domain
of analyticity of these objects is reduced at every step.
To characterize this fact, 
we introduce parameters $a_1>1$,
$a_2>1$, $a_3=3 \frac{a_1}{a_1-1} \frac{a_2}{a_2-1}$ and define
\[
\rho_{0} = \rho, \qquad 
\delta_{0} = \frac{\rho_{0}}{a_3}, \qquad
\rho_{s} = \rho_{s-1} - 3 \delta_{s-1}, \qquad
\delta_{s} = \frac{\delta_{0}}{a_1^s}, \qquad
\rho_{\infty} = \lim_{s \rightarrow \infty} \rho_{s} = 
\frac{\rho_{0}}{a_2}.
\]
We can select the above parameters to optimize the convergence of the KAM
process for a particular problem.  This has been used for example
in~\cite{LlaveR91}.  Due to the quadratic convergence of the scheme, a good
strategy is to optimize the first numbers $\delta_0$, $\delta_1$, \ldots,
$\delta_m$.

We denote the objects at the $s$-step as $K_{s}$, $E_{s}$, $B_{s}$ and $T_{s}$,
respec\-tively. We observe that Condition~\eqref{eq:hyp:small} is required at
every step but the construction has been performed in such a way that we can
control $\snorm{\Dif K_{s}}_{\rho_{s}}$, $\snorm{\Dif K_{s}^\top}_{\rho_{s}}$,
$\snorm{B_{s}}_{\rho_{s}}$, $\mathrm{dist} (K_{s}(\TT^n_{\rho_{s}}),\partial
\B)$, and $|\avg{T_{s}}^{-1}|$ uniformly with respect to $s$, so the constants
that appear in Lemma~\ref{prop:newton1step} are taken to be the same for all
steps by considering the worst value of $\delta_s$, that is,
$\delta_0=\rho_0/a_3$.

Now we proceed by induction. We suppose that we have applied $s$ times
Lemma~\ref{prop:newton1step}, for certain $s\geq 0$, so we have to verify that
we can apply it again. To this end, we first compute the error $E_{s}$ in terms
of $E_{0}$ as follows
\[
\snorm{E_{s}}_{\rho_{s}} < \frac{\hat C_5}{\gamma^4
\delta_{s-1}^{4\tau}} \snorm{E_{s-1}}_{\rho_{s-1}}^2 =
\frac{\hat C_5 a_1^{4\tau (s-1)}}{\gamma^4
\delta_{0}^{4\tau}} \snorm{E_{s-1}}_{\rho_{s-1}}^2
\]
and iterating this sequence backwards (we use that
$1+2+\ldots+2^{s-1}= 2^s-1$
and
$1(s-1)+2(s-2)+2^2(s-3)\ldots+2^{s-2}1=2^s-s-1$) we obtain
\begin{equation}\label{eq:errors}
\snorm{E_{s}}_{\rho_{s}} < 
\bigg(
\frac{a_1^{4\tau} \hat C_5 \snorm{E_{0}}_{\rho_{0}}}{\gamma^4 \delta^{4\tau}_{0}}
\bigg)^{2^s-1} a_1^{-4\tau s} \snorm{E_{0}}_{\rho_{0}}.
\end{equation}

We use this expression in order to verify
Condition~\eqref{eq:hyp:small} so we can perform the step $s+1$.
Before that, in order to 
produce a decreasing sequence of errors, we
assume that
\begin{equation}\label{eq:teo:cond1}
\frac{a_1^{4\tau} \hat C_5 \snorm{E_{0}}_{\rho_{0}}}{\gamma^4
\delta^{4\tau}_{0}} < 1
\end{equation}
thus including this condition in 
Hypothesis~\eqref{eq:hyp:small0}.
Now, to verify the inequality in~\eqref{eq:hyp:small} we
observe that in the expression for $\hat{\mathfrak{C}}_1$, given by Equation~\eqref{eq:hyp:prop:exp},
there are two types of
conditions. On the one hand, we have conditions
like~\eqref{eq:prop:cond1},
depending only on the error $E_{s}$ and $\delta_{s}$. On the other
hand, we have conditions like~\eqref{eq:prop:cond2} depending also on other objects
at the $s$-step. For example, Condition~\eqref{eq:prop:cond1} is
direct using Equation~\eqref{eq:errors} and $\tau\geq n$
\[
\frac{2 C_3 \snorm{E_{s}}_{\rho_{s}}}{\gamma
\delta_{s}^{\tau+1}}
< 
\frac{2C_3 a_1^{(\tau+1)s}}{\gamma \delta_{0}^{\tau+1}}
\bigg(
\frac{a_1^{4\tau} \hat C_5 \snorm{E_{0}}_{\rho_{0}}}{\gamma^4 \delta^{4\tau}_{0}}
\bigg)^{2^s-1} a_1^{-4\tau s} \snorm{E_{0}}_{\rho_{0}} 
< \frac{2C_3}{\gamma \delta_{0}^{\tau+1}}
\snorm{E_{0}}_{\rho_{0}} < 1,
\]
where the last inequality is included in~\eqref{eq:hyp:small0}. If the
condition depends also on other objects, we have to relate it to the
initial one. For example, Condition~\eqref{eq:prop:cond2} requires to
compute, using Equation~\eqref{eq:teo:cond1}, the following
{\allowdisplaybreaks
\begin{align*}
\snorm{\Dif K_{s}}_{\rho_{s}} & + \frac{n \hat C_2
\snorm{E_{s}}_{\rho_{s}}}{\gamma^2
\delta_{s}^{2 \tau+1}}
<
\snorm{\Dif K_{0}}_{\rho_{0}} + 
\sum_{j=0}^{s} 
\frac{n \hat C_2
\snorm{E_{j}}_{\rho_{j}}}{\gamma^2
\delta_{j}^{2 \tau+1}} \\
< {} & \snorm{\Dif K_{0}}_{\rho_{0}} + 
\sum_{j=0}^{\infty} \frac{n \hat C_2 a_1^{(2\tau+1)j}}{\gamma^2
\delta_{0}^{2\tau+1}}
\bigg(
\frac{a_1^{4\tau} \hat C_5 \snorm{E_{0}}_{\rho_{0}}}{\gamma^4
\delta^{4\tau}_{0}}
\bigg)^{2^j-1} a_1^{-4\tau j} \snorm{E_{0}}_{\rho_{0}} \\
< {} & \snorm{\Dif K_{0}}_{\rho_{0}} + \frac{n \hat C_2}{\gamma^2
\delta_{0}^{2\tau+1}} \bigg(\frac{1}{1-a_1^{1-2\tau}}\bigg)
\snorm{E_{0}}_{\rho_{0}} < \sub{\sigma}{\Dif K}. 
\end{align*}}
As usual, the last inequality is included in~\eqref{eq:hyp:small0}.
Then, we perform analogous computations to guarantee the conditions
in~\eqref{eq:hyp:small} and we obtain 
the sufficient condition
\begin{equation}\label{eq:teo:cond2}
\frac{\hat C_8 \snorm{E_0}_{\rho}}{\gamma^2 \delta_0^{2\tau+1}} < 1,
\end{equation}
where $\hat C_8$ is given by
\begin{equation}
\label{eq:hyp:prop:exp2}
\hat C_8 := \max \bigg\{
2C_3 \gamma \delta_{0}^\tau,
\frac{\hat C_6}{1-a_1^{1-2\tau}},
\frac{\hat C_7}{1-a_1^{-2\tau}}
\bigg\}
\end{equation}
with
\begin{equation}\label{eq:hyp:prop:hat6}
\hat C_6 := \max \bigg\{
\frac{n \hat C_2}{\sub{\sigma}{\Dif K}-\snorm{\Dif K_{0}}_{\rho_{0}}},
\frac{2 n \hat C_2}{\sub{\sigma}{\Dif K^\top}-\snorm{\Dif K_{0}^\top}_{\rho_{0}}},
\frac{\hat C_3}{\sub{\sigma}{B} - \snorm{B_{0}}_{\rho_{0}}}, 
\frac{\hat C_4}{\sub{\sigma}{T} - \left|\avg{ T_{0}}^{-1}\right|}
\bigg\}
\end{equation}
and
\[
\hat C_7:=
\frac{\hat C_2 \delta_{0}}{\mathrm{dist} (K_{0}(\TT^n_{\rho_{0}}),\partial \B)}.
\]

Since
Hypotheses
$H_1$ to $H_4$ and 
Condition~\eqref{eq:hyp:small}
are satisfied, we can apply
Lemma~\ref{prop:newton1step} again.
Note that the sequence of
errors satisfies $\snorm{E_{s}}_{\rho_{s}} \rightarrow 0$ when
$s\rightarrow \infty$, so the iterative scheme converges to a true quasi-periodic torus
$K_{\infty}$.
Condition~\eqref{eq:hyp:small0} of the smallness of $\snorm{E_{0}}_{\rho_{0}}$
is obtained by merging Conditions~\eqref{eq:teo:cond1}
and~\eqref{eq:teo:cond2}. Indeed, we have
\begin{equation}\label{eq:Cstar}
\mathfrak{C}_1 := \max \bigg\{(a_1a_3)^{4\tau} \hat C_5,(a_3)^{2\tau+1} \hat
C_8 \gamma^2 \rho_{0}^{2\tau-1}\bigg\},
\end{equation}
where
$\hat C_5$ is given in~\eqref{eq:prop:new:error},
$\hat C_8$ is given in~\eqref{eq:hyp:prop:exp2}
and we used that $\delta_0=\rho_0/a_3$.
Finally, we obtain the constant
\begin{equation}\label{eq:Cstar2}
\mathfrak{C}_2:= a_3^{2 \tau} \hat C_2/(1-a_1^{1-2\tau})
\end{equation}
that appears in~\eqref{eq:teo:close}, controlling that
the torus is close to the initial approximation.
\end{proof}

\section{On the approximation of periodic functions 
using discrete Fourier transform}\label{sec:DFT}

In the core of the computer assisted methodology presented in this work, we
have to bound the error produced when approximating a periodic function by its
discrete Fourier transform.  This is a very natural problem that has been
considered in the approximation theory literature \cite{Rivlin69}. It is well
known that error estimates improve commensurately as the functions become
smoother \cite{Epstein04}. 
We refer the reader to~\cite{Minton13,SchenkelWW00} for problems
where similar ideas have been used.
Motivated by the setting of the present paper, we
address the problem for analytic functions. The estimates presented in this
section improve the ones given in \cite{Epstein04} for this specific case (see
Section~\ref{ssec:Four:1D}). 

\subsection{Notation regarding discretization of the torus and Fourier
transforms}\label{ssec:notation:dft}

Given a function $f: \TT^n \rightarrow \CC$,
we consider its Fourier series
\[
f(\theta)=\sum_{k \in \ZZ^n} f_k \ee^{2 \pi \mathrm{i} k \cdot \theta},
\]
where the Fourier coefficients are given by the \emph{Fourier transform} (${\rm
FT}$)
\begin{equation}\label{eq:fourier:coef}
f_k = \int_{[0,1]^n} f(\theta) \ee^{-2 \pi \mathrm{i} k \cdot \theta} \dif
\theta.
\end{equation}

We consider a sample of points on the regular grid of size
$\NF=(\Ni{1},\ldots,\Ni{n}) \in \NN^n$
\begin{equation}\label{eq:sample:torus}
\theta_j:=(\theta_{j_1},\ldots,\theta_{j_n})=
\left(\frac{j_1}{\Ni{1}},\ldots,
\frac{j_n}{\Ni{n}}\right),
\end{equation}
where $j= (j_1,\ldots,j_n)$, with $0\leq j_\ell < \Ni{\ell}$ and $1\leq \ell
\leq n$. This defines an $n$-dimensional sampling $\{f_j\}$, with
$f_j=f(\theta_j)$. The total number of points is  $\Ntot = \Ni{1} \cdots
\Ni{n}$.  The integrals in Equation~\eqref{eq:fourier:coef} are approximated
using the trapezoidal rule on the regular grid, obtaining the discrete Fourier
transform (DFT)
\[
\tilde f_k= \frac{1}{\Ntot} \sum_{0\leq j < \NF} f_j \mathrm{e}^{-2\pi
\mathrm{i} k \cdot \theta_j},
\] 
where the sum runs over integer subindices $j \in \ZZ^n$ such that $0\leq
j_\ell < \Ni{\ell}$ for $\ell= 1,\dots, n$.  Notice that $\tilde f_k$ is
periodic with respect to the components $k_1,\dots, k_n$ of $k$, with
periods $\Ni{1}, \dots, \Ni{n}$, respectively. 
The periodic function $f$ is approximated by the discrete Fourier approximation
\begin{equation}\label{eq:four:approx}
\tilde f(\theta)= \sum_{k \in \INF} \tilde f_k \ee^{2 \pi \mathrm{i} k \cdot
\theta},
\end{equation}
where $\INF$ is the finite set of multi-indices given by 
\begin{equation}\label{eq:INF}
 \INF= \bigg\{ k \in \ZZ^n \,|\, -\frac{\Ni{\ell}}{2} \leq k_\ell <
 \frac{\Ni{\ell}}{2}, 1\leq \ell \leq n \bigg\}.
\end{equation}

%

Along this section we will use the standard notation $[x]$ for the integer part
of $x$: $[x]=\min\left\{j\in\mathbb Z: x\leq j\right\}$.

\subsection{Error estimates on the approximation of analytic periodic functions}

Motivated by the setting of the present paper, we will work in spaces of
analytic functions on a complex strip of the torus (see
Section~\ref{sec-anal-prelims}), but most of the arguments
can be adapted to other spaces. The main goal is to control
the error between $\tilde f$ and $f$, using suitable norms.  As a
previous step, we establish estimates of the approximation $\tilde f_k$ of
$f_k$. 

\begin{lemma}\label{eq:lemaDTF}
The coefficients of the DFT are obtained from the coefficients of the FT
by
\[
\tilde f_k= \sum_{m\in \ZZ^n} f_{k+\NF(m)},
\]
where $\NF(m)= (\Ni{1}m_1, \dots, \Ni{n} m_n).$
\end{lemma}

The proof of Lemma~\ref{eq:lemaDTF} is direct. Using this result we obtain a
bound for the difference between $\tilde f_k$ and $f_k$ as follows:

\begin{proposition}
\label{DFTcoefs} Let $f:\TT^{n}_{\hrho} \to \CC$ be an analytic and bounded
function in the complex strip $\TT^{n}_{\hrho}$ of size $\hrho>0$. Let $\tilde
f$ be the discrete Fourier approximation of $f$ in the regular grid of size
$\NF= (\Ni{1},\dots, \Ni{n}) \in \NN^n$.  Then, for $-\frac{\NF}{2}\leq k <
\frac{\NF}{2}$: 
\[
|\tilde f_k - f_k| \leq s_{\NF}^*(k,\hrho) \snorm{f}_{\hrho}
\]
where 
\[
   s_{\NF}^*(k,\hrho) =  \prod_{\ell= 1}^n \left( \ee^{-\pi\hrho \Ni{\ell}}
   \frac{\ee^{2\pi\hrho(|k_\ell|- \Ni{\ell}/2)}+ \ee^{-2\pi\hrho
   (|k_\ell|-\Ni{\ell}/2)}}{1-\ee^{ -2\pi\hrho \Ni{\ell}}} \right)-
   \ee^{-2\pi\hrho |k|_1}.
\]
\end{proposition}
\bproof Let $k\in \ZZ^n$ be a multi-index. From 
Lemma~\ref{eq:lemaDTF} and standard bounds of the Fourier coefficients
of analytic functions, we obtain
\[
|\tilde f_k - f_k| \leq  \sum_{m\in \ZZ^n\backslash\{ 0\}} |f_{k+\NF(m)}| \leq
\sum_{m\in \ZZ^n\backslash\{ 0\}} \ee^{-2\pi\hrho | k + \NF(m)|_1}
\snorm{f}_{\hrho}.
\]
Then we define
\[
s_{\NF}(k,\hrho) =  \sum_{m\in \ZZ^n} \ee^{-2\pi\hrho | k + \NF(m)|_1}, \qquad
s_{\NF}^*(k,\hrho) = \sum_{m\in \ZZ^n\backslash\{ 0\}} \ee^{-2\pi\hrho | k +
\NF(m)|_1}.
\]
Notice that $s_{\NF}(k,\hrho) = s_{\NF}(k',\hrho)$ for every $k'\in \ZZ^n$ such
that $|k_i|= |k_i'|$ for all $i= 1,\dots n$. Then, we write
\[
s_{\NF}(k,\hrho) = \prod_{\ell= 1}^n s_{\Ni{\ell}}(k_\ell,\hrho), 
\]
where
\[
s_{\Ni{\ell}}(k_\ell,\hrho)=\sum_{m\in \ZZ} \ee^{-2\pi\hrho | k_\ell +
\Ni{\ell} m|}.
\]
Then, by defining $r_\ell \equiv k_\ell \pmod{\Ni{\ell}}$ for
$\ell= 1,\dots,n$, we obtain
\[
\begin{split}
 s_{\Ni{\ell}}(k_\ell,\hrho) & = s_{\Ni{\ell}}(r_{\ell},\hrho) =
 \sum_{m_{\ell}\geq 0} \ee^{-2\pi\hrho (r_{\ell} +
 \Ni{\ell} m_{\ell})} + \sum_{m_{\ell}<0}
 \ee^{-2\pi\hrho (-r_{\ell} - \Ni{\ell} m_{\ell})}\\
 & = \frac{\ee^{2\pi\hrho(r_{\ell}-\Ni{\ell})}+ \ee^{-2\pi\hrho
 r_{\ell}}}{1-\ee^{ -2\pi\hrho \Ni{\ell}}}.
\end{split}
\]
The result follows directly from
$
s_{\NF}^*(k,\hrho) =  s_{\NF}(k,\hrho)  - \ee^{-2\pi\hrho |k|_1}.
$
\eproof

Next, we state the main result of this section, that allows us to control the
error between $\tilde f$ and $f$.

\begin{theorem}
\label{ADFT} Let $f:\TT^{n}_{\hrho} \to \CC$ be an analytic and bounded
function in the complex strip $\TT^{n}_{\hrho}$ of size $\hrho>0$.  Let $\tilde
f$ be the discrete Fourier approximation of $f$ in the regular grid of size
$\NF= (\Ni{1},\dots, \Ni{n}) \in \NN^n$.  Then
\[
\snorm{\tilde f-f}_\rho \leq C_{\NF}(\rho, \hrho) \snorm{f}_{\hrho},
\]
for $0\leq \rho < \hrho$,
where $C_{\NF}(\rho, \hrho)= S_\NF^{*1}(\rho,\hrho) + S_\NF^{*2}(\rho,\hrho)  +
T_\NF(\rho,\hrho)$ is given by
\[
S_\NF^{*1}(\rho,\hrho) = 
\prod_{\ell= 1}^n \frac{1}{1-\ee^{-2\pi  \hrho\Ni{\ell} }}
\sum_{\begin{array}{c} \sigma\in \{-1,1\}^n \\ \sigma\neq (1,\dots,1) \end{array}}
\prod_{\ell= 1}^n \ee^{(\sigma_\ell-1)\pi\hrho \Ni{\ell}} \nuNi{\ell}(\sigma_\ell\hrho-\rho),
\]
\[
S_\NF^{*2}(\rho,\hrho) = \prod_{\ell= 1}^n \frac{1}{1-\ee^{-2\pi
\hrho\Ni{\ell} }}  \left(1- \prod_{\ell= 1}^n  \left(1-\ee^{-2\pi
\hrho\Ni{\ell} }\right)\right) \prod_{\ell= 1}^n \nuNi{\ell}(\hrho-\rho)
\]
and
\[
T_\NF(\rho,\hrho)= \left(  \frac{\ee^{2\pi (\hrho-\rho)} + 1}{\ee^{2\pi
(\hrho-\rho)} -1} \right)^n \ \left( 1 - \prod_{\ell= 1}^n \left(1-
\muNi{\ell}(\hrho-\rho)\ e^{-\pi(\hrho-\rho) \Ni{\ell}} \right) \right),
\]
with
\[
\nuNi{\ell}(\delta)= \frac{\ee^{2\pi \delta} + 1 }{\ee^{2\pi \delta} -1}
\left(1- \muNi{\ell}(\delta) \ \ee^{-\pi \delta \Ni{\ell}}\right) \qquad
\mbox{and} \qquad \muNi{\ell}(\delta) = 
\begin{cases}  
\ 1 &\mbox{if $\Ni{\ell}$ is even} \\ \displaystyle \frac{2
\ee^{\pi\delta}}{\ee^{2\pi\delta}+1} &\mbox{if $\Ni{\ell}$ is odd}
\end{cases}.
\]
\end{theorem}

\bproof From the definition of the discrete Fourier approximation $\tilde f$ of
$f$, we have
\[
\snorm{\tilde f -f }_\rho \leq \sum_{k \in \INF} |\tilde f_k - f_k| \ee^{2\pi
\rho |k|_1} +   \sum_{k \notin \INF} |f_k| \ee^{2\pi \rho |k|_1},
\]
where $\INF$ is the finite set of multi-indices given by Equation~\eqref{eq:INF}.  From
Proposition~\ref{DFTcoefs} and the
growth rate properties of the Fourier coefficients of an analytic function, we
get 
\[
   \snorm{\tilde f -f }_\rho \leq (S_\NF^{*}(\rho,\hrho) + T_\NF(\rho,\hrho))
   \snorm{f}_{\hrho}, 
\]
where 
\[
   S_\NF^{*}(\rho,\hrho) = \sum_{k \in \INF} s^*_N(k,\hrho) \ee^{2\pi\rho
   |k|_1},
\]
and
\[
	T_\NF(\rho,\hrho)= \sum_{k \notin \INF}  \ee^{2\pi(\rho-\hrho) |k|_1}.
\]

Next, we obtain a computable expression for $T_\NF(\rho,\hrho)$. Notice that 
\[
T_\NF(\rho,\hrho) 
 = \sum_{k \in\ZZ^n}  \ee^{2\pi(\rho-\hrho) |k|_1} - \sum_{k \in \INF}  \ee^{2\pi(\rho-\hrho) |k|_1} \\
 = \left(\frac{\ee^{2\pi (\hrho-\rho)}+1}{\ee^{2\pi (\hrho-\rho)}-1}\right)^n - \prod_{\ell= 1}^n \nuNi{\ell}(\hrho-\rho),
\]
where 
\[
\nuNi{\ell}(\delta)= \sum_{k_\ell=
-\left[\frac{\Ni{\ell}}{2}\right]}^{\left[\frac{\Ni{\ell}-1}2\right]}
\ee^{-2\pi\delta |k_\ell|}.
\]
Then, the formula stated in the proposition follows by 
distinguishing the cases where 
$\Ni{\ell}$ is odd and even. 

To obtain a suitable expression for $S_\NF^*(\rho,\hrho)$, we compute
\[
\begin{split}
S_\NF(\rho,\hrho) &  = \sum_{k \in \INF} s_\NF(k,\hrho)
\ee^{2\pi\rho |k|_1}\\ & = \sum_{k\in \INF} \prod_{\ell= 1}^n \left(
\ee^{-\pi\hrho \Ni{\ell}} \frac{\ee^{2\pi\hrho(|k_\ell|- \Ni{\ell}/2)}+
\ee^{-2\pi\hrho (|k_\ell|-\Ni{\ell}/2)}}{1-\ee^{ -2\pi\hrho \Ni{\ell}}}
\ee^{2\pi\rho |k_\ell|} \right) \\ & = 	\prod_{\ell= 1}^n
\frac{\ee^{-\pi\hrho \Ni{\ell}}}{1-\ee^{ -2\pi\hrho \Ni{\ell}}}
\sum_{\sigma\in\{-1,1\}^n}\sum_{k\in \INF} \prod_{\ell= 1}^n \ee^{-2\pi
(\sigma_\ell \hrho-\rho)  |k_\ell|} \ee^{\pi\sigma_\ell\hrho \Ni{\ell}} \\ &=
\prod_{\ell= 1}^n \frac{\ee^{-\pi\hrho \Ni{\ell}}}{1-\ee^{ -2\pi\hrho
\Ni{\ell}}} \sum_{\sigma\in\{-1,1\}^n}\prod_{\ell= 1}^n
\sum_{-\frac{\Ni{\ell}}{2}\leq k_\ell<\frac{\Ni{\ell}}{2}} \ee^{-2\pi
(\sigma_\ell \hrho-\rho)  |k_\ell|} \ee^{\pi\sigma_\ell\hrho \Ni{\ell}} \\ &=
\prod_{\ell= 1}^n \frac{\ee^{-\pi\hrho \Ni{\ell}}}{1-\ee^{ -2\pi\hrho
\Ni{\ell}}} \sum_{\sigma\in\{-1,1\}^n}\prod_{\ell= 1}^n
\ee^{\pi\sigma_\ell\hrho \Ni{\ell}} \nuNi{\ell}(\sigma_\ell\hrho-\rho).
\end{split}
\]
Finally, we use that 
\[
S_\NF^*(\rho,\hrho) 
 = S_\NF(\rho,\hrho)  - \prod_{\ell= 1}^n \nuNi{\ell}(\hrho-\rho), 
\]
and we decompose the resulting expression in the two functions $S_\NF^{*1}$ and
$S_\NF^{*2}$. \eproof

\begin{remark}
In the above formulae, there are expressions of the form $1- \prod_{\ell= 1}^n
(1-x_\ell)$, where $0<x_\ell<1$ for $\ell= 1\dots n$. In our applications, it
turns out that $0<x_\ell \ll 1$, so we have to be aware of the propagation of
the error when enclosing this expression using interval arithmetics.  To this
end, we will use the formulae
\[
1- \prod_{\ell= 1}^n (1-x_\ell) = \sum_{j= 1}^n (-1)^{j -1} \sum_{\substack{
\ell_1<\dots<\ell_j \\
1\leq l_i \leq n 
}}
x_{\ell_1} \dots x_{\ell_j}.
\]
Notice that the dominant term of the expression is $\displaystyle \sum_{\ell=
1}^n x_\ell$.
\end{remark}

\begin{remark}
It is interesting to characterize the dominant terms in the expression
$C_{\NF}(\rho,\hrho)= S_\NF^{*1}(\rho,\hrho) + S_\NF^{*2}(\rho,\hrho)  +
T_\NF(\rho,\hrho)$.  The dominant term of $S_\NF^{*1}(\rho,\hrho)$ corresponds
to the multi-indices $\sigma\in \{-1,1\}^n$ for which only one component is
$-1$.  Hence, we have
\[
S_\NF^{*1}(\rho,\hrho)   \simeq   \left( \frac{\ee^{2\pi (\hrho- \rho)} + 1
}{\ee^{2\pi (\hrho- \rho)} -1} \right)^{n-1} \left(\frac{\ee^{2\pi (\hrho+
\rho)} +1}{\ee^{2\pi (\hrho+ \rho)}-1} \right)   \ \sum_{\ell= 1}^n
\muNi{\ell}(\hrho-\rho)\ \ee^{-\pi(\hrho-\rho) \Ni{\ell}}.
\]
Then, we observe that the dominant term of $S_\NF^{*2}(\rho,\hrho)$, 
\[
S_\NF^{*2}(\rho,\hrho)   \simeq   \left( \frac{\ee^{2\pi (\hrho- \rho)} + 1
}{\ee^{2\pi (\hrho- \rho)} -1} \right)^{n} \ \sum_{\ell= 1}^n \ee^{-2\pi \hrho
\Ni{\ell}}, 
\]
is much smaller than $S_\NF^{*1}(\rho,\hrho)$. Finally, 
\[
   T_\NF(\rho,\hrho)\simeq \left( \frac{\ee^{2\pi (\hrho- \rho)} + 1
   }{\ee^{2\pi (\hrho- \rho)} -1} \right)^{n}\ \sum_{\ell= 1}^n
   \muNi{\ell}(\hrho-\rho)\ \ee^{-\pi(\hrho-\rho) \Ni{\ell}},
\]
which is of the same order as $S_\NF^{*1}(\rho,\hrho)$.
Putting this together we obtain
\[
C_{\NF}(\rho,\hrho) \simeq \left( \frac{\ee^{2\pi (\hrho- \rho)} + 1
}{\ee^{2\pi (\hrho- \rho)} -1} \right)^{n-1} \left( \frac{\ee^{2\pi (\hrho-
\rho)} + 1 }{\ee^{2\pi (\hrho- \rho)} -1}  + \frac{\ee^{2\pi (\hrho+ \rho)} + 1
}{\ee^{2\pi (\hrho+ \rho)} -1} \right) \  \sum_{\ell= 1}^n
\muNi{\ell}(\hrho-\rho)\ \ee^{-\pi(\hrho-\rho) \Ni{\ell}}\,
\]
and $\muNi{\ell}(\hrho-\rho)\leq 1$ implies
\[
C_{\NF}(\rho,\hrho) 
\simeq O(\ee^{-\pi (\hat \rho-\rho) \min_{\ell} \{\Ni{\ell}\}}).
\]
\end{remark}

\subsection{Comments on the $1$-dimensional case}\label{ssec:Four:1D}

The simplest case $n= 1$ deserves especial attention, as it is a
common situation in the literature. Let us formulate Proposition~\ref{DFTcoefs}
and Theorem~\ref{ADFT} in this case.

\begin{corollary}\label{cor:1d:Fou:k}
Let $f:\TT_{\hrho} \to \CC$ be an analytic and bounded function in the 
complex strip $\TT_{\hrho}$ of size $\hrho>0$. Let $\tilde f$ be the discrete Fourier approximation of $f$ in the regular grid of 
size $\NF \in \NN$. Then, for $k= -\left[\frac{\NF}{2}\right], \dots,  \left[\frac{\NF-1}2\right]$, 
\[
|\tilde f_k - f_k| \leq s_{\NF}^*(k,\hrho) \snorm{f}_{\hrho}
\]
where 
\[
	s_{\NF}^*(k,\hrho) =  \frac{\ee^{-2\pi\hrho \NF}}{1-\ee^{-2\pi\hrho \NF}} \left( \ee^{2\pi\hrho k}+ \ee^{-2\pi\hrho k} \right).
\]
\end{corollary}

\begin{corollary}\label{cor:1d:Fou}
Let $f:\TT_{\hrho} \to \CC$ be an analytic and bounded function in the 
complex strip $\TT_{\hrho}$ of size $\hrho>0$. 
Let $\tilde f$ be the discrete Fourier approximation of $f$ in the regular grid of 
size $\NF$.
Then, for $0\leq \rho < \hrho$, we have
\[
\snorm{\tilde f-f}_\rho \leq C_{\NF}(\rho, \hrho) \snorm{f}_{\hrho},
\]
where  $C_{\NF}(\rho, \hrho)= S_\NF^{*1}(\rho,\hrho) + S_\NF^{*2}(\rho,\hrho)  + T_\NF(\rho,\hrho)$, with 
\[
S_\NF^{*1}(\rho,\hrho) = 
 \frac{\ee^{-2\pi\hrho \NF} }{1-\ee^{-2\pi  \hrho\NF }} \ 
\frac{\ee^{-2\pi (\hrho+\rho)} + 1 }{\ee^{-2\pi (\hrho+\rho)} -1} 
\left(1- \muNi{1} (-\hrho-\rho) \ \ee^{\pi(\hrho+\rho) \NF}\right),
\]
\[
S_\NF^{*2}(\rho,\hrho) = 
\frac{\ee^{-2\pi\hrho \NF} }{1-\ee^{-2\pi  \hrho \NF }} \ 
 \frac{\ee^{2\pi (\hrho-\rho)} + 1 }{\ee^{2\pi (\hrho-\rho)} -1} 
\left(1-\muNi{1}(\hrho-\rho)\ \ee^{-\pi(\hrho-\rho)\NF}\right) 
\]
and
\[
T_\NF(\rho,\hrho)=   \frac{\ee^{2\pi (\hrho-\rho)} + 1}{\ee^{2\pi (\hrho-\rho)} -1} \ \muNi{1}(\hrho-\rho)\ e^{-\pi(\hrho-\rho) \NF} .
\]
\end{corollary}

In order to compare 
with~\cite{Epstein04} we consider the odd case, 
for
$\NF= 2M+1$.
In this reference, the following uniform bound was obtained
\[
|\tilde f_k - f_k| \leq {\tilde s}_{\NF}^*(\hrho) \snorm{f}_{\hrho},
\]
for $k= -M,\dots, M$, where 
\[
\tilde s_{\NF}^*(\hrho) = \frac{4 \ee^{-2\pi\hrho M}}{\ee^{2\pi\hrho}-1}.
\]
Let us now compare $s_{\NF}^*(k,\hrho)$ with $\tilde s_{\NF}^*(\hrho)$, for $k= -M,\dots, M$:
\[	
\frac{s_{\NF}^*(k,\hrho)}{\tilde s_{\NF}^*(\hrho)} \leq 
\frac{s_{\NF}^*(M,\hrho)}{\tilde s_{\NF}^*(\hrho)} = 
\frac{1}{4} \left( 1 - \ee^{-2\pi\hrho} \right) \frac{1+\ee^{-2\pi\hrho (\NF-1)}}{1-\ee^{-2\pi\hrho \NF}} \leq \frac{1}{4}.
\]
Notice that the estimates produced in Corollary~\ref{cor:1d:Fou:k} are at least
four times better than the estimates produced in \cite{Epstein04}.

Now we consider Corollary~\ref{cor:1d:Fou} in
the case $\rho=0$. First, we observe that
$\muNi{1}(-\delta)= \muNi{1}(\delta)$, so we get
\[
\snorm{\tilde f-f}_0 \leq C_{\NF}(0, \hrho) \snorm{f}_{\hrho},
\]
with 
\[
C_{\NF}(0, \hrho) = 2\ \frac{\ee^{2\pi \hrho} + 1}{\ee^{2\pi \hrho} -1} \ \muNi{1}(\hrho)\ e^{-\pi\hrho \NF}.
\]
In the odd case, with $\NF= 2M+1$, this expression reads as follows:
\[
C_{\NF}(0, \hrho) = \frac{4\ \ee^{-2\pi\hrho M}}{\left( \ee^{2\pi\hrho} -1\right) \left(1-\ee^{-2\pi\hrho \NF} \right)}.
\]
It is worth 
pointing out
that the best uniform approximation $p^*$ of the form 
$
p(\theta)= \sum_{-M \leq k \leq M} \tilde p_k \ee^{2 \pi \mathrm{i} k \cdot \theta}
$
satisfies (c.f.~\cite{Epstein04})
\[
\snorm{p^*-f}_0 \leq \frac{2  \ee^{-2\pi\hrho M}}{\left( \ee^{2\pi\hrho} -1\right)} \snorm{f}_{\hrho}.
\]
Hence, the discrete Fourier approximation is very close to optimal,
since the corresponding error (approximately) doubles 
the less possible error. 

\subsection{Matrices of periodic functions}\label{ssec:dft:matrices}

In this section we consider some extensions of Theorem~\ref{ADFT} to deal
with matrix functions $A: \TT^n \rightarrow \CC^{m_1 \times m_2}$. Our goal
is to control the propagation of the error when we perform matrix operations.
Specifically, we are interested in the study of products and inverses,
but the ideas
given below can be adapted to control other operations if necessary.

The first result is obtained directly from Theorem~\ref{ADFT}:

\begin{corollary}\label{cor:matrix:multi}
Let us consider two matrix functions $A: \TT^n \rightarrow \CC^{m_1 \times
m_2}$, and $B : \TT^n \rightarrow \CC^{m_2 \times m_3}$, such that their
entries are analytic and bounded functions in the
complex strip $\TT^n_{\hat \rho}$ of size $\hat \rho >0$.  We denote by $A B$
the product matrix and $\smash{\widetilde{A B}}$ the corresponding
approximation given by DFT. Given a grid of size $\NF=(\Ni{1},\ldots,\Ni{n})$,
we evaluate $A$ and $B$ in the grid, and we interpolate the points $A
B(\theta_j)=A(\theta_j) B(\theta_j)$.  Then, we have
\begin{equation}\label{eq:dtf:matrix}
\snorm{AB-\widetilde{A B}}_\rho \leq C_{\NF}(\rho,\hat \rho) \snorm{A}_{\hat \rho}
\snorm{B}_{\hat \rho}
\end{equation}
for every $0 \leq \rho < \hat \rho$,
where $C_{\NF}(\rho,\hat \rho)$ is given in Theorem~\ref{ADFT}.
\end{corollary}

Notice that Corollary~\ref{cor:matrix:multi} is useful to control
the product of approximated objects. If $\tilde A$ and
$\tilde B$ are the corresponding approximations of $A$ and $B$ given by DFT,
then 
\begin{equation}\label{eq:dtf:matrix:trunc}
\snorm{\tilde A \tilde B-\widetilde{\tilde A \tilde B}}_\rho 
\leq C_{\NF}(\rho,\hat \rho) \snorm{\tilde A}_{\hat \rho} \snorm{\tilde B}_{\hat \rho} 
\leq C_{\NF}(\rho,\hat \rho) \snorm{\tilde A}_{F,\hat \rho} \snorm{\tilde B}_{F,\hat \rho} 
\end{equation}
for every $0 \leq \rho < \hat \rho$. 
Notice that since $\tilde A$ and $\tilde B$
are Fourier series with finite support, then
it is interesting to control Equation~\eqref{eq:dtf:matrix:trunc}
using Fourier norms.

The second result allows us to control the inverse of a matrix using the
discrete Fourier approximation:

\begin{corollary}\label{cor:matrix:inv}
Let us consider a matrix function $A: \TT^n \rightarrow \CC^{m \times m}$ whose
entries are analytic and bounded functions in the
complex strip $\TT^n_{\hat \rho}$ of size $\hat \rho >0$.  Given a grid of size
$\NF=(\Ni{1},\ldots,\Ni{n})$, we evaluate $A$ in the grid and compute
the inverses $X(\theta_j)=A(\theta_j)^{-1}$.  Then, if $\tilde X$ is the
corresponding discrete Fourier approximation associated to the sample
$X(\theta_j)$, the error $E(\theta)=I_m-A(\theta) \tilde
X(\theta)$ satisfies
\begin{equation}\label{eq:error:inverse}
\snorm{E}_\rho \leq C_{\NF} (\rho,\hat \rho) \snorm{A}_{\hat \rho} \snorm{\tilde
X}_{\hat \rho},
\end{equation}
for $0 \leq \rho < \hat \rho$. Moreover, if $\snorm{E}_\rho<1$, there exists an
analytic inverse $A^{-1}: \TT^n \rightarrow \CC^{m \times m}$
satisfying
\begin{equation}\label{eq:inverese:A}
\snorm{A^{-1}-\tilde X}_\rho \leq \frac{\snorm{\tilde X}_{\hat \rho}
\snorm{E}_\rho}{1-\snorm{E}_\rho}.
\end{equation}
\end{corollary}

\begin{proof}
To obtain Equation~\eqref{eq:error:inverse} we observe that if $\widetilde{ A \tilde X}$
is the discrete Fourier approximation of $A \tilde X$, then it turns out that
\[
(A \tilde X)(\theta_j)=A(\theta_j) \tilde X(\theta_j)=I_m
\]
for all points in the grid. This implies that $\widetilde{ A \tilde X}=I_m$
and we end up with
\[
\snorm{E}_\rho = \snorm{I_m - A \tilde X}_\rho = \snorm{\widetilde{A \tilde X}
- A \tilde X}_\rho
\]
and Inequality~\eqref{eq:error:inverse} follows applying
Corollary~\ref{cor:matrix:multi}. 
Inequality~\eqref{eq:inverese:A} follows from the expression
$E=I_m - A \tilde X$, simply writing $A^{-1} = \tilde X (I_m - E)^{-1}$ and
using a Neumann series argument. 
\end{proof}

\section{Dealing with the small divisors}\label{sec:CAPlemmas}

In this Section we discuss two technical auxiliary results that
play a fundamental role in KAM theory: the characte\-rization
of the Diophantine constants $(\gamma, \tau)$ and the computation
of the R\"ussmann constant $c_R$. 
In Section~\ref{ssec:diophantine} 
we propose a general method to assign Diophantine constants to a given interval vector of frequencies, such that
the corresponding set of Diophantine vectors has positive measure.
In Section~\ref{ssec:superrussmann} we revisit the classic R\"ussmann estimates.
To take into account the effect of small divisors, we compute
the first elements explicitly and
then we control the remaining tail analytically. In this way,
with the help of the computer, we obtain sharper estimates than
in the classic literature.

\subsection{On the characterization of Diophantine
constants}\label{ssec:diophantine}

A fundamental hypothesis of Theorem~\ref{theo:KAM} is the fact that $\omega \in \RR^n$
satisfies Diophan\-tine conditions. 
To ensure it, we enclose $\omega$ with an interval vector
$\iomega$ and we look for constants $(\gamma,\tau)$ such that $\iomega$
contains $(\gamma, \tau)$-Diophantine vectors.
The estimates presented in this section
are based on two elementary observations. First, that we only need to give
a lower bound of the measure of vectors   $\omega\in \iomega$ satisfying 
$|k\cdot \omega - m| \geq \gamma |k|_1^{-\tau}$, for every $k \in \ZZ^n \backslash\{0\}$ and $m \in
\ZZ$. Second, that 
this lower bound is obtained by splitting the
computations in two parts: the low resonances are checked rigorously with the
help of the computer, while the measure of the high resonances are bounded
analytically.

Consider an interval vector of the form $\iomega=\prod_{i=1}^n [a_i, b_i]$, and Diophantine constants $(\gamma,\tau)$.
For each index $k\in \ZZ^n\backslash \{0\}$ we define the $k$-resonant set of type $(\gamma,\tau)$ as 
\[
\mathrm{Res}_k(\iomega, \gamma, \tau)=
\bigcup_{m \in \ZZ}
\bigg\{\omega \in \iomega : |k\cdot \omega -m| < \frac{\gamma}{|k|_1^\tau} \bigg\}, 
\]
so that the resonant set of type $(\gamma,\tau)$ is 
\[
	\mathrm{Res}(\iomega, \gamma, \tau) = \bigcup_{k\in\mathbb Z^n\backslash\{0\} } \mathrm{Res}_k(\iomega,\gamma, \tau).
\]
The relative measure  of the set of $(\gamma,\tau)$-Diophantine vectors
in  $\iomega$ is
\begin{equation}
\label{def: relative measure}
	p(\omega,\gamma,\tau) = 
	1- \frac{\meas\left(\mathrm{Res}(\iomega,\gamma,\tau)\right)}{\meas (\iomega)},
\end{equation}
where $\meas (A)$ stands for the Lebesgue measure of a Borel set $A$. 
Our goal is to obtain positive lower bounds of $p(\omega,\gamma,\tau)$. To do so, we control the resonant set by fixing $M$ sufficiently big, and using the decomposition
\[
\mathrm{Res}( \iomega, \gamma,\tau) = \mathrm{Res}_{\leq M} (\iomega, \gamma, \tau)\cup \mathrm{Res}_{> M}
(\iomega, \gamma, \tau), 
\]
where $\mathrm{Res}_{\leq M} (\iomega, \gamma, \tau)$ and $\mathrm{Res}_{> M}
(\iomega, \gamma,\tau)$ are, respectively, the sets of resonances with index
$k$ satisfying $|k|_1 \leq M$ and $|k|_1 > M$.  By choosing $\gamma$ sufficienty small, we get 
$\mathrm{Res}_{\leq M} (\iomega, \gamma, \tau) = \emptyset$, and then we have to get an  upper bound of the measure of
$\mathrm{Res}_{> M} (\iomega, \gamma, \tau)$. These arguments are the core of the proof of the following proposition. 

\begin{proposition}\label{cor:Dioph}
Let $\iomega=\prod_{i=1}^n [a_i, b_i]$ be an interval vector, whose diameter is $\mathrm{diam}(\iomega)=\sqrt{\sum_{i=1}^n(b_i-a_i)^2}$.
Given $M\geq n$, we assume that  for any $\omega\in\iomega$, 
$k\in \ZZ^n$ such that $0<|k|_1\leq M$, and $m \in \ZZ$, we have $k\cdot \omega-m\neq 0$.
For any  $\tau>n$, we define 
\[
	\gamma_M(\iomega,\tau)=  \min \{ |k \cdot \omega-m| |k|_1^{\tau} : \, \omega \in \iomega, \, 0<|k|_1\leq M, \, m \in \ZZ  \}.
\]
Then, for any positive $\gamma\leq \gamma_M(\iomega,\tau)$, we have 
\begin{equation}
\label{eq: lower bound}
	p(\iomega,\gamma,\tau) 
	>      1- \frac{C(\iomega,n) \gamma}{(\tau-n) M^{\tau-n}} 
	\geq 1- \frac{C(\iomega,n) \gamma_M(\iomega,\tau)}{(\tau-n) M^{\tau-n}}, 
\end{equation}
where 
\[
	C(\iomega,n)= \frac{2^{2n}}{(n-1)!} \frac{(\mathrm{diam}(\iomega))^n}{\meas(\iomega)}.
\]
Moreover, the equation for $\tau$ 
\begin{equation}
\label{eq: bound resonant}
	1-\frac{C(\iomega,n) \gamma_M(\iomega,\tau)}{(\tau-n) M^{\tau-n}} = 0
\end{equation}
has a unique solution $\tau_M(\iomega)$, 
for which we define $\gamma_{M}(\iomega)= \gamma_{M}(\iomega,\tau_{M}(\iomega))$. 
As a consequence, for any pair $(\gamma,\tau)$ with $\tau\geq \tau_{M}(\iomega)$ and 
$\gamma\leq \gamma_{M}(\iomega)$, we have  $p(\iomega,\gamma,\tau)>0$.
\end{proposition}
\begin{proof}
Notice that, for  $\gamma\leq \gamma_M(\iomega,\tau)$, we obtain $\mathrm{Res}(\iomega, \gamma, \tau) = \mathrm{Res}_{>M} (\iomega, \gamma, \tau)$, since $\mathrm{Res}_{\leq M} (\iomega, \gamma, \tau) = \emptyset$.
In order to get an upper bound of $\meas (\mathrm{Res}_{> M}(\iomega, \gamma, \tau ))$, we use the elementary estimate 
\[
\meas (\mathrm{Res}_k(\iomega, \gamma, \tau ))\leq 2\gamma |k|_1^{-\tau} (\mathrm{diam}(\iomega))^{n}, 
\]
for any $k\in\ZZ^n\setminus\{0\}$ (see the proof of Lemma 2.11 of~\cite{Llave01}).
Consequently,
\begin{equation}\label{eq:reso:cont}
\meas (\mathrm{Res}_{> M} (\iomega, \gamma, \tau))  
\leq 2 (\mathrm{diam}(\iomega))^{n} \gamma \sum_{|k|_1> M} |k|_1^{-\tau}
< \frac{2^{2n} (\mathrm{diam}(\iomega))^{n}}{(n-1)! (\tau-n) M^{\tau-n}}.
\end{equation}
The lower bound \eqref{eq: lower bound} follows immediately. 

Let us consider the function $\hat p:(n,\infty)\to \RR$ defined by 
\begin{equation}
\hat p(\tau)= 1-\frac{C(\iomega,n) \gamma_M(\iomega,\tau)}{(\tau-n) M^{\tau-n}}.
\end{equation}
Then, from the definition of $\gamma_M(\iomega,\tau)$, 
\[
\hat p(\tau)= 1-\frac{C(\iomega,n) }{(\tau-n) M^{-n}} \min \left\{ |k \cdot \omega-m| \left(\frac{|k|_1}{M}\right)^{\tau} : \, \omega \in \iomega, \, 0<|k|_1\leq M, \, m \in \ZZ  \right\},
\]
from where we deduce that $\hat p$ is a strictly increasing function of $\tau$, and $\displaystyle\lim_{\tau\to n^+} \hat p(\tau)= -\infty$. Moreover, since 
\[
\hat p(\tau) \geq 1-\frac{C(\iomega,n) }{2(\tau-n) M^{-n}},
\]
then $\displaystyle\lim_{\tau\to \infty} \hat p(\tau)= 1$.
Hence, there exists a unique $\tau= \tau_M(\iomega)>n$ such that $\hat p(\tau)= 0$. The rest of the proof follows immediately. 
\end{proof}

\begin{remark}\label{cubic interval vector}
The constant $C(\iomega,n)$ in Proposition~\ref{cor:Dioph} is specially simple if the length of the edges of the interval vector are equal:
\[
	C(\iomega,n) = \frac{2^{2n} n^{\tfrac{n}2}}{(n-1)!} =: C(n).
\]
In particular, we have  $C(1)= 4$ and $C(2)= 32$.
\end{remark}

In this paper we are not interested in maximizing the measure of Diophantine
numbers but simply in guaranteeing that it is positive. 
Proposition~\ref{cor:Dioph} provides a simple algorithm to associate pairs
$(\gamma,\tau)$ to an interval vector $\iomega$. Specifically, we take $M\geq n$ and solve Equation~\eqref{eq: bound resonant} 
to find the pair $(\gamma_M(\iomega), \tau_M(\iomega))$. 
Notice that  $\tau_M(\iomega)$ is the minimum value of $\tau$ that
guarantees the existence of Diophantine vectors in 
$\iomega$ for $\gamma\leq \gamma_M(\iomega)$. In order to produce
larger sets of Diophantine vectors, one can take 
a pair $(\gamma_M(\iomega,\tau),\tau)$, with $\tau\geq \tau_M(\iomega)$ (since the function $\hat p$ defined in the proof is strictly increasing). 
Giving measure estimates of invariant tori
(both in frequency space and in phase space)
is a very interesting problem that deserves full attention
and will be considered in future work.

%

To illustrate the previous construction we consider (tight) interval frequencies 
$\iomega$
enclosing
\begin{equation}\label{eq:omega:a,b}
[\omega_{a,b} - 2^{-50},\omega_{a,b} + 2^{-50}], \quad \mbox{with} \quad
\omega_{a,b}=\frac{\sqrt{b^2 + 4 b/a}-b}{2}.
\end{equation}
In Table~\ref{tab:apenA:1d}
we provide Diophantine constants $(\gamma,\tau)$,
for some of these interval frequencies, 
such that the Lebesgue measure of
$(\gamma, \tau)$-Diophantine frequencies in $\iomega$ is positive. 
The computations have been performed using the interval arithmetics
library MPFI (see~\cite{RevolR05}) taking a precision of 64 bits.
The values of Table~\ref{tab:apenA:1d} are obtained using $M=1000$. If we take a larger value
of $M$, then we can obtain a smaller lower value of $\tau$. For example,
for the golden mean (with $a=b=1$),
taking $M=10^3$ we obtain $\tau \geq 1.26$,
taking $M=10^4$ we obtain $\tau \geq 1.22$, taking
$M=10^5$ we obtain $\tau \geq 1.19$,
and taking
$M=10^6$ we obtain $\tau \geq 1.17$,

\begin{table}[!h]
\centering
{\scriptsize
\begin{tabular}{|c c| c c || c c| c c || c c |c c|}
\hline
$a$ & $b$ & $\gamma \leq$ & $\tau \geq$ & 
$a$ & $b$ & $\gamma \leq$ & $\tau \geq$ & 
$a$ & $b$ & $\gamma \leq$ & $\tau \geq$ \\
\hline \hline
1 & 1 & 0.381966011250104 & 1.26 & 3 & 1 & 0.263762615825972 & 1.23 & 5 & 1 & 0.170820393249935& 1.19 \\
1 & 2 & 0.267949192431121 & 1.23 & 3 & 2 & 0.290994448735804 & 1.24 & 5 & 2 & 0.183215956619922& 1.20 \\
1 & 3 & 0.208712152522079 & 1.21 & 3 & 3 & 0.302775637731993 & 1.24 & 5 & 3 & 0.188194301613412& 1.20 \\
1 & 4 & 0.171572875253808 & 1.19 & 3 & 4 & 0.277309053319640 & 1.23 & 5 & 4 & 0.190890230020663& 1.20 \\
1 & 5 & 0.145898033750314 & 1.18 & 3 & 5 & 0.223037765858308 & 1.21 & 5 & 5 & 0.192582403567251& 1.20 \\
1 & 6 & 0.127016653792582 & 1.17 & 3 & 6 & 0.187329140491556 & 1.20 & 5 & 6 & 0.193743884534261& 1.20 \\
2 & 1 & 0.366025403784437 & 1.26 & 4 & 1 & 0.207106781186546 & 1.21 & 6 & 1 & 0.145497224367901& 1.18 \\
2 & 2 & 0.413767832000904 & 1.27 & 4 & 2 & 0.224744871391588 & 1.22 & 6 & 2 & 0.154700538379250& 1.18 \\
2 & 3 & 0.300011472016747 & 1.24 & 4 & 3 & 0.232050807568876 & 1.22 & 6 & 3 & 0.158312395177698& 1.19 \\
2 & 4 & 0.235323972166368 & 1.22 & 4 & 4 & 0.236067977499788 & 1.22 & 6 & 4 & 0.160246899469285& 1.19 \\
2 & 5 & 0.192798030208926 & 1.20 & 4 & 5 & 0.238612787525829 & 1.22 & 6 & 5 & 0.161453237111884& 1.19 \\
2 & 6 & 0.163806299636515 & 1.19 & 4 & 6 & 0.206140402288459 & 1.21 & 6 & 6 & 0.162277660168378& 1.19 \\
\hline
\end{tabular}
\caption{ {\footnotesize Rigorously computed Diophantine constants
$(\gamma,\tau)$ ensuring positive measure for several $1$-dimensional (tight) interval
frequencies $\iomega$ enclosing intervals given
by Equation~\eqref{eq:omega:a,b}.  We use the methodology derived from
Proposition~\ref{cor:Dioph} with $M=1000$.  }} \label{tab:apenA:1d}}
\end{table}

As an illustration for
$n=2$, we consider (tight) interval
frequency vectors 
$\iomega$ enclosing
\begin{equation}\label{eq:omegap}
[\omega_p - 2^{-50},\omega_p + 2^{-50}] \times
[\omega_q - 2^{-50},\omega_q + 2^{-50}] \subset \RR^2,
\end{equation}
with $\omega_p= \sqrt{p}-[\sqrt{p}]$
and 
$\omega_q= \sqrt{q}-[\sqrt{q}]$,
where $[\, \cdot \,]$ stands for the integer part.
The computations have been performed using the interval arithmetics
library MPFI (see~\cite{RevolR05}) taking a precision of 64 bits. Diophantine constants
associated to these intervals are given in Table~\ref{tab:apenA:2d}, using
the methodology derived from Proposition~\ref{cor:Dioph} with $M=1000$.

\begin{table}[!h]
\centering
{\scriptsize
\begin{tabular}{|c| c c c c c c|}
\hline
 &
$p=2$ &
$p=3$ &
$p=5$ &
$p=7$ &
$p=11$ &
$p=13$ \\
\hline \hline
$q=2$ & \cellcolor[gray]{0.8}& $0.1421950391579065$ & $0.0100027079360758$ & $0.0216249622296120$ & $0.1108017984300913$ & $0.0642768464253883$ \\
$q=3$ & $2.40$ & \cellcolor[gray]{0.8}& $0.1706190237467459$ & $0.3712175538099537$ & $0.0587500580246455$ & $0.2621757106909654$ \\
$q=5$ & $2.14$ & $2.42$ & \cellcolor[gray]{0.8}& $0.0090434339906178$ & $0.0939733872704087$ & $0.1120569499083642$ \\
$q=7$ & $2.20$ & $2.51$ & $2.13$ & \cellcolor[gray]{0.8}& $0.0139696161322762$ & $0.1940680771272715$ \\
$q=11$& $2.37$ & $2.30$ & $2.35$ & $2.17$ & \cellcolor[gray]{0.8}& $0.0044815206623438$ \\
$q=13$ & $2.31$ & $2.47$ & $2.37$ & $2.43$ & $2.09$ & \cellcolor[gray]{0.8}\\
\hline
\end{tabular}
\caption{
{\footnotesize 
Rigorously computed Diophantine constants $(\gamma,\tau)$ ensuring
positive measure for several $2$-dimensional
(tight) interval frequencies
$\iomega$ enclosing intervals given by Equation~\eqref{eq:omegap}.
We use the methodology derived from Proposition~\ref{cor:Dioph} with $M=1000$.
An upper value of $\gamma$ is given above the diagonal and a lower value of $\tau$
is given below the diagonal.}}
\label{tab:apenA:2d}
}
\end{table}

\subsection{On the R\"ussmann estimates}\label{ssec:superrussmann}

In this section we present a version of the R\"ussmann estimates that 
is tailored to be evaluated with the
help of the computer. We use the notation introduced in
Section~\ref{sec-anal-prelims}.

\begin{lemma}\label{lem-Russmann}
Let $\omega\in \RR^n$ be a $(\gamma,\tau)$-Diophantine frequency vector, for
certain $\gamma>0$ and $\tau\geq n$ (see Definition~\ref{def-Diophantine}).
Then, for any analytic function $v: \TT_\rho^n \rightarrow \CC$, with
$\snorm{v}_\rho < \infty$ and $\rho>0$,  there exists a unique zero-average
analytic solution $u:  \TT_\rho^n \rightarrow \CC$ of $\L u=v-\avg{v}$, denoted
by $u= \R v$. Moreover, given $L\in \NN$, for any $0<\delta<\rho$ we have 
\begin{equation}\label{eq:Russmann}
\snorm{u}_{\rho-\delta} \leq \frac{c_R(\delta)}{\gamma\delta^{\tau}}
\snorm{v}_{\rho},
\end{equation}
where
\begin{equation}\label{eq:russmann:mejor}
c_R(\delta) =\sqrt{ \gamma^2 \delta^{2\tau} 2^n \sum_{0<|k|_1\leq L}
\frac{\mathrm{e}^{-4 \pi |k|_1\delta}}{4 |\sin(\pi k \cdot \omega)|^2}  +
2^{n-3} \zeta(2,2^\tau) (2\pi)^{-2\tau} \int_{4\pi\delta (L+1)}^\infty
u^{2\tau} e^{-u}\ \dif u}
\end{equation}
and $\zeta(a,b)=\sum_{j\geq 0} (b+j)^{-a}$ is the Hurwitz zeta function.
\end{lemma}

\begin{proof}
We follow standard arguments  (see~\cite{ Llave01, HaroCFLM, LuqueV11,
Russmann75, Russmann76a}), with an eye in the feasibility of computing rigorous
upper bounds of finite sums (up to order $L$).

We control the divisors in the expansion of the function $u(\theta)= \R
v(\theta)$, formally given by Equation~\eqref{eq:small:formal}, as
\[
| 1-\mathrm{e}^{2\pi \mathrm{i} k \cdot \omega}  | =  2 |\sin(\pi k \cdot
\omega)| \geq 2^2 \min_{m \in \ZZ} |k\cdot \omega -m|,
\]
where we used that $\sin x \geq 2x/\pi$ if $x \leq \pi/2$ and that ${\min_{m\in
\ZZ} |k\cdot \omega-m| < 1/2}$.  Then, it is natural to introduce the notation
$d_k= k\cdot \omega - m_k$, such that $|d_k|=  \min_{m\in \ZZ} |k\cdot
\omega-m|$.  Notice that the divisors $d_k$ satisfy $d_{k_1} \neq d_{k_2}$ if
$k_1\neq k_2$, and $d_{-k}= -d_k$.  The Diophantine condition in
\eqref{eq:Diophantine} reads $|d_k| \geq \gamma |k|_1^{-\tau}$. 

We then control the norm of $u$ as 
\begin{align}
\|u\|_{\rho-\delta} \leq {} & \|u\|_{F,\rho-\delta} \leq \sum_{k \in
\ZZ^n\backslash \{0\}} \frac{|\hat v_k|}{2 |\sin(\pi k \cdot \omega)|}
\mathrm{e}^{2 \pi |k|_1 (\rho-\delta)} \nonumber \\ \leq {} & \bigg( \sum_{k
\in \ZZ^n\backslash \{0\}} |\hat v_k|^2 \mathrm{e}^{4\pi |k|_1\rho}
\bigg)^{1/2} \bigg( \sum_{k \in \ZZ^n\backslash \{0\}} \frac{\mathrm{e}^{-4 \pi
|k|_1\delta}}{2^2\sin^2(\pi k \cdot \omega) } \bigg)^{1/2},
\label{eq:Cauchy-Schwarz}
\end{align}
where we used Cauchy-Schwarz inequality. On the one hand, the first term 
is bounded by
\begin{equation}\label{eq:Bessel}
\sum_{k \in \ZZ^n\backslash \{0\}} |\hat v_k|^2 \mathrm{e}^{4 \pi |k|_1\rho} \leq 
\sum_{k \in \ZZ^n} |\hat v_k|^2 \mathrm{e}^{4 \pi |k|_1\rho} \leq
2^n \|v\|^2_\rho
\end{equation}
(see \cite{Russmann75}), 
and on the other hand, the second term is bounded by computing the sum up to order $L$ and controlling the tail,
\begin{equation}\label{eq:summa}
\sum_{k \in \ZZ^n\backslash \{0\}} \frac{\mathrm{e}^{-4 \pi |k|_1\delta} }{4 |\sin(\pi k \cdot \omega)|^2} 
  = \sum_{0<|k|_1\leq L} \frac{\mathrm{e}^{-4 \pi |k|_1\delta}}{4 |\sin(\pi k \cdot \omega)|^2} + 
    \sum_{|k|_1 > L} \frac{\mathrm{e}^{-4 \pi |k|_1\delta}}{4 |\sin(\pi k \cdot \omega)|^2}.
\end{equation}
By using that $2 |\sin(\pi k \cdot \omega)|\geq 2^2 |d_k|$, and the Abel summation formula, we bound the tail by
\begin{equation}\label{eq:tail}
\sum_{|k|_1 > L} \frac{\mathrm{e}^{-4 \pi |k|_1\delta}}{4 |\sin(\pi k \cdot \omega)|^2} \leq 
\sum_{\ell=L+1}^\infty \bigg(\sum_{\substack{k \in \ZZ^n\backslash \{0\} \\ |k|_1 \leq \ell}} \frac{1}{2^4|d_k|^2}\bigg) (\mathrm{e}^{- 4 \pi \ell \delta} - \mathrm{e}^{-4 \pi (\ell+1) \delta}).
\end{equation}

Then, given $\ell \in \NN$, we define the set of positive divisors up to order $\ell$ as 
\[
\mathcal{D}_\ell = \{k \in \ZZ^n\backslash \{0\}\,:\, |k|_1 \leq \ell~\mbox{and}~ d_k >0\},
\]
and we sort the divisors according to $0 < d_{k_1} < \ldots < d_{k_{\# \mathcal{D}_\ell}}$ with $k_j \in \mathcal{D}_\ell$, for $j=1,\ldots,\#\mathcal{D}_\ell$. 
We obtain recursively that 
\[
	d_{k_j} = (d_{k_j} - d_{k_{j-1}}) + \dots  + (d_{k_2} -d_{k_{1}} ) + d_{k_1} \\
	            \geq (j-1) \gamma (2\ell)^{-\tau} + \gamma \ell^{-\tau},
\]
where we used that $|k_1|_1 \leq  \ell$ and $|{k_i} - {k_{i-1}}|_1 \leq 2\ell$.
Then, we have
\[
\sum_{j=1}^{\# \mathcal{D}_\ell} \frac{1}{(d_{k_{j}})^2} \leq \sum_{j=1}^\infty \frac{\ell^{2\tau}}{\gamma^2(1+(j-1) 2^{-\tau})^2} = 
2^{2\tau} \sum_{j= 0}^\infty \frac1{(2^\tau+j)^2}  \frac{\ell^{2\tau}}{\gamma^2} = 
2^{2\tau} \zeta(2,2^\tau) \frac{\ell^{2\tau}}{\gamma^2}.
\]
The same result is obtained for the sum corresponding to the negative
divisors up to order $\ell$. 

Finally, we control the sum of Equation~\eqref{eq:tail} as follows
\begin{align}
\sum_{|k|_1 > L} & \frac{\mathrm{e}^{-4 \pi |k|_1\delta}}{4 |\sin(\pi k \cdot \omega)|^2} \nonumber
\leq 
\frac{2^{2\tau} \zeta(2,2^\tau) }{2^3 \gamma^2} \sum_{\ell>L} \ell^{2\tau} \int_\ell^{\ell+1} 4 \pi \delta \mathrm{e}^{-4 \pi \delta x} \dif x \nonumber \\
& \leq \frac{4 \pi \delta 2^{2\tau} \zeta(2,2^\tau) }{2^3 \gamma^2} \int_{L+1}^{\infty}x^{2\tau}\mathrm{e}^{-4 \pi \delta x} \dif x
= \frac{2^{-3} \zeta(2,2^\tau)}{\gamma^2 (2\pi \delta)^{2\tau}} \int_{4\pi\delta(L+1)}^\infty u^{2\tau} e^{-u} \dif u .
\label{eq:control:div}
\end{align}
Combining Equations~\eqref{eq:Cauchy-Schwarz}, \eqref{eq:Bessel}, and \eqref{eq:control:div},  we end up with the stated estimate.
\end{proof}

\begin{remark}
Taking $L= 0$ we obtain 
\begin{equation}\label{eq:classic:russmann}
c_R(\delta) =\sqrt{ 2^{n-3} \zeta(2,2^\tau) (2\pi)^{-2\tau} \int_{4\pi\delta}^\infty u^{2\tau} e^{-u}\ \dif u}	 \leq 
	 \sqrt{ 2^{n-3} \zeta(2,2^\tau) (2\pi)^{-2\tau} \Gamma(2\tau +1)},
\end{equation}
which gives us the classic (uniform) R\"ussmann estimate.
\end{remark}

If we use a computer to control the first divisors explicitly,
then it turns out that the expression $c_R(\delta)$ in Equation~\eqref{eq:russmann:mejor}
improves the classic estimate in Equation~\eqref{eq:classic:russmann}. To this end,
we enclose $\omega$ with an interval vector $\iomega$, as described in
Section~\ref{ssec:diophantine}, and
we rigorously enclose the finite sum for $0<|k|_1 \leq L$
using interval arithmetics. We consider upper bounds of the integral in the tail
using that, if $y>x$,
\[
\int_{y}^\infty u^x e^{-u} \dif u \leq \frac{y}{y-x} y^x e^{-y}.
\]
Applying this last estimate requires to take $L$ such that $4\pi\delta(L+1)> 2 \tau$ in \eqref{eq:control:div}.
In practice, we take a value of $L$ large enough 
in such a way that the contribution of the tail is smaller than the contribution
of the finite sum. 

%
%
In Tables~\ref{tab:russmann:1d} and~\ref{tab:russmann:2d} we present some
rigorous bounds of $c_R(\delta)$, given by Equation~\eqref{eq:russmann:mejor}, for several values of $\delta$ and we compare them
with the classic constant $c_R$, given by Equation~\eqref{eq:classic:russmann}. Specifically, in
Table~\ref{tab:russmann:1d} we consider 
the same 1-dimensional interval frequencies characterized in Table~\ref{tab:apenA:1d}.
In Table~\ref{tab:russmann:2d} we consider 
the same 2-dimensional interval frequencies characterized in Table~\ref{tab:apenA:2d}.
The computations have been performed using the interval arithmetics
library MPFI (see~\cite{RevolR05}) taking a precision of 64 bits.
We observe that the improvement of Equation~\eqref{eq:russmann:mejor} is remarkable and,
in some cases, we improve the classic constant by several orders of magnitude.
As we will see, this represents an important improvement in order to apply the KAM theorem.

\begin{table}[!h]
\centering
{\scriptsize
\begin{tabular}{|c c| c | c c c c c |}
\hline
\multicolumn{2}{|c|}{$\omega_{a,b}$} &
\multicolumn{6}{c|}{$c_R(\delta) \leq$} \\
\hline
$a$ & $b$ & $L=0$ & 
$\delta=0.1$ &
$\delta=0.01$ &
$\delta=0.001$ &
$\delta=0.0001$ &
$\delta=0.00001$ \\
\hline \hline
1 & 1 & 6.53700395e-02 & 1.70002315e-02 & 1.01408017e-02 & 5.57856565e-03 & 3.06566441e-03 & 1.68472062e-03\\
1 & 2 & 6.76832915e-02 & 1.40967097e-02 & 8.84204788e-03 & 5.26914629e-03 & 3.08884783e-03 & 1.79851814e-03\\
1 & 3 & 6.92956514e-02 & 1.29193920e-02 & 8.71895511e-03 & 5.18084015e-03 & 3.29070636e-03 & 1.98865837e-03\\
1 & 4 & 7.09671322e-02 & 1.25818941e-02 & 8.36835924e-03 & 5.84462956e-03 & 3.60202865e-03 & 2.21507905e-03\\
1 & 5 & 7.18258809e-02 & 1.23983728e-02 & 7.88617196e-03 & 5.47041537e-03 & 3.93080619e-03 & 2.62830129e-03\\
1 & 6 & 7.27004339e-02 & 1.24386826e-02 & 8.01767004e-03 & 5.19907101e-03 & 3.59023334e-03 & 2.59913106e-03\\
2 & 1 & 6.53700395e-02 & 1.74727821e-02 & 1.04651045e-02 & 5.78939864e-03 & 3.21601209e-03 & 1.75767438e-03\\
2 & 2 & 6.46258261e-02 & 1.87353432e-02 & 1.09024839e-02 & 5.85645816e-03 & 3.14665921e-03 & 1.68979306e-03\\
2 & 3 & 6.68984767e-02 & 1.63024413e-02 & 9.88023684e-03 & 5.56751156e-03 & 3.12759392e-03 & 1.77052228e-03\\
2 & 4 & 6.84822445e-02 & 1.53872642e-02 & 9.71562658e-03 & 5.85967409e-03 & 3.53201756e-03 & 2.12887318e-03\\
2 & 5 & 7.01238363e-02 & 1.50961843e-02 & 9.60632258e-03 & 5.78570281e-03 & 3.43840456e-03 & 2.04301934e-03\\
2 & 6 & 7.09671322e-02 & 1.48540346e-02 & 9.04822601e-03 & 5.05304849e-03 & 2.95729166e-03 & 1.98349009e-03\\
3 & 1 & 6.76832915e-02 & 1.46809408e-02 & 1.00993505e-02 & 5.63999620e-03 & 3.50708184e-03 & 1.95894122e-03\\
3 & 2 & 6.68984767e-02 & 1.43496537e-02 & 8.98100848e-03 & 5.22239107e-03 & 3.08135263e-03 & 1.82028295e-03\\
3 & 3 & 6.68984767e-02 & 1.58579091e-02 & 9.83769747e-03 & 5.65400813e-03 & 3.24897579e-03 & 1.86887029e-03\\
3 & 4 & 6.76832915e-02 & 1.62566583e-02 & 9.69074106e-03 & 5.59759584e-03 & 3.32500705e-03 & 1.99628047e-03\\
3 & 5 & 6.92956514e-02 & 1.50813939e-02 & 8.58023012e-03 & 5.31739360e-03 & 3.58299882e-03 & 2.46021534e-03\\
3 & 6 & 7.01238363e-02 & 1.42772183e-02 & 7.56136730e-03 & 5.20508341e-03 & 4.02265302e-03 & 2.54652970e-03\\
4 & 1 & 6.92956514e-02 & 1.31882573e-02 & 9.51364826e-03 & 6.14222557e-03 & 3.50022063e-03 & 2.17640510e-03\\
4 & 2 & 6.84822445e-02 & 1.26639499e-02 & 8.35820776e-03 & 5.03133457e-03 & 3.02541383e-03 & 1.81923509e-03\\
4 & 3 & 6.84822445e-02 & 1.32535840e-02 & 9.03724309e-03 & 5.61174299e-03 & 3.33248407e-03 & 1.90879693e-03\\
4 & 4 & 6.84822445e-02 & 1.39802574e-02 & 9.34078901e-03 & 5.55619944e-03 & 3.32472818e-03 & 2.05369884e-03\\
4 & 5 & 6.84822445e-02 & 1.48184252e-02 & 9.49849042e-03 & 5.85472306e-03 & 3.88991086e-03 & 2.07085456e-03\\
4 & 6 & 6.92956514e-02 & 1.38269028e-02 & 8.66584042e-03 & 6.22603715e-03 & 3.90134199e-03 & 1.99958702e-03\\
5 & 1 & 7.09671322e-02 & 1.26864212e-02 & 8.86482905e-03 & 6.21678963e-03 & 4.26563195e-03 & 2.68508905e-03\\
5 & 2 & 7.01238363e-02 & 1.22725538e-02 & 8.83966465e-03 & 5.80010426e-03 & 3.74378915e-03 & 2.37051014e-03\\
5 & 3 & 7.01238363e-02 & 1.24917622e-02 & 9.34589411e-03 & 5.62325597e-03 & 3.07349358e-03 & 1.89609801e-03\\
5 & 4 & 7.01238363e-02 & 1.27481886e-02 & 9.35639250e-03 & 5.19209738e-03 & 3.32315143e-03 & 2.18847977e-03\\
5 & 5 & 7.01238363e-02 & 1.30462113e-02 & 9.23511743e-03 & 5.79164682e-03 & 3.51817789e-03 & 2.36970971e-03\\
5 & 6 & 7.01238363e-02 & 1.33875654e-02 & 9.20387961e-03 & 6.71312444e-03 & 3.57979848e-03 & 2.30908333e-03\\
6 & 1 & 7.18258809e-02 & 1.24345830e-02 & 8.50000514e-03 & 5.57984051e-03 & 3.88857456e-03 & 2.76605873e-03\\
6 & 2 & 7.18258809e-02 & 1.23988684e-02 & 9.62126814e-03 & 6.51124354e-03 & 4.05219896e-03 & 2.34111803e-03\\
6 & 3 & 7.09671322e-02 & 1.22004393e-02 & 9.29930218e-03 & 4.89324773e-03 & 2.80122249e-03 & 2.10296918e-03\\
6 & 4 & 7.09671322e-02 & 1.22880935e-02 & 8.91873609e-03 & 4.68287840e-03 & 3.19881083e-03 & 2.50645650e-03\\
6 & 5 & 7.09671322e-02 & 1.23862489e-02 & 8.53842453e-03 & 5.33661343e-03 & 3.55154114e-03 & 2.28605325e-03\\
6 & 6 & 7.09671322e-02 & 1.24977263e-02 & 8.40455003e-03 & 5.93603968e-03 & 3.81558500e-03 & 2.25377574e-03\\
\hline
\end{tabular}
\caption{ {\footnotesize 
We give a rigorous upper bound of $c_R(\delta)$, given by Equation~\eqref{eq:russmann:mejor}
for  several $1$-dimensional tight interval frequencies
frequencies $\iomega$ enclosing intervals given
by Equation~\eqref{eq:omega:a,b}.
We use the values $(\gamma,\tau)$ provided in Table~\ref{tab:apenA:1d}. The column
$L=0$ corresponds to the classic R\"ussmann constant in Equation~\eqref{eq:classic:russmann}.
}} \label{tab:russmann:1d}}
\end{table}

\begin{table}[!h]
\centering
{\scriptsize
\begin{tabular}{|c c| c | c c c c c |}
\hline
\multicolumn{2}{|c|}{} &
\multicolumn{6}{c|}{$c_R(\delta) \leq$} \\
\hline
$p$ & $q$ & $L=0$ & 
$\delta=0.1$ &
$\delta=0.05$ &
$\delta=0.01$ &
$\delta=0.005$ &
$\delta=0.001$ \\
\hline \hline
$2$ & $3$  & 3.62859961e-02 & 3.10060284e-03 & 6.00977402e-03 & 1.50395887e-03 & 1.23391239e-03 & 8.15687337e-04\\
$2$ & $5$  & 4.24412098e-02 & 6.88136970e-04 & 8.38062576e-04 & 6.85210209e-03 & 3.48793730e-03 & 1.39435290e-03\\
$2$ & $7$  & 4.08296165e-02 & 1.51913516e-03 & 1.64056165e-03 & 5.96490692e-03 & 2.86266455e-03 & 1.08117640e-03\\
$2$ & $11$ & 3.68950309e-02 & 3.40361447e-03 & 5.37454913e-03 & 2.42226866e-03 & 1.16692138e-03 & 5.72003609e-04\\
$2$ & $13$ & 3.81857929e-02 & 3.34411162e-03 & 1.88186836e-03 & 5.70267600e-03 & 6.35005202e-03 & 2.59988091e-03\\
$3$ & $5$  & 3.58927516e-02 & 6.12958788e-03 & 3.95221883e-03 & 2.42682994e-03 & 1.30875574e-03 & 1.17588322e-03\\
$3$ & $7$  & 3.42417879e-02 & 6.62761410e-03 & 4.15150247e-03 & 2.08654672e-03 & 1.29692066e-03 & 1.82961175e-03\\
$3$ & $11$ & 3.84108018e-02 & 5.79145725e-03 & 5.55693971e-03 & 8.86690139e-04 & 7.03975422e-04 & 1.57560298e-03\\
$3$ & $13$ & 3.49523429e-02 & 5.16292254e-03 & 6.09178577e-03 & 2.29447926e-03 & 1.84095693e-03 & 1.36098576e-03\\
$5$ & $7$  & 4.27225331e-02 & 6.29450376e-03 & 6.95330075e-03 & 8.63897764e-04 & 3.99507510e-04 & 1.95384727e-03\\
$5$ & $11$ & 3.73142594e-02 & 2.16195541e-03 & 3.88935752e-03 & 1.44076003e-03 & 7.74345647e-04 & 6.57036231e-03\\
$5$ & $13$ & 3.68950309e-02 & 2.73635212e-03 & 4.34566648e-03 & 2.63259686e-03 & 1.00779405e-03 & 1.73212787e-03\\
$7$ & $11$ & 4.16193859e-02 & 2.03472877e-03 & 6.68639919e-03 & 2.04760389e-03 & 7.38986151e-04 & 2.71601280e-04\\
$7$ & $13$ & 3.56998570e-02 & 4.23255275e-03 & 4.16600489e-03 & 2.44375364e-03 & 1.91977769e-03 & 4.23864594e-03\\
$11$& $13$ & 4.38862946e-02 & 2.62790710e-03 & 7.60559938e-03 & 2.07028161e-03 & 6.59592023e-04 & 2.09531497e-04\\
\hline
\end{tabular}
\caption{ {\footnotesize 
We give a rigorous upper bound of $c_R(\delta)$, given by Equation~\eqref{eq:russmann:mejor}
for several $2$-dimensional (tight) interval frequencies
$\iomega$ enclosing intervals given by Equation~\eqref{eq:omegap}.
We use the values $(\gamma,\tau)$ provided in Table~\ref{tab:apenA:2d}. The column
$L=0$ corresponds to the classic R\"ussmann constant in Equation~\eqref{eq:classic:russmann}.
}} \label{tab:russmann:2d}}
\end{table}

\section{Validation algorithm to apply the KAM theorem}\label{sec:validation:algorithm}

In this section we present a methodology to perform computer assisted
validations of Lagrangian quasi-periodic invariant tori in exact symplectic
maps.  Given an approximately $F$-invariant torus (e.g. obtained numerically)
we have to rigorously bound the error of such an approximation, and to
rigorously verify the \emph{a priori} hypotheses of Theorem~\ref{theo:KAM}.
One of the applications of the proposed methodology falls into one of the main
strains of the field of {\em validated numerics}
and, for this reason, we use
the term \emph{validation algorithm} in what follows.

The computational cost of the proposed methodology is of order $O(\Ntot \log
\Ntot)$, where $\Ntot$ is the total number of Fourier coefficients what we use.
This is because we use 
fast Fourier transform to manipulate approximations of
periodic functions 
and then we use the error estimates discussed in
Section~\ref{sec:DFT}.  Therefore, we do not require to perform any symbolic
manipulation of Fourier series.  An important aspect of our approach is that
the asymptotic cost of the validation coincides with the asymptotic cost of
computing invariant tori using the parameterization method
(see~\cite{CallejaL10,FoxM14,HaroCFLM,HuguetLS12}).

\subsection{A validation algorithm}
\label{ssec:val:algo}

In this section we propose an algorithm to apply
Theorem~\ref{theo:KAM} in particular problems. The algorithm is stated at a formal level but, 
when implementing it using a computer,
operations must be performed using interval arithmetics.
In Algorithm~\ref{algo:validation} we overview the procedure.
A detailed discussion of each step is given in separate subsections.

Let us first present some useful notation.
We recall that $F: \A \rightarrow \A$ is homotopic to the
identity and $K: \TT^n \rightarrow \A$ is homotopic to the zero section.  For
this reason, it is interesting to introduce the notation
\[
F(x,y)=(x,0)+F_p(x,y),
\]
and
\[
K(\theta)=(\theta,0)+K_p(\theta).
\]

To handle periodic functions, we approximate them using discrete Fourier
transform and we control such approximation using the estimates presented in
Section~\ref{sec:DFT}.  Given a periodic function $f: \TT^n \rightarrow \CC$
and a regular grid of size $\NF=(\Ni{1},\ldots,\Ni{n}) \in \NN^n$, satisfying
that $\Ni{\ell}= 2^{q_\ell}$, with $q_\ell \in \NN$, for $\ell=1,\ldots,n$, we
consider a sample of points
$\theta_j \in \TT^n$, given by~\eqref{eq:sample:torus},
that defines
an $n$-dimensional array $\{f_j\}$, $f_j=f(\theta_j)$,
where $j=(j_1,\ldots,j_n)$, with $0\leq j_\ell < \Ni{\ell}$ and $1\leq \ell \leq n$.

We recall that the total number of points is given by $\Ntot = \Ni{1}\cdots \Ni{n}$
and that we denote by $\tilde f$ the discrete Fourier approximation given
by Equation~\eqref{eq:four:approx}. 
For every sample $\{f_j\}$ we introduce the forward discrete Fourier transform
\begin{equation}\label{eq:DFTformulaF}
\{\tilde f_k\}= \FDFT(\{f_j\}), \quad \mbox{with} \quad \tilde f_k=
\frac{1}{\Ntot} \sum_{0\leq j < \NF} f_j \mathrm{e}^{-2\pi \mathrm{i} k \cdot
\theta_j},
\end{equation}
with $k=(k_1,\ldots,k_n)$ and
$-\tfrac{\Ni{\ell}}{2} \leq k_\ell < \tfrac{\Ni{\ell}}2$. Similarly, we introduce the backward discrete Fourier
transform
\begin{equation}\label{eq:DFTformulaB}
\{f_j\}= \BDFT(\{\tilde f_k\}) , \quad \mbox{with} \quad f_j=
\sum_{-\frac{\NF}{2} \leq  k < \frac{\NF}{2}} \tilde f_k \mathrm{e}^{2\pi
\mathrm{i} k \cdot \theta_j}.
\end{equation}

Notice that Equations~\eqref{eq:DFTformulaF} and~\eqref{eq:DFTformulaB} are just
formal definitions. When implementing the validation algorithm we will resort
to FFT to evaluate these expressions. Such FFT algorithms carry operations in a
particular ordering, which may not be optimal in terms of rounding operations
(when using interval arithmetics), but allows performing fast computations.
From now on we use the notation $\FFFT$ and $\BFFT$ to refer to
Equations~\eqref{eq:DFTformulaF} and~\eqref{eq:DFTformulaB} evaluated
according to the FFT algorithms.

Finally, we give some details on the manipulation of functions discretized in
the Fourier space.  Given a periodic function $f$, discretized as $\{\tilde
f_k\}$, we compute the Fourier discretization of a partial derivative
$\partial_{\theta_l}f$ as $\{(\widetilde{\partial_{\theta_l}
f})_k\}=\{2\pi k_l i \tilde f_k\}$, which corresponds to a diagonal operator
in Fourier space.  Similarly, given an interval vector $\iomega$, the
composition $f \comp \Shift{\iomega}$ is approximated as $\{ (\widetilde {f
\comp \Shift{\iomega}})_k \} = \{\ee^{2\pi \ii k \cdot \iomega} \tilde f_k \}$.

\begin{algorithm}\label{algo:validation}
Given an annulus $\A$ endowed with an exact symplectic structure $\sform = \dif
\aform$ represented by $\Omega(z)$ and $a(z)$, let us
consider an exact symplectic map $F : \A \rightarrow \A$ and the following
input:
\begin{itemize}
\item \textbf{Input 1}. Sampling of an approximately
invariant torus:
a sampling $\{K_{p,j}\}$, with $K_{p,j}=K_0(\theta_j)-(\theta_j,0)$
on a regular grid on $\TT^n$ of size
$\NF=(\Ni{1},\ldots,\Ni{n})$, with $\Ni{\ell}= 2^{q_\ell}$ and $q_\ell \in
\NN$, for $\ell=1,\ldots,n$.
\item \textbf{Input 2}. Transversal vectors:
a map $N_0 : \TT^n \rightarrow \RR^{2 n \times n}$.
\item \textbf{Input 3}. Frequency vector:
a pair of constants $\tau > 0$ and $\gamma >0$, and an interval vector
$\iomega=(\iomega_1,\ldots,\iomega_n) \in \IR^n$ containing a
$(\gamma,\tau)$-Diophantine vector of frequencies.
\item \textbf{Input 4}. KAM parameters:
constants $\delta,\sigma,\rho, \hat \rho, \rho_\infty$, and $d_\B$ satisfying
$0<\delta<\rho/3$, $0<\rho<\hat \rho$, $1<\sigma$, $0< \rho_\infty <
\rho$, and $d_\B > 0$.
\end{itemize}
Then, proceed 
as follows:
\begin{itemize}
\item \textbf{Step 0}. Parameterization of the approximately invariant torus:
compute the parameterization $K(\theta)=(\theta,0)+\tilde K_p(\theta)$, and
characterize the global complex domain $\B$.  Details are provided in
Section~\ref{ssec:details:step0}.

\item \textbf{Step 1}. Error of invariance:
compute a constant $\sub{b}{E}$ such that $\snorm{E}_\rho \leq \sub{b}{E}$.  Details are
provided in Section~\ref{ssec:details:step1}.

\item \textbf{Step 2}. Symplectic frame:
compute constants $\sub{b}{\Dif K}$, $\sub{b}{\Dif K^\top}$, $\sub{b}{B}$, $\sub{b}{A}$, $\sub{b}{N}$, and $\sub{b}{N^\top}$ such
that 
$\snorm{\Dif K}_\rho \leq \sub{b}{\Dif K}$,
$\snorm{\Dif K^\top}_\rho \leq \sub{b}{\Dif K^\top}$,
$\snorm{B}_\rho \leq \sub{b}{B}$,
$\snorm{A}_\rho \leq \sub{b}{A}$,
$\snorm{N}_\rho \leq \sub{b}{N}$, and
$\snorm{N^\top}_\rho \leq \sub{b}{N^\top}$.
Details are provided in Section~\ref{ssec:details:step2}.

\item \textbf{Step 3}. Torsion matrix:
compute the discrete Fourier approximation $\tilde T$ of $T$ and
compute a constant $\sub{b}{T}$ such that $|\avg{T}| \leq \sub{b}{T}$.
Details are
provided in Section~\ref{ssec:details:step3}.

\item \textbf{Step 4}. Hypotheses of the theorem:
introduce the constants
\begin{equation}\label{eq:const:Step4}
\sub{\sigma}{\Dif K}=\sub{b}{\Dif K} \, \sigma, \qquad
\sub{\sigma}{\Dif K^\top}=\sub{b}{\Dif K^\top} \,\sigma, \qquad
\sub{\sigma}{B}=\sub{b}{B} \, \sigma, \qquad
\sub{\sigma}{T}=\sub{b}{T} \, \sigma,
\end{equation}
so that (using $\sigma>1$) Hypotheses $H_1$, $H_2$,
$H_3$, $H_4$, and $H_5$ in Theorem~\ref{theo:KAM} are satisfied.
Then, compute constants $\mathfrak{C}_1$ and $\mathfrak{C}_2$.
Details are
provided in Section~\ref{ssec:details:step4}.

\end{itemize}
If the condition
\[
\frac{\mathfrak{C}_1 \sub{b}{E}}{\gamma^{4} \rho^{4 \tau}} < 1
\]
holds, then, for every $(\gamma,\tau)$-Diophantine frequency $\omega \in
\iomega$ there exists an $F$-invariant torus
$K_{\infty,\omega} (\TT^n)$, analytic in the strip of
width $\rho_{\infty}=\rho/a_2$. Moreover, these invariant tori satisfy
\[
\snorm{\Dif K_{\infty,\omega}}_{\rho_\infty} < \sub{\sigma}{\Dif K}, \qquad
\snorm{\Dif K^\top_{\infty,\omega}}_{\rho_\infty} < \sub{\sigma}{\Dif K^\top}, 
\qquad
\snorm{K_{\infty,\omega}-K}_{\rho_\infty} < 
\frac{\mathfrak{C}_2 \sub{b}{E}}{\gamma^{2} \rho^{2\tau}}.
\]
\end{algorithm}

\begin{remark}\label{rem:pbound}
Incidentally, the above process gives a lower bound of the relative measure of the vectors $\omega$ in $\iomega$
for which the KAM theorem applies.
Assume that the pair $(\gamma,\tau)$ has been assigned to the interval vector $\iomega$ 
by the algorithm derived from Proposition~\ref{cor:Dioph} (using a given value of $M$).
Notice that for $\gamma_0< \gamma$ the constant $\mathfrak{C}_1= \mathfrak{C}_1(\gamma)$ is an upper bound of $\mathfrak{C}_1(\gamma_0)$.
Then, if we take $\gamma_0$ such that  
\[
\frac{\mathfrak{C}_1(\gamma) \sub{b}{E}}{\gamma_0^{4} \rho^{4 \tau}} = 1, 
\]
it turns out that the theorem also applies for $(\gamma_0,\tau)$-Diophantine vectors. A straightforward computation gives that 
\begin{equation}\label{eq:pbound}
	p(\iomega,\gamma_0,\tau)  >   1- \frac{C(\iomega,n) \gamma_0}{(\tau-n) M^{\tau-n}} = 
	1 -  \frac{C(\iomega,n) \sqrt[4]{\mathfrak{C}_1(\gamma) \sub{b}{E}}}{(\tau-n) M^{\tau-n} \rho^\tau}. 
\end{equation}
\end{remark}

\subsection{Implementation details of Step 0}\label{ssec:details:step0}

This preliminary step is performed in order to simplify the implementation
of the algorithm. We observe that Theorem~\ref{theo:KAM} deals with real
analytic objects with real analytic derivatives. Notice that this property is not preserved
by DFT. A simple way to avoid this problem is to consider the validation of a
suitable parameterization. To this end, we perform
the following computations:
\begin{itemize}
\item We compute $\tilde K_p = \{\tilde K_{p,k}\} = \FFFT (\{K_{p,j}\})$, where  $\{K_{p,j}\}$ is the sample given in Input~1.
\item We set to zero the coefficients $\tilde K_{p,k}$ 
with $k=(k_1,\ldots,k_n)$ such that 
$-\Ni{\ell}/2 \leq k_\ell \leq -\Ni{\ell}/4$ or $\Ni{\ell}/4 \leq k_\ell
<\Ni{\ell}/2$
for some index $k_\ell$.
\item
We set the parameterization $K(\theta)=(\theta,0)+\tilde K_p(\theta)$ and
redefine the sampling $\{K_{p,j}\} = \BFFT(\{\tilde K_{p,k}\})$, replacing the
original one.
\end{itemize}

\begin{remark}
An alternative formulation of Algorithm~\ref{algo:validation} would consist in
considering the constructed discrete Fourier approximation $\tilde K=\{\tilde
K_{p,k}\}$ in Input~1.
\end{remark}

Now we introduce the domain $\B$ given by
\begin{equation}\label{eq:domainB}
\B = \{ (x,y) \in \CC^n/\ZZ^n \times \CC^n \,:\, 
|{\im\,x_i}|\leq d_{\B}+\rho+\snorm{K^{x_i}_p}_{F,\rho} \, , \,
|y_i|\leq d_{\B}+\snorm{K^{y_i}_p}_{F,\rho} \},
\end{equation}
and the domain $\hat \B$ given by
\[
\hat \B = \{ (x,y) \in \CC^n/\ZZ^n \times \CC^n \,:\, 
|{\im\,x_i}|\leq \hat\rho+\snorm{K^{x_i}_p}_{F,\hat \rho} \, , \,
|y_i|\leq \snorm{K^{y_i}_p}_{F,\hat \rho} \}.
\]
Finally, we compute the following upper estimates 
$\snorm{\Dif a}_{\B} \leq \sub{c}{\Dif a}$,
$\snorm{\Dif^2a}_{\B} \leq \sub{c}{\Dif^2 a}$,
$\snorm{\Omega}_{\B} \leq \sub{c}{\Omega}$,
$\snorm{\Dif \Omega}_{\B} \leq \sub{c}{\Dif \Omega}$,
$\snorm{\Dif F}_{\B} \leq \sub{c}{\Dif F}$, and
$\snorm{\Dif^2 F}_{\B} \leq \sub{c}{\Dif^2 F}$,
which appear in the statement
of Theorem~\ref{theo:KAM}. We also compute
the upper estimates
$\snorm{\Omega}_{\B} \leq \sub{\hat c}{\Omega}$,
and
$\snorm{F_p}_{\hat \B} \leq \sub{\hat c}{F_p}$.

\subsection{Implementation details of Step 1}\label{ssec:details:step1}

To evaluate the error $E(\theta)$ we use the 
expression 
\begin{equation}\label{eq:error:real}
E(\theta)=
\begin{pmatrix}
K_p^x(\theta)+F_p^x(K(\theta)) - K_p^x(\theta+\iomega) - \iomega\\
\phantom{K_p^x(\theta)+} F_p^y(K(\theta)) - K_p^y(\theta+\iomega)
\phantom{-\iomega}
\end{pmatrix}.
\end{equation}
To evaluate this formula at the grid points, we first compute $F_p \comp K$ and
$K_p \comp \Shift{\iomega}$: the first term is computed directly from the grid,
thus obtaining
\[
\{(F_p\comp K)_j\}
=\{F_p(\theta_j+K_{p,j}^x,K_{p,j}^y)\},
\]
and the second term is computed in Fourier space, 
thus obtaining $\{\widetilde{(K_p \comp
\Shift{\iomega})_k} \}$. Then we compute
\[
\{(K_p \comp \Shift{\iomega})_j \}
=\BFFT(\{\widetilde{(K_p \comp \Shift{\iomega})_k} \}).
\]
From these expressions, the computation of Equation~\eqref{eq:error:real} at the grid,
$\{E_j\}$, is straightforward.
Then, we compute $\{\tilde E_k\}= \FFFT (\{E_j\})$.

Finally, using component-wise Theorem~\ref{ADFT}, we have
\begin{align*}
\snorm{\tilde E-E}_\rho = {} & \max \{\snorm{\tilde E^x-E^x}_\rho,
\snorm{\tilde E^y-E^y}_\rho\} \leq C_{\NF}(\rho,\hat \rho) \max \{\snorm{E^x}_{\hat
\rho}, \snorm{E^y}_{\hat \rho}\} \\ \leq {} & C_{\NF}(\rho,\hat \rho) \max \left\{
\snorm{F^x_p \circ K}_{\hat \rho} + 2 \snorm{K^x_p}_{\hat \rho}+|\iomega|,
\snorm{F^y_p \circ K}_{\hat \rho} + \snorm{K^y_p}_{\hat \rho} \right\}, \\ \leq
{} & C_{\NF}(\rho,\hat \rho)  \max \left\{ \sub{\hat c}{F_p} + 2 \snorm{K^x_p}_{F,\hat
\rho}+|\iomega|, \sub{\hat c}{F_p} + \snorm{K^y_p}_{F,\hat \rho} \right\},
\end{align*}
thus obtaining
\[
\snorm{E}_\rho \leq \snorm{\tilde E}_{F,\rho}+C_{\NF}(\rho,\hat \rho)  \max \left\{
\sub{\hat c}{F_p} + 2 \snorm{K^x_p}_{F,\hat \rho}+|\iomega|, \sub{\hat c}{F_p} +
\snorm{K^y_p}_{F,\hat \rho} \right\} =: \sub{b}{E}.
\]

\subsection{Implementation details of Step 2}\label{ssec:details:step2}

To construct the frame $P(\theta)$ we represent the tangent vectors $\Dif K(\theta)$
as \label{eq:DK:F}
\[
\{\widetilde{\Dif K}_k\}= \left\{
\begin{pmatrix}
I_n \\
O_n
\end{pmatrix} \delta_{k,0} + \widetilde{\Dif K}_{p,k}
\right\},
\qquad
\{\Dif K_j\}=\left\{
\begin{pmatrix}
I_n \\
O_n
\end{pmatrix} 
+\Dif K_{p,j}\right\},
\]
where $\delta_{k,0}$ is Kronecker's delta, and the computation of
$\{\widetilde{\Dif K}_{p,k} \}$ is performed in Fourier space.
Finally compute $\{\Dif K_{p,j}\} = \BFFT(\{\widetilde{\Dif K}_{p,k} \})$.

We compute an approximation of the matrix
$G(\theta)=-\Dif K(\theta)^\top \Omega(K(\theta)) N_0(\theta)$
in a grid as
\[
\{G_j\}=\{-\Dif K_j^\top \Omega_j N_{0,j}\},
\]
where
$\Omega_j=\Omega((\theta_j,0)+K_{p,j})$ and $N_{0,j} = N_0(\theta_j)$.
We complement $\Dif K(\theta)$ by computing $N(\theta)$ as
\[
\{N_j\}=\{\Dif K_j A_j + N_j^0 B_j\},
\]
where
\begin{equation}\label{eq:B:real}
\{B_j\}=\{G_j^{-1}\} \quad \mbox{and} \quad
\{A_j\}= \{-\frac{1}{2} (B_j^\top {N_{0,j}}^\top
\Omega_j N_{0,j} B_j)\}.
\end{equation}
Then, we obtain $\{\tilde N_k\}$ using ${\rm DFT}$, thus ending up with
\[
\{\tilde P_k\}=  
\{\big(\widetilde{\Dif K}_k \quad \tilde N_k\big)\},
\qquad
\{P_j\}=
\{\big(\Dif K_j \quad N_j \big)\}.
\]

Let us observe that, since the parameterization is a truncated series,
we have $\Dif K(\theta)=\widetilde{\Dif K} (\theta)$ and we set
\[
\snorm{\Dif K}_\rho \leq \snorm{\widetilde{\Dif K}}_{F,\rho} =: \sub{b}{\Dif K}, \qquad
\snorm{\Dif K^\top}_\rho \leq \snorm{\widetilde{\Dif K}^\top}_{F,\rho} =: \sub{b}{\Dif K^\top}.
\]
We also need to control
$\snorm{N_0}_\rho \leq \sub{c}{N_0}$,
$\snorm{N_0^\top}_\rho \leq \sub{c}{N_0^\top}$,
$\snorm{N_0^\top (\Omega \circ K) N_0}_\rho \leq \sub{c}{N_0^\top (\Omega \circ K) N_0}$,
$\snorm{N_0}_{\hat \rho} \leq \sub{\hat c}{N_0}$, and
$\snorm{N_0^\top}_{\hat \rho} \leq \sub{\hat c}{N_0^\top}$.
Since the selection of $N_0$ depends on the particular application at hand, we do not
give here explicit details for the estimation of these objects. They follow using the
same ideas that we discuss next to control the error of discrete Fourier
approximations of the remaining objects.
To use Corollary~\ref{cor:matrix:multi}
we control the norm of $I_n-G(\theta) \tilde B(\theta)$
as follows
\begin{align*}
\snorm{I_n-G \tilde B}_\rho \leq {} & C_{\NF}(\rho,\hat\rho) \snorm{G}_{\hat \rho} \snorm{\tilde B}_{\hat\rho} 
\leq
C_{\NF}(\rho,\hat \rho) \snorm{\widetilde{\Dif K}^\top}_{\hat \rho} \snorm{\Omega}_{\hat \B} \snorm{N_0}_{\hat \rho} \snorm{\tilde B}_{\hat \rho} \\
\leq {} &
 C_{\NF}(\rho,\hat\rho) \sub{\hat c}{\Omega} \sub{\hat c}{N_0} \snorm{\widetilde{\Dif K}^\top}_{F,\hat \rho} \snorm{\tilde B}_{\hat\rho} =: t_B.
\end{align*}
Then, if $t_B<1$
we use Corollary~\ref{cor:matrix:inv}, thus obtaining
\begin{equation}
\label{eq:tB}
\snorm{B}_\rho \leq \snorm{\tilde B}_{F,\rho} + \frac{t_B \snorm{\tilde B}_{F,\hat\rho}}{1-t_B}
=:
\sub{b}{B}.
\end{equation}

Finally, we obtain direct estimates for $\snorm{A}_\rho$ and $\snorm{N}_\rho$. On the one hand, using
Equation~\eqref{eq:tB}, we have
\begin{equation}\label{eq:tA}
\snorm{A}_\rho \leq \frac{1}{2} \snorm{B^\top}_\rho \snorm{N_0^\top (\Omega \circ K) N_0}_\rho 
\snorm{B}_\rho
\leq \frac{n}{2} \sub{c}{N_0^\top (\Omega \circ K) N_0} (\sub{b}{B})^2 =:\sub{b}{A},
\end{equation}
and on the other hand, using Equations~\eqref{eq:tB} and~\eqref{eq:tA}, we have
\begin{align}
& \snorm{N}_\rho \leq \snorm{\Dif K}_\rho \snorm{A}_\rho + \snorm{N_0}_\rho \snorm{B}_\rho \leq \sub{b}{\Dif K} \sub{b}{A}+ \sub{c}{N_0} \sub{b}{B}=:\sub{b}{N}, \label{eq:tN} \\
& \snorm{N^\top}_\rho \leq \snorm{A}_\rho \snorm{\Dif K^\top}_\rho + \snorm{B^\top}_\rho \snorm{N_0^\top}_\rho
\leq \sub{b}{A} \sub{b}{\Dif K^\top} + n \sub{b}{B} \sub{c}{N_0^\top}=:\sub{b}{N^\top} . \label{eq:tNt}
\end{align}

\subsection{Implementation details of Step 3}\label{ssec:details:step3}

To compute the torsion matrix $T(\theta)$, we first obtain the
shifted normal frame $\{ (\tilde N \circ \Shift{\iomega} )_k\}$
and $\{ (\tilde N \circ \Shift{\iomega} )_j \} = \BFFT(\{ (\tilde N \circ \Shift{\iomega} )_k\})$.
Then, we compute
\[
\{T_j\} = \{ (N \comp \Shift{\iomega})_j^\top (\Omega \comp K \comp \Shift{\iomega})_j (\Dif F \comp K)_j N_j \},
\]
and approximate the average $\avg{T}$ as follows
\[
\tilde T_0 = \frac{1}{\Ntot} \sum_j T_j.
\]
Then we compute
\[
|\tilde T_0 - \avg{T}| 
\leq s_{\NF}^*(0,\rho) \snorm{T}_\rho 
\leq s_{\NF}^*(0,\rho) \snorm{N^\top}_\rho \snorm{\Omega}_{\B} \snorm{\Dif F}_{\B} \snorm{N}_\rho
\leq s_{\NF}^*(0,\rho) \sub{c}{\Omega} \sub{c}{\Dif F} \sub{b}{N} \sub{b}{N^\top} := t_T
\]
using the upper estimates in~\eqref{eq:tN} and~\eqref{eq:tNt} for $ \snorm{N}_\rho$ and $\snorm{N^\top}_\rho$, respectively.
Finally, we check that
\[
|\tilde T_0^{-1}| t_T< 1,
\]
and we obtain (using a Neumann series argument)
\[
|\avg{T}^{-1}| \leq \frac{|\tilde T_0^{-1}|}{1- |\tilde T_0^{-1}||\tilde T_0 - \avg{T}| } \leq
\frac{|\tilde T_0^{-1}|}{1 - |\tilde T_0^{-1}|t_T } =: \sub{b}{T}.
\]

\subsection{Implementation details of Step 4}\label{ssec:details:step4}

We notice that the choice in~\eqref{eq:const:Step4} introduces some suitable
simplifications in the expression of Constants
$\mathfrak{C}_1$ and $\mathfrak{C}_2$.  Indeed, after simple manipulations of
the expressions described in Section~\ref{ssec:proof:KAM} and
using the Equations~\eqref{eq:const:Step4} and~\eqref{eq:domainB}, we obtain
\[
\mathfrak{C}_1 = \max \left\{ 2(a_3)^{\tau+1} \gamma^3 \rho^{3 \tau-1} C_3, 
\mathfrak{C}_{3}, 
\mathfrak{C}_{4}, 
\mathfrak{C}_{5}
\right\},
\qquad
\mathfrak{C}_2:= \frac{a_3^{2 \tau} \hat C_2}{1-a_1^{1-2\tau}},
\]
where
\begin{equation}\label{eq:C3C4C5}
\mathfrak{C}_{3}= (a_1 a_3)^{4 \tau} \hat C_5, \qquad
\mathfrak{C}_{4}
= \frac{\sigma_* (a_3)^{2\tau +1} \gamma^2 \rho^{2\tau -1} \hat C_2}{(\sigma-1) (1-a_1^{1-2\tau})}, \qquad
\mathfrak{C}_{5}
= \frac{(a_3)^{2\tau} \gamma^2 \rho^{2\tau} \hat C_2}{d_\B (1-a_1^{-2\tau})}.
\end{equation}
Finally, we provide expressions for the constants $\sigma_*$, $\hat C_2$, and $\hat C_5$ in terms
of the initial data. The first one is given by
\begin{equation}
\label{def: sigma*}
\sigma_* = \max \left\{\frac{n}{\sub{b}{\Dif K}},\frac{2n}{\sub{b}{\Dif K^\top}},\frac{\beta_1}{\sub{b}{B}},\frac{\beta_5}{\sub{b}{T}}\right\}  
\end{equation}
where
\begin{align*}
\beta_1 = {} &  2 \sub{\sigma}{B}^2 \sub{c}{N_0} \Big(\sub{\sigma}{\Dif K^\top} \sub{c}{\Dif \Omega} \delta + 2 n \sub{c}{\Omega} \Big), \\
\beta_2 = {} & \tfrac{n}{2} (\sub{\sigma}{B})^2 \sub{c}{\Dif \Omega} \delta + \tfrac{n+1}{2} \sub{c}{N_0^\top (\Omega \circ K) N_0} \beta_1, \\
\beta_3 = {} & 
\sub{\sigma}{\Dif K} \beta_2 + n \sub{c}{A} + c_{N_0} \beta_1, \\
\beta_4 = {} &
\sub{\sigma}{\Dif K^\top} \beta_2 + 2 n \sub{c}{A} + n \sub{c}{N_0^\top} \beta_1, \\
\beta_5 = {} & 
2 (\sub{\sigma}{T})^2 
\Big(\sub{c}{N^\top} \sub{c}{N} (\sub{c}{\Omega} \sub{c}{\Dif^2 F} + \sub{c}{\Dif \Omega} \sub{c}{\Dif F} )\delta
+ \sub{c}{\Omega} \sub{c}{\Dif F} (\sub{c}{N^\top} \beta_3 + \sub{c}{N} \beta_4) \Big)
\end{align*}
and
\[
\sub{c}{A} = \tfrac{1}{2} n \sub{c}{N_0^\top (\Omega \circ K) N_0}  (\sub{\sigma}{B})^2,
\quad
\sub{c}{N} = \sub{\sigma}{\Dif K} \sub{c}{A} + \sub{c}{N_0}\sub{\sigma}{B},
\quad
\sub{c}{N^\top}= \sub{c}{A}\sub{\sigma}{\Dif K^\top} + n \sub{\sigma}{B}\sub{c}{N_0^\top}.
\]
The constants $\hat C_2$ and $\hat C_5$ follow from the next sequence of computations:
{\allowdisplaybreaks
\begin{align*}
\sub{c}{P} = {} & \sub{\sigma}{\Dif K} + \sub{c}{N}, \\
\sub{c}{T} = {} & \sub{c}{N^\top}\sub{c}{\Omega}\sub{c}{\Dif F}\sub{c}{N}, \\
C_1 = {} &  \sub{\sigma}{\Dif K^\top}\sub{\sigma}{\Dif K}\sub{c}{\Dif \Omega}\delta + n\sub{\sigma}{\Dif K^\top}\sub{c}{\Omega} + 2n\sub{c}{\Omega}\sub{c}{\Dif F}\sub{\sigma}{\Dif K}, \\
C_2 = {} & c_R C_1, \\
C_3 = {} & (1+\sub{c}{A})\max (1,\sub{c}{A}) C_2, \\
C_4 = {} & n\sub{c}{N^\top}\sub{c}{\Omega}\gamma\delta^\tau + \sub{c}{A}C_2, \\
C_5 = {} & C_2 + n\sub{\sigma}{\Dif K^\top}\sub{c}{\Omega}\gamma\delta^\tau, \\
C_6 = {} & \sub{c}{A}C_2 + \sub{\sigma}{\Dif K^\top}\sub{c}{\Dif \Omega}\sub{c}{\Dif F}\sub{c}{N}\gamma\delta^{\tau+1} + 2n\sub{c}{\Omega}\sub{c}{\Dif F}\sub{c}{N}\gamma\delta^\tau, \\
C_7 = {} & \max (C_4,C_5+C_6), \\
C_8 = {} & 2c_R\sub{\sigma}{\Dif K^\top}\sub{c}{\Omega}, \\
C_9 = {} & C_8 + \sub{\sigma}{T}(\sub{c}{N^\top}\sub{c}{\Omega}\gamma\delta^\tau + \sub{c}{T}C_8), \\
C_{10}= {} & c_R(\sub{c}{N^\top}\sub{c}{\Omega}\gamma\delta^\tau + \sub{c}{T}C_9), \\
\hat C_2= {} & \sub{\sigma}{\Dif K}C_{10} + \sub{c}{N}C_9\gamma\delta^\tau, \\
C_{15}= {} & (C_3+C_7)\max (C_9\gamma\delta^\tau,C_{10}) + 2n\sub{c}{\Dif a} \gamma^3 \delta^{3 \tau} + \tfrac{1}{2}\sub{c}{\Dif^2 a} \gamma^3 \delta^{3 \tau+1}, \\
\hat C_5= {} & 2\sub{c}{P}C_{15}\gamma\delta^{\tau-1}+ \tfrac{1}{2}\sub{c}{\Dif^2 F} \hat C_2^2,
\end{align*}}

\vspace{-0.5cm}
\noindent where the value $c_R$ is computed using Equation~\eqref{eq:russmann:mejor}.
We recall that if we take the value $L=0$ in~\eqref{eq:russmann:mejor} we
obtain the classic expression in~\eqref{eq:classic:russmann}.
\begin{remark}\label{rem:details:improvements}
It is worth mentioning that the above expressions are very general. Using
specific information from a particular
problems, it is possible to improve some estimates.  For
example, if $n=1$ then the phase space is $2$-dimensional and every
$1$-dimensional subspace is Lagrangian. This has the immediate consequence that
we can take $C_1=0$, $c_A=0$ and $\sub{c}{N_0^\top (\Omega \circ K) N_0}=0$,
thus simplifying the computations and reducing the size of the subsequent
constants.  For the same reason, if $n=1$ we do not have to control
$\snorm{\bar A - A}_{\rho-3\delta}$ so we can take $C_{12}=0$, and so,
$\beta_2=0$.  Moreover, we can take advantage of the specific expression of
the map $F$ and the matrix $N_0$ in order to control the norm of the twist
matrix $T$.  More precisely, we can replace the estimate $\snorm{T}_\rho \leq
\sub{c}{T} = \sub{c}{N^\top} \sub{c}{\Omega} \sub{c}{\Dif F} \sub{c}{N}$ by an
\emph{ad hoc} estimation for the considered problem.
\end{remark}

\section{Application of the KAM theorem in some examples}
\label{sec:examples}

In this section we apply the techniques described in the paper to prove
existence of invariant tori in different scenarios.
A common feature in the selected examples is that the
objects $F$, $\Omega$, and $a$ are explicit. 
As a consequence we directly obtain global estimates
for
$\snorm{\Dif F}_\B$,
$\snorm{\Dif^2 F}_\B$,
$\snorm{\Omega}_\B$,
$\snorm{\Dif \Omega}_\B$,
$\snorm{\Dif a}_\B$, and
$\snorm{\Dif^2 a}_\B$.
This allows us to focus in the fundamental
steps involved in the computer assisted proof (CAP).
If the symplectic map is defined by means of an
implicit equation or it is given by the discretization of
the flow of a Hamiltonian vector field, then the 
control of global estimates may deserve a particular attention
and can become a very difficult problem. We plan to
approach the study of such problems in a
subsequent research.

In order to avoid repeated explanations in the description of the different
examples, we summarize next some general details regarding the input
of Algorithm~\ref{algo:validation} that we use to perform the CAPs.

\begin{itemize}
\item \textbf{Input 1}.
The numerical values of the sampling $\{K_{p,j}\}$,
with $K_{p,j}=K_0(\theta_j)-(\theta_j,0)$ in a
regular grid of size $\NF=(\Ni{1},\ldots,\Ni{n})$,
are provided by means
of a data file. These numbers are read
rounding to the nearest representable number.
These numbers are used to construct the
parameterization that will be validated (Step 0 of the Algorithm~\ref{algo:validation}).
\item \textbf{Input 2}.
The transversal map $N_0$ is selected
according to the particular problem at hand. For example, if we look for an
invariant curve and we know that it is a graph, then we can introduce
a constant vector complementing the tangent vector of the curve at every point. 
We also can use the geometric structure of the problem to obtain the transversal map.
The evaluation rule for the map $N_0$
is provided by means of a separate
subroutine.
The reader is referred to Chapter 4 of~\cite{HaroCFLM} for a detailed discussion 
that summarizes different approaches in the literature.
\item \textbf{Input 3}.
When the frequency vector is defined by means of an algebraic
equation then $\iomega$ is obtained by enclosing the solution of this equation. When
the frequency vector has been obtained in a numerical computation, then
the enclosing interval vector $\iomega$ is selected according with the
precision of the computation.
If the Diophantine constants of a target frequency are know, they can be provided
(e.g. for the golden mean). Otherwise, we obtain a pair of constants $\tau > n$ and $\gamma >0$ using
the methodology derived from Proposition~\ref{cor:Dioph}.

\item \textbf{Input 4}.
For every parameterization given in Input 1,
we need suitable constants
$\delta,\sigma,\rho, \hat \rho, \rho_\infty$, and $d_\B$ satisfying
$0<\delta<\rho/3$, $0<\rho<\hat \rho$, $1<\sigma$, $0\leq \rho_\infty <
\rho$, and $d_\B > 0$. These constants are obtained from the heuristic
(non-rigorous) methodology described in Appendix~\ref{app:optimization}, and are provided by a data file. 
They are read rounding to the nearest representable number.

\end{itemize}

The sampling $\{K_{p,j}\}$ is obtained numerically
using the implementation of the parameterization method proposed in Chapter 4
of~\cite{HaroCFLM}. Analogous implementations of the method in different contexts have been
previously presented in the literature adapted to several
problems~\cite{CallejaL10,CallejaF,CanadellH14,LlaveL11,FoxM14,HaroL06b,HuguetLS12}.  
We want to remark that the validation algorithm admits complete freedom in using a different
method to obtain the approximate parameterization. Indeed, the numerical
computation of invariant tori has been a fruitful area of research in the last
years and there is a wide set of numerical methods available. The
parameterization method is a suitable choice for several reasons: it has cost
$O(\Ntot)$ in storage and only $O(\Ntot \log(\Ntot))$ in time, where
$\Ntot=\Ni{1}\cdots\Ni{n}$; it does not use a perturbative setting of the
problem so we can consider any invariant torus in phase space, without
performing any perturbative analysis of the parameters of the problem; the
numerical algorithm of the parameterization method has a similar structure than
Algorithm~\ref{algo:validation} so we can take advantage of the same codes
using a suitable overloading of arithmetics.  Finally, we observe that the
bottleneck of Algorithm~\ref{algo:validation} is the error of invariance of the
parameterization defined by the sampling $\{K_{p,j}\}$ (see details in
Section~\ref{ssec:details:step0}). To this end, the fact that the convergence
of the parameterization method is quadratic and that the iterations are fast
allows us to obtain approximations of invariant tori with very high accuracy.

\subsection{Standard map}\label{ssec:standard}

We first consider the study of quasi-periodic invariant
curves of the so-called Chirikov standard map~\cite{Chirikov79}
\begin{equation}\label{eq:standard:map}
\begin{array}{rcl}
F:
\TT \times \RR & 
\longrightarrow &
\TT \times \RR \\
(x,y) & \longmapsto & (\bar x, \bar y)=\left(x+\bar y, y-\frac{\epsilon}{2\pi} \sin (2\pi x)\right)
\end{array}.
\end{equation}
For $\eps=0$ the dynamics is very
simple: the orbit of any point $(x,y)\in \TT\times \RR$ is given by the rigid
rotation $F^n(x,y)=(x+ny,y)$. Note that
if $y =p/q \in \QQ$ the corresponding orbit is periodic, i.e.,
$F^q(x,y)=(x+p,y)=(x,y)$. On the contrary, if $y \in \RR\backslash
\QQ$, the orbit is dense in the invariant curve $\TT \times \{y\}$.
In any case, the orbit of a given point $(x,y)$ has rotation number $\omega=y$
for every $x \in \TT$.

For $\eps > 0$, sufficiently small, {\KAM} theory concludes that ``most'' of the previous
invariant curves persist, although they are slightly deformed.
The deformation preserves the homotopy class of these rotational invariant curves, also called {\em primary tori}.
These curves are successively destroyed as $\eps$ is
increased.
The value of $\eps$
for which an invariant curve is destroyed is called critical value $\eps_c=\eps_c(\omega)$.
A particularly interesting case is the golden rotation $\omega=\tfrac{\sqrt{5}-1}{2}$.

For the standard map, given by Equation~\eqref{eq:standard:map}, we have $\A=\TT \times \RR$,
$\aform=\aform_0= y\dif x$ and $\sform= \sform_0= \dif y \wedge \dif x$.
Hence, we take $\sub{c}{\Omega}=1$, $\sub{c}{\Dif \Omega}=0$, $\sub{c}{\Dif a}=1$ and $\sub{c}{\Dif^2 a}=0$.
We select the transversal vector $N_0 : \TT \rightarrow \RR^{2\times 1}$ as
\[
N_0(\theta)= N_0=
\begin{pmatrix}
0 \\
1
\end{pmatrix}.
\]

In Step 0 of Algorithm~\ref{algo:validation}, we introduce the domain $\B$ as
\[
\B = \{ (x,y) \in \CC/\ZZ \times \CC \,:\, 
|{\im\,x}|\leq r_1 \, , \, |y|\leq r_2 \},
\]
and the domain $\hat \B$ as
\[
\hat \B = \{ (x,y) \in \CC/\ZZ \times \CC \,:\, 
|{\im\,x}|\leq \hat r_1 \, , \, |y|\leq \hat r_2 \},
\]
where
\[
r_1 = d_{\B}+\rho+\snorm{K^x_p}_{F,\rho}, \qquad
r_2 = d_{\B}+\snorm{K^y_p}_{F,\rho}, \qquad
\hat r_1= \hat\rho+\snorm{K^x_{p}}_{F,\hat \rho}, \qquad
\hat r_2= \snorm{K^y_{p}}_{F,\hat \rho}.
\]
Then, global estimates of the symplectic map \eqref{eq:standard:map} in these domains are characterized as follows
\begin{align*}
& \snorm{\Dif F}_\B \leq  \sub{c}{\Dif F} := 2 + \eps \cosh(2\pi r_1), \\
& \snorm{\Dif^2 F}_\B \leq  \sub{c}{\Dif^2 F} := 2\pi \eps \cosh(2\pi r_1), \\
& \snorm{F_p}_{\hat \B} \leq  \sub{\hat c}{F_p} := \hat r_2 + \frac{\eps}{2\pi} \cosh(2\pi \hat r_1).
\end{align*}

Let us characterize the constants that appear in Step 2 
of Algorithm~\ref{algo:validation}. Since $N_0$ is constant, 
we take
$\sub{c}{N_0}=1$,
$\sub{c}{N_0^\top}=1$,
$\sub{\hat c}{N_0}=1$, and
$\sub{\hat c}{N_0^\top}=1$.
Moreover, we observe that
$N_0^\top (\Omega \circ K) N_0$
vanishes identically, so we take
$\sub{c}{N_0^\top (\Omega \circ K) N_0}=0$
(see Remark~\ref{rem:details:improvements}).

In Table~\ref{tab:standard} we show the fundamental information in the
computer assisted proof for the existence of the golden invariant curve for
different values of $\eps$.  The second column is the number of points $\NF$ in
the regular grid where the sampling $\{K_{p,j}\}$ is defined.  As it
is mentioned in the introduction of this section, this sampling is
obtained via the parameterization method
asking for a tolerance of $10^{-33}$ in the error of invariance 
(using the norm $\snorm{E}_{F,0}$).
For each
value of $\eps$ and the corresponding sampling, we take the
values $\rho$, $\delta$, $\sigma$, $d_\B$ and $\hat\rho$ that are given in
columns $3$ to $7$.  These 
values have been obtained using the
heuristic methodology described in Appendix~\ref{app:optimization}. In all
computations we take $a_2=1000$ so that $\rho_\infty = \rho/1000$.  
We use the specific Diophantine constants
$\gamma=\tfrac{3-\sqrt{5}}{2}$ and $\tau=1$ of the golden mean.
After applying the rigorous computations described in
Algorithm~\ref{algo:validation}, 
using 267 bits,
in the last two
columns we provide the values
of the left-hand side of the KAM condition and the bound of the correction of
the true invariant tori, which are given by
\[
\frac{\mathfrak{C}_1 \sub{b}{E}}{\gamma^4 \rho^{4 \tau}}, \qquad 
\snorm{K_{\infty}-K}_{\rho_\infty} < \frac{\mathfrak{C}_2 \sub{b}{E}}{\gamma^2 \rho^{2 \tau}},
\]
respectively.
It is worth mentioning that the results presented in
Table~\ref{tab:standard} are non-perturbative, in the sense that the computations are performed
independently of each value of $\epsilon$, without using any information related to smaller
values of the parameter.
The computational time of the CAP for the case $\eps=0.96$ is 117 seconds
in a single processor Intel(R) Xeon(R) CPU at 2.40 GHz.

\begin{table}[!t]
\centering
{\scriptsize
\begin{tabular}{|c| c c c c c c || c c|}
\hline
$\epsilon$ &
$\NF$ &
$\rho$ &
$\delta$ &
$\sigma-1$ &
$d_{\B}$ &
$\hat \rho$ &
$\frac{\mathfrak{C}_1 \sub{b}{E}}{\gamma^4 \rho^{4 \tau}}$ &
$\frac{\mathfrak{C}_2 \sub{b}{E}}{\gamma^2 \rho^{2 \tau} }$ \\ 
\hline \hline
0.06 & 128 & 1.606160e-02 & 3.212319e-03 & 1.670325e-01 & 5.064098e-06 & 2.569855e-01 & 1.35e-28 & 9.47e-34\\
0.16 & 256 & 1.369960e-02 & 2.739919e-03 & 9.673976e-02 & 2.937365e-06 & 1.369960e-01 & 9.24e-28 & 3.77e-33\\
0.26 & 256 & 1.369960e-02 & 2.739919e-03 & 6.974093e-02 & 2.044422e-06 & 1.301462e-01 & 1.74e-26 & 4.94e-32\\
0.36 & 512 & 1.369960e-02 & 2.739919e-03 & 5.229422e-02 & 1.400906e-06 & 7.534778e-02 & 4.24e-25 & 8.26e-31\\
0.46 & 512 & 1.369960e-02 & 2.739919e-03 & 3.941981e-02 & 9.278480e-07 & 7.534778e-02 & 1.76e-23 & 2.27e-29\\
0.56 & 512 & 4.520867e-03 & 8.908112e-04 & 1.268703e-02 & 9.401294e-08 & 6.329214e-02 & 9.39e-24 & 1.24e-30\\
0.66 & 1024 & 3.300233e-03 & 5.973272e-04 & 1.047736e-02 & 4.061043e-08 & 3.300233e-02 & 1.88e-23 & 1.11e-30\\
0.76 & 1024 & 2.310163e-03 & 4.017675e-04 & 5.924431e-03 & 1.166394e-08 & 3.003212e-02 & 2.32e-18 & 3.98e-26\\
0.86 & 2048 & 1.178183e-03 & 1.996921e-04 & 1.921375e-03 & 1.234843e-09 & 1.531638e-02 & 1.74e-17 & 3.19e-26\\
0.96 & 32768 & 1.178183e-04 & 1.971855e-05 & 3.648874e-05 & 5.996316e-13 & 1.060365e-03 & 2.34e-12 & 2.09e-24\\
\hline
\end{tabular}
\caption{ {\footnotesize Application of Theorem~\ref{theo:KAM} using
Algorithm~\ref{algo:validation} for the golden invariant curve
of the standard map~\eqref{eq:standard:map}
for different
values of $\eps$.
We use $\gamma=\tfrac{3-\sqrt{5}}{2}$, $\tau=1$ and the \emph{ad hoc} R\"ussmann estimates in~\eqref{eq:russmann:mejor}.
The result (last two columns) is given with 2 significant digits.
}} \label{tab:standard} }
\end{table}

The application of the KAM theorem becomes
computationally more demanding as we approach the critical value $\eps_c$.
Indeed, from $\eps=0.02$ to $\eps=0.96$ the KAM condition has worsened 
by 16 orders
of magnitude. From this point, the number of Fourier coefficients
required to apply the theorem increase dramatically (exponentially).  In this
situation, the R\"ussmann estimates proposed in
Section~\ref{ssec:superrussmann} play a significant role in improving the
applicability of the KAM theorem. For example, if we repeat the
CAP for $\eps = 0.96$
using the classical R\"ussmann estimates in~\eqref{eq:classic:russmann} 
we obtain
\[
\frac{\mathfrak{C}_1 \sub{b}{E}}{\gamma^4 \rho^{4 \tau}} \leq 5.42 \cdot 10^{-6}, \qquad 
\frac{\mathfrak{C}_2 \sub{b}{E}}{\gamma^2 \rho^{2 \tau}} \leq 1.71 \cdot 10^{-21}.
\]
In this case, expression $\mathfrak{C}_1 \sub{b}{E} \gamma^{-4}\rho^{-4 \tau}$ is 6 orders of
greater than the value obtained using the \emph{ad hoc} estimes in~\eqref{eq:russmann:mejor}.
The difference increases when we approach to the critical value. 
The last value of $\eps$ for which we have applied the KAM theorem is
the following:

\begin{theorem}\label{theo:golden}
For $\eps=0.9716$ the standard map has
a rotational invariant curve with golden rotation number.
\end{theorem}

\begin{proof}
We consider a parameterization $K$
obtained using the parameterization method with $\NF=8388608$ Fourier coefficients
(we show the significant ones in Figure~\ref{fig:golden}).
This parameterization
satisfies $\snorm{E}_{F,0} \leq 2.74 \cdot 10^{-41}$.
Again, we
take 
the Diophantine constants
$\gamma=\tfrac{3-\sqrt{5}}{2}$ and $\tau=1$,
and we use
the improved R\"ussmann constant in~\eqref{eq:russmann:mejor}.
Setting the parameters
\begin{align*}
\rho = {} & 3.748290 \cdot 10^{-7}, \\
\delta={} & 6.273289 \cdot 10^{-8}, \\
\sigma={} &1+1.610158 \cdot 10^{-9}, \\
d_\B={}   &3.159428 \cdot 10^{-21}, \\
\hat\rho={} &4.872777 \cdot 10^{-6},
\end{align*}
and applying Algorithm~\ref{algo:validation} with precision of 367 bits
we obtain the following rigorous bound:
\[
\frac{\mathfrak{C}_1 \sub{b}{E}}{\gamma^4 \rho^{4 \tau}} \leq 0.0823.
\]
This allows us to apply the KAM theorem. Moreover, the
golden curve satisfies
\[
\snorm{K_\infty-K}_{\rho_\infty} < \frac{\mathfrak{C}_2 \sub{b}{E}}{\gamma^2 \rho^{2 \tau}} \leq 3.89 \cdot 10^{-22}.
\]
The computational time of this CAP is 11404 seconds
in a single processor Intel(R) Core(R) CPU at 3.50 GHz. It required the use of almost 32 GB of RAM.
\end{proof}

\begin{figure}[!t]
\centering
\includegraphics[scale=0.35]{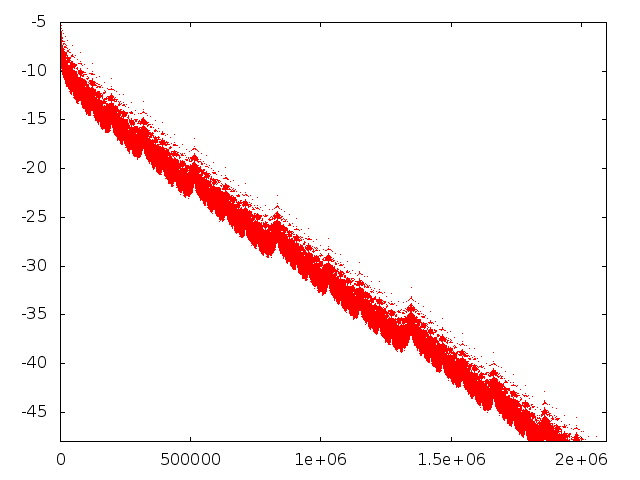}
\includegraphics[scale=0.35]{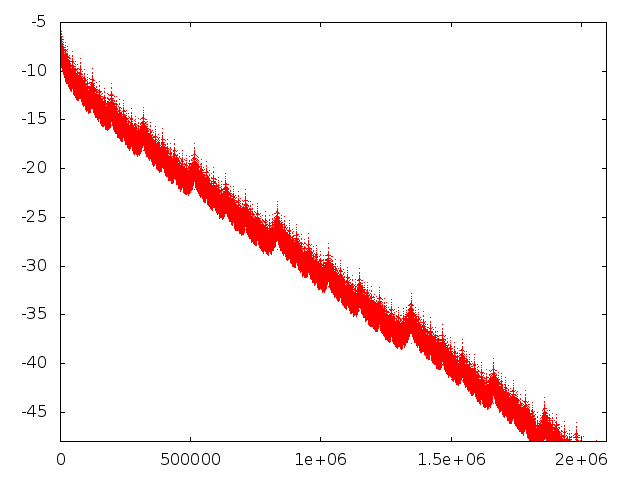}
\caption{{\footnotesize
Fourier coefficients (for positive $k$) of the
validated parameterization of the golden invariant curve 
of the standard map~\eqref{eq:standard:map}
for $\eps=0.9716$.
In the left plot we show $k \mapsto \log_{10}(|K_{p,k}^x|)$ and
in the right plot we show $k \mapsto \log_{10}(|K_{p,k}^y|)$.}} \label{fig:golden}
\end{figure}

Now we illustrate the methodology proposed in Section~\ref{ssec:diophantine}
to obtain a pair of constants $(\gamma,\tau)$ for a given interval vector.
For example, we consider the interval $\iomega$ obtained by
computing $\tfrac{\sqrt{5}-1}{2}$ using MPFI with $267$ bits and taking $M=1000$.
In Table~\ref{tab:standard2} we repeat the CAPs given
in Table~\ref{tab:standard} using the obtained constants
$\gamma=0.38196601125010$ and $\tau=1.26$. Using the argument
in Remark~\ref{rem:pbound} we obtain a rigorous upper bound
of the relative measure of the set of frequencies in $\iomega$ for which the KAM theorem does not apply.

\begin{table}[!t]
\centering
{\scriptsize
\begin{tabular}{|c|| c c||c|}
\hline
$\epsilon$ &
$\frac{\mathfrak{C}_1 \sub{b}{E}}{\gamma^4 \rho^{4 \tau}}$ &
$\frac{\mathfrak{C}_2 \sub{b}{E}}{\gamma^2 \rho^{2 \tau} }$ & $1-p(\iomega,\gamma_0,\tau)$\\ 
\hline \hline
0.06 & 3.50e-28 & 7.56e-34 & 1.33e-07\\
0.16 & 2.40e-27 & 3.01e-33 & 2.16e-07\\
0.26 & 4.51e-26 & 3.94e-32 & 4.50e-07\\
0.36 & 1.10e-24 & 6.60e-31 & 9.99e-07\\
0.46 & 4.56e-23 & 1.81e-29 & 2.53e-06\\
0.56 & 2.39e-23 & 9.83e-31 & 2.16e-06\\
0.66 & 4.25e-23 & 8.64e-31 & 2.49e-06\\
0.76 & 5.00e-18 & 3.08e-26 & 4.61e-05\\
0.86 & 3.64e-17 & 2.46e-26 & 7.58e-05\\
0.96 & 4.84e-12 & 1.61e-24 & 1.45e-03\\
\hline
\end{tabular}
\caption{ {\footnotesize Application of Theorem~\ref{theo:KAM} using
Algorithm~\ref{algo:validation} around the golden invariant curve
of the standard map~\eqref{eq:standard:map}
for different
values of $\eps$.
We use the same implementation parameters as in Table~\ref{tab:standard}.
We take the constants $\gamma=0.38196601125010$, $\tau=1.26$ and
use the \emph{ad hoc} R\"ussmann estimates in~\eqref{eq:russmann:mejor}.
A rigorous upper bound of the relative measure of the set of frequencies in $\iomega$
for which the KAM theorem does not apply (we use expression~\eqref{eq:pbound}).
The results are given with 2 significant digits.
}} \label{tab:standard2} }
\end{table}

\begin{table}[!t]
\centering
{\scriptsize
\begin{tabular}{|c c| c || c c c c c || c c|}
\hline
$a$ & $b$ &
$\epsilon$ &
$\rho$ &
$\delta$ &
$\sigma-1$ &
$d_{\B}$ &
$\hat \rho$ &
$\frac{\mathfrak{C}_1 \sub{b}{E}}{\gamma^4 \rho^{4 \tau}}  $ &
$\frac{\mathfrak{C}_2 \sub{b}{E}}{\gamma^2 \rho^{2 \tau}} $ \\ 
\hline \hline
1 & 2 & 0.87 & 9.226561e-05 & 1.544194e-05 & 5.101229e-06 & 4.377668e-14 & 9.687889e-04 & 1.71e-07 & 9.11e-21\\
1 & 3 & 0.76 & 1.919125e-04 & 3.211924e-05 & 1.149569e-05 & 1.945721e-13 & 2.111037e-03 & 1.57e-11 & 3.85e-24\\
1 & 4 & 0.67 & 2.140562e-04 & 3.582531e-05 & 7.808704e-06 & 1.020386e-13 & 2.140562e-03 & 7.08e-11 & 9.33e-24\\
1 & 5 & 0.60 & 3.321562e-04 & 5.559099e-05 & 5.647400e-06 & 6.322830e-14 & 2.159015e-03 & 3.38e-08 & 2.75e-21\\
1 & 6 & 0.54 & 1.992937e-04 & 3.335460e-05 & 2.297353e-06 & 1.249607e-14 & 2.092584e-03 & 7.49e-09 & 1.20e-22\\
2 & 1 & 0.93 & 1.845312e-04 & 3.088388e-05 & 1.340147e-05 & 3.284838e-13 & 1.107187e-03 & 2.90e-10 & 1.17e-22\\
2 & 2 & 0.95 & 1.033375e-04 & 1.729498e-05 & 7.353570e-06 & 1.004165e-13 & 1.033375e-03 & 1.70e-10 & 2.14e-23\\
2 & 3 & 0.91 & 1.919125e-04 & 3.211924e-05 & 1.297609e-05 & 3.508463e-13 & 1.151475e-03 & 4.39e-10 & 1.97e-22\\
2 & 4 & 0.86 & 1.771500e-04 & 2.964853e-05 & 1.090351e-05 & 2.725768e-13 & 2.037225e-03 & 6.87e-11 & 2.42e-23\\
2 & 5 & 0.82 & 1.107187e-04 & 1.853033e-05 & 2.959595e-06 & 2.577350e-14 & 1.051828e-03 & 1.85e-09 & 6.21e-23\\
2 & 6 & 0.78 & 9.595623e-05 & 1.605962e-05 & 1.243173e-06 & 6.667071e-15 & 1.055519e-03 & 2.51e-08 & 2.19e-22\\
3 & 1 & 0.83 & 2.066750e-04 & 3.458995e-05 & 1.428538e-05 & 3.464827e-13 & 2.066750e-03 & 6.95e-12 & 3.02e-24\\
3 & 2 & 0.89 & 1.144094e-04 & 1.914801e-05 & 8.411203e-06 & 1.146242e-13 & 1.029684e-03 & 1.45e-10 & 2.05e-23\\
3 & 3 & 0.88 & 1.771500e-04 & 2.964853e-05 & 1.476158e-05 & 3.809012e-13 & 1.151475e-03 & 3.09e-11 & 1.47e-23\\
3 & 4 & 0.86 & 1.144094e-04 & 1.914801e-05 & 5.127246e-06 & 6.122894e-14 & 1.029684e-03 & 4.37e-10 & 3.43e-23\\
3 & 5 & 0.83 & 1.771500e-04 & 2.964853e-05 & 1.026994e-05 & 1.831959e-13 & 1.860075e-03 & 1.66e-07 & 3.95e-20\\
3 & 6 & 0.80 & 1.992937e-04 & 3.335460e-05 & 1.067795e-05 & 2.109941e-13 & 2.092584e-03 & 1.14e-11 & 3.14e-24\\
4 & 1 & 0.74 & 1.771500e-04 & 2.964853e-05 & 4.904504e-06 & 4.722540e-14 & 1.948650e-03 & 4.97e-07 & 2.88e-20\\
4 & 2 & 0.80 & 1.771500e-04 & 2.964853e-05 & 1.002141e-05 & 1.894322e-13 & 1.860075e-03 & 2.32e-07 & 5.61e-20\\
4 & 3 & 0.81 & 9.226561e-05 & 1.544194e-05 & 4.429583e-06 & 3.382568e-14 & 1.014922e-03 & 1.74e-09 & 7.33e-23\\
4 & 4 & 0.79 & 2.140562e-04 & 3.582531e-05 & 1.653331e-05 & 4.949205e-13 & 2.140562e-03 & 2.19e-12 & 1.37e-24\\
4 & 5 & 0.78 & 1.291718e-04 & 2.161872e-05 & 5.383696e-06 & 5.697885e-14 & 1.097961e-03 & 5.25e-10 & 3.79e-23\\
4 & 6 & 0.76 & 1.771500e-04 & 2.964853e-05 & 6.388444e-06 & 9.570476e-14 & 1.948650e-03 & 6.39e-08 & 7.93e-21\\
5 & 1 & 0.66 & 1.328625e-04 & 2.223640e-05 & 2.056577e-06 & 8.775943e-15 & 1.062900e-03 & 2.64e-08 & 3.03e-22\\
5 & 2 & 0.72 & 1.144094e-04 & 1.914801e-05 & 3.830449e-06 & 2.609694e-14 & 1.029684e-03 & 1.88e-09 & 6.41e-23\\
5 & 3 & 0.73 & 1.033375e-04 & 1.729498e-05 & 4.269049e-06 & 3.062480e-14 & 1.033375e-03 & 2.50e-09 & 9.36e-23\\
5 & 4 & 0.72 & 2.066750e-04 & 3.458995e-05 & 1.094161e-05 & 2.103895e-13 & 2.066750e-03 & 8.81e-11 & 2.40e-23\\
5 & 5 & 0.71 & 1.919125e-04 & 3.211924e-05 & 8.565614e-06 & 1.320616e-13 & 1.919125e-03 & 1.51e-08 & 2.52e-21\\
5 & 6 & 0.70 & 1.439343e-04 & 2.408943e-05 & 3.888326e-06 & 3.451748e-14 & 1.079508e-03 & 1.38e-09 & 5.90e-23\\
6 & 1 & 0.59 & 1.771500e-04 & 2.964853e-05 & 2.193874e-06 & 1.019455e-14 & 2.037225e-03 & 2.24e-08 & 2.75e-22\\
6 & 2 & 0.65 & 8.488436e-05 & 1.420659e-05 & 1.660222e-06 & 4.346395e-15 & 9.337279e-04 & 1.24e-04 & 7.09e-19\\
6 & 3 & 0.66 & 8.857498e-05 & 1.482426e-05 & 1.695908e-06 & 6.855800e-15 & 9.743248e-04 & 2.46e-06 & 2.03e-20\\
6 & 4 & 0.65 & 2.140562e-04 & 3.582531e-05 & 1.192707e-05 & 2.027028e-13 & 2.140562e-03 & 2.16e-11 & 5.59e-24\\
6 & 5 & 0.65 & 9.226561e-05 & 1.544194e-05 & 1.994275e-06 & 7.521929e-15 & 1.014922e-03 & 3.06e-08 & 2.97e-22\\
6 & 6 & 0.64 & 1.181000e-04 & 1.976569e-05 & 2.045468e-06 & 1.035221e-14 & 1.062900e-03 & 1.33e-08 & 1.76e-22\\
\hline
\end{tabular}
}
\caption{ \label{tab:standard:quadratics}
Application of Theorem~\ref{theo:KAM} using Algorithm~\ref{algo:validation} for the standard map~\eqref{eq:standard:map}.
We consider several interval frequencies $\iomega$ enclosing intervals given
by Equation~\eqref{eq:omega:a,b}. For each curve, we present the larger value of $\eps$ (second column) for which
we can apply the KAM theorem with {\em ad hoc} R\"ussmann estimates~\eqref{eq:russmann:mejor} using a grid of size $\NF=32768$. 
}
\end{table}

In Table~\ref{tab:standard:quadratics} we present the application of the KAM
theorem for other invariant curves. Specifically, we consider the 
rotation numbers that have been characterized in Table~\ref{tab:apenA:1d}. We
fix 
$\NF=32768$ and we show the maximum number of $\eps$
for which we have been able to apply the KAM
theorem
(taking jumps of length $0.01$ in $\eps$).
Computations are performed using interval arithmetics with 267 bits.
We perform the CAP using both the
classic R\"ussmann estimates in~\eqref{eq:classic:russmann}
and the \emph{ad hoc} estimates in~\eqref{eq:russmann:mejor}.
Numerical approximations of the critical values of
some of these curves have been reported in~\cite{FoxM14}, for example,
$\eps_c \simeq 0.957447$ for $\omega_{2,2}$ (we prove existence for
$\eps=0.95$), $\eps_c \simeq 0.87608$ for $\omega_{1,2}$ (we prove existence
for $\eps=0.87$), $\eps_c \simeq 0.89086$ fur $\omega_{3,3}$ (we prove
existence for $\eps=0.88$), and $\eps_c \simeq 0.77242$ for $\omega_{1,3}$ (we
prove existence for $0.76$). As in Theorem~\ref{theo:golden}, we can obtain sharper rigorous lower
bounds of these critical values by increasing the number of Fourier
coefficients.
We remark again that the use of \emph{ad hoc}
R\"ussmann estimates represents a significant advantage in order to apply the
KAM theorem. In some cases we improve up to $17$ orders of magnitude the size
of the smallness condition.
The computational time of these CAPs ranges between 75 and 123 seconds
in a single processor Intel(R) Xeon(R) CPU at 2.40 GHz.

\subsection{Non-twist standard map}\label{ssec:nontwist:standard}

The second application falls in the context of the so-called non-twist maps.  It
is well known that there is an analogue of KAM theory in the non-twist scenario
(see for example \cite{DelshamsL00,GonzalezHL13,Simo98a}).
The loss of the twist condition introduces different properties with respect to
the twist case, for example the fact that the Birkhoff Graph Theorem does not
apply 
and folded invariant curves are observed. 
A classic mechanism that creates 
such folded invariant curves is called
\emph{reconnection}
(see~\cite{CastilloNGM96,Simo98a}).
Reconnection is a global bifurcation of the invariant
manifolds of two or more distinct hyperbolic periodic orbits having the same
winding number.
This creates a meandering
region having folded quasi periodic curves.
Among these orbits, of special interest is the one
that corresponds to an invariant curve having a local extremum
in the rotation number, called shearless invariant curve.
In this section we are not interested in
applying singular KAM theory to study the shearless invariant curves (in the
spirit of~\cite{GonzalezHL13}) but we will consider
invariant curves inside the meandering region that are non-degenerate (in the spirit of~\cite{Simo98a}).
The study of shearless invariant curves and their bifurcations can be
performed in combination with the tools in~\cite{GonzalezHL13}.

Let us consider the non-twist standard map
\begin{equation}\label{eq:nontwist:map}
\begin{array}{rcl}
F:
\TT \times \RR & 
\longrightarrow &
\TT \times \RR \\
(x,y) & \longmapsto & (\bar x, \bar y)=\left(x + (\bar y+\lambda_1)(\bar y+\lambda_2), y - \frac{\eps}{2\pi} \sin (2\pi x)\right)
\end{array}
\end{equation}
where $(x,y)\in \TT \times \RR$ are phase space coordinates and $\lambda_1,\lambda_2 \in \RR$ and $\eps>0$ are parameters. 
Although this family is not generic (it contains just one harmonic),
it describes the essential features of non-twist
systems with a local quadratic extremum in the rotation number.

For the non-twist standard map, given by Equations~\eqref{eq:nontwist:map}, we have $\A=\TT \times \RR$,
$\aform=\aform_0= y\dif x$ and $\sform= \sform_0= \dif y \wedge \dif x$.
Hence, we take $\sub{c}{\Omega}=1$, $\sub{c}{\Dif \Omega}=0$, $\sub{c}{\Dif a}=1$ and $\sub{c}{\Dif^2 a}=0$.

Contrary to the example discussed in Section~\ref{ssec:standard}, we notice
that the tangent vectors to the invariant curves are not 
in general 
in the horizontal direction, since invariant curves are not expected to be graphs. 
Using the ambient structure, we select the transversal vectors $N_0 : \TT \rightarrow \RR^{2\times 1}$ as
\begin{equation}\label{eq:normalN0:nontwist}
N_0(\theta)= \Omega_0 \Dif K(\theta),
\end{equation}
and it is clear that $\Dif K(\theta)$ and $N_0(\theta)$ form a basis of $\RR^2$
for every $\theta \in \TT$. 

\begin{figure}[!t]
\centering
\includegraphics[scale=0.3,angle=-90]{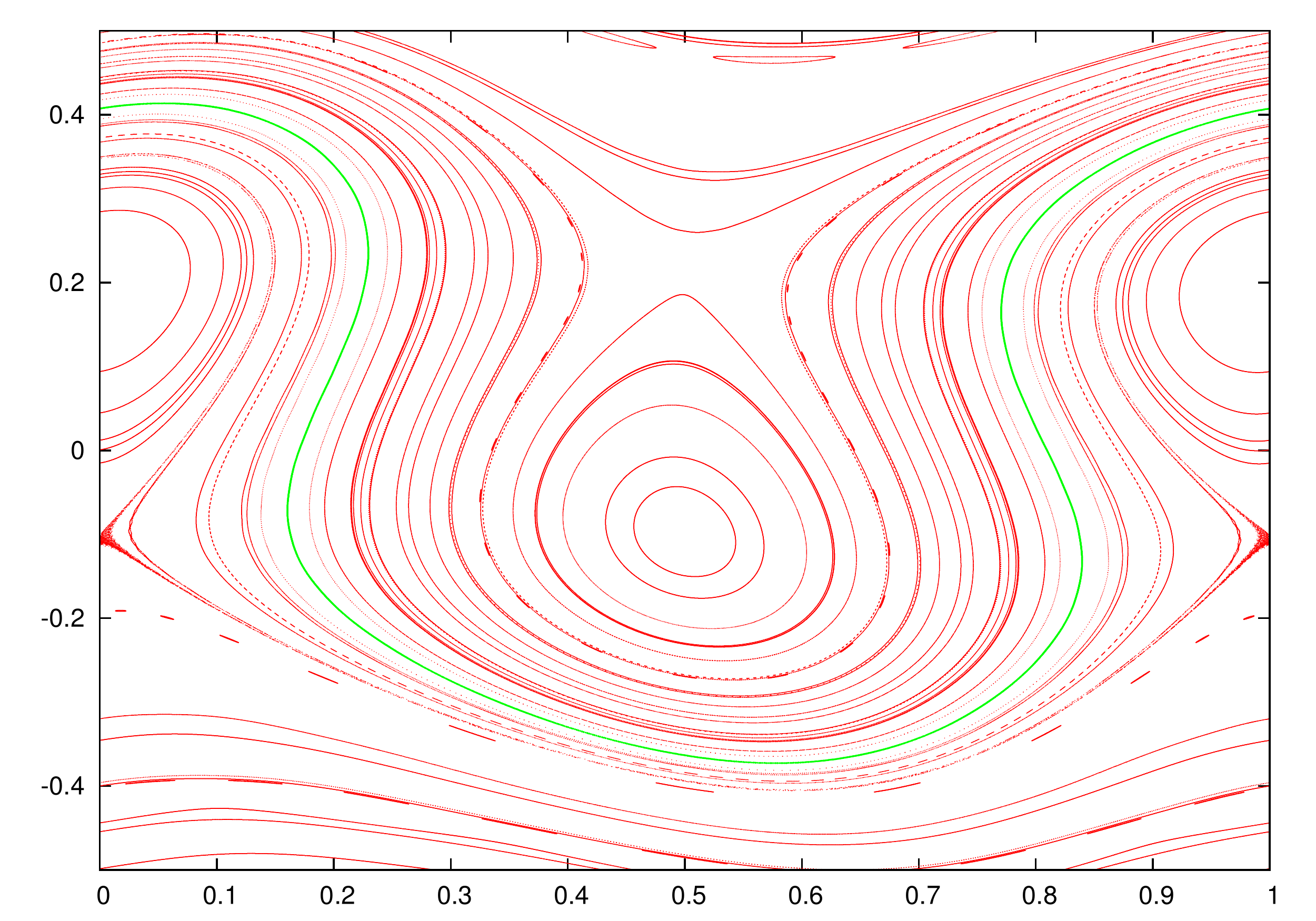}
\caption{{\footnotesize
Meandering invariant curves for the non-twist standard map~\eqref{eq:nontwist:map}
for parameters
$\lambda_1=0.1$, $\lambda_2=-0.2$ and $\eps=0.45$. 
}} \label{fig:meandering1}
\end{figure}

In Step 0 of Algorithm~\ref{algo:validation}, we introduce the domain $\B$ as
\[
\B = \{ (x,y) \in \CC/\ZZ \times \CC \,:\, 
|{\im\,x}|\leq r_1 \, , \, |y|\leq r_2 \},
\]
and the domain $\hat \B$ as
\[
\hat \B = \{ (x,y) \in \CC/\ZZ \times \CC \,:\, 
|{\im\,x}|\leq \hat r_1 \, , \, |y|\leq \hat r_2 \}.
\]
where
\[
r_1 = d_{\B}+\rho+\snorm{K^x_p}_{F,\rho}, \qquad
r_2 = d_{\B}+\snorm{K^y_p}_{F,\rho}, \qquad
\hat r_1= \hat\rho+\snorm{K^x_{p}}_{F,\hat \rho}, \qquad
\hat r_2= \snorm{K^y_{p}}_{F,\hat \rho}.
\]
Then, global estimates of the symplectic map \eqref{eq:nontwist:map} in these domains are characterized as follows
\begin{align*}
& \snorm{\Dif F}_\B \leq  \sub{c}{\Dif F} := \max \bigg\{
1+|\partial_x\bar y|,
1+(2 |\bar y|+|\lambda_1+\lambda_2|) (1+|\partial_x \bar y|),
\bigg\}, \\
& \snorm{\Dif^2 F}_\B \leq  \sub{c}{\Dif^2 F} := \max \bigg\{ 
|\partial_{xx} \bar y|,
2 + 2 |\partial_x \bar y| + 2 |\partial_x \bar y|^2 +
(2|\bar y|+|\lambda_1+\lambda_2|) |\partial_{xx}\bar y | 
\bigg\}, \\
& \snorm{F_p}_{\hat \B} \leq  \sub{\hat c}{F_p} := 
\max \bigg\{ 
\hat r_2 + \eps \cosh(2\pi \hat r_1),
(\hat r_2 + \eps \cosh(2\pi \hat r_1)+|\lambda_1|)
(\hat r_2 + \eps \cosh(2\pi \hat r_1)+|\lambda_2|)
\bigg\},
\end{align*}
where
\[
|\bar y| = r_2 + \frac{\epsilon}{2\pi} \cosh (2\pi r_1), \quad |\partial_x \bar y|= \epsilon \cosh(2\pi r_1), \quad
|\partial_{xx}\bar y| = 2\pi \epsilon \cosh(2\pi r_1).
\]

Let us characterize the constants that appear in  Step 2 of Algorithm~\ref{algo:validation}.
Using Equation~\eqref{eq:normalN0:nontwist} and the fact that $\Omega_0$ is constant, we control
the analytic norms related to $N_0$ 
using
\[
\sub{c}{N_0}=\snorm{N_0}_{F,\rho}=\snorm{\Dif K}_{F,\rho},\quad 
\sub{c}{N_0^\top}=\snorm{\Dif K^\top}_{F,\rho}, \quad
\sub{\hat c}{N_0}=\snorm{\Dif K}_{F,\hat \rho}, \quad 
\sub{\hat c}{N_0^\top}=\snorm{\Dif K^\top}_{F,\hat \rho}.
\]
Notice that 
we can evaluate
the above norms directly since $K_p$ is a trigonometric polynomial.
We also observe that $N_0^\top (\Omega \circ K) N_0$
vanishes identically, so we take $\sub{c}{N_0^\top (\Omega \circ K) N_0}=0$.

Let us describe a simple general approach to obtain candidates of invariant
curves using a global (non-perturbative) analysis of the problem. Given fixed
values of $\lambda_1$, $\lambda_2$ and $\eps$, we take an initial condition
$(x_0,y_0)$.  To this initial condition we associate a rotation number
(frequency) using the extrapolation method of~\cite{LuqueV09}. Notice that we
can define a (non-equispaced) sampling of a parameterization as follows
$K(n\omega)=(x_n,y_n)$, where $(x_n,y_n)$ is the orbit of the previous initial
condition and $\omega$ is the computed frequency. Then, we obtain a first
equispaced sampling $\{K_{p,j}\}$ by means of polynomial interpolation (as used
for example in~\cite{OlveraP08}).  If the error of invariance of this
approximation is not good enough, we apply the
parameterization method.

From now on, we fix $\lambda_1=0.1$, $\lambda_2=-0.2$ and
$\eps=0.45$.  In Figure~\ref{fig:meandering1} we show the iteration of several
initial conditions and we observe meandering (folded) invariant curves. 
We consider one of the two invariant curves having rotation number
$\omega=\tfrac{\sqrt{5}-1}{32}$, which
corresponds to the green curve in Figure~\ref{fig:meandering1}. Next
we show that,
for the above values of the parameters,
there exists a true invariant curve nearby.
To this end, we obtain an approximation of the corresponding
parameterization using $\NF=2048$ and with a numerically estimated error of
invariance of $10^{-42}$ (using the norm $\|\cdot\|_{F, 0}$). 
Computations are performed using interval arithmetics with 267 bits.
As input
of Algorithm~\ref{algo:validation} we enclose $\omega$ with
a tight interval of length $10^{-80}$, and we use the parameters
\begin{align*}
\rho={} & 1.223945 \cdot 10^{-3}, \\
\delta={} & 2.048444 \cdot 10^{-4}, \\
\sigma={} & 1+1.601973 \cdot 10^{-11}, \\
d_\B={} & 8.333835 \cdot 10^{-22}, \\
\hat \rho={} & 1.835918 \cdot 10^{-2}.
\end{align*}
We use also 
the improved R\"ussmann estimates in~\eqref{eq:russmann:mejor}. The obtained
result is
\[
\frac{\mathfrak{C}_1 \sub{b}{E}}{\gamma^4 \rho^{4 \tau}} \leq 0.0343, \qquad 
\frac{\mathfrak{C}_2 \sub{b}{E}}{\gamma^2 \rho^{2 \tau}} \leq 3.78 \cdot 10^{-23}.
\]
Then, we can ensure the existence of a meandering invariant curve
close to the green curve in Figure~\ref{fig:meandering1}.
Notice that, even though the curve is quite regular, we need a very good approximation
of the invariant curve in order to apply the KAM theorem.
The reason is that the twist condition around the curve is quite weak. Indeed,
we obtain $c_T\leq 2388.12$ and $\sigma_T \leq 33.11$ that propagate
significantly along the computation of $\mathfrak{C}_1$ and $\mathfrak{C}_2$.
Still, we are able to conclude that the distance, measure by the analytic norm, between the true invariant
curve and the initial approximation is controlled by $3.78 \cdot 10^{-23}$.
The computational time of this CAP is 65 seconds
in a single processor Intel(R) Xeon(R) CPU at 2.40 GHz.

\subsection{Froeschl\'e map}\label{ssec:froeschle}

Next we illustrate the use of Algorithm~\ref{algo:validation} to prove
existence of Lagrangian invariant tori in a higher dimensional case. We
consider the so-called Froeschl\'e map (see~\cite{Froeschle72}), which consists
in two coupled standard maps, given by
\begin{equation}\label{eq:map:froeschle}
\begin{array}{rcl}
F: 
\TT^2 \times \RR^2 & 
\longrightarrow &
\TT^2 \times \RR^2 \\
\begin{pmatrix}
x_1 \\
x_2 \\
y_1 \\
y_2
\end{pmatrix}
 &\longrightarrow &
\begin{pmatrix}
\bar x_1 \\
\bar x_2 \\
\bar y_1 \\
\bar y_2
\end{pmatrix}
=
\begin{pmatrix}
x_1+\bar y_1 \\
x_2+\bar y_2 \\
y_1 - \frac{\lambda_1}{2\pi} \sin (2\pi x_1) - \frac{\eps}{2\pi} \sin (2\pi(x_1+x_2)) \\
y_2 - \frac{\lambda_2}{2\pi} \sin (2\pi x_2) - \frac{\eps}{2\pi} \sin (2\pi(x_1+x_2))
\end{pmatrix},
\end{array}
\end{equation}
where $(x,y)\in \TT^2 \times \RR^2$ are phase space coordinates and $\lambda_1,\lambda_2,\eps$ are parameters. 
This family has been extensively studied in the literature as a model to understand instability channels and
the destruction of invariant tori \cite{Haro98,KanekoB85,MacKayMS89,Tompaidis96}.

\begin{table}[!t]
\centering
{\scriptsize
\begin{tabular}{|c || c c c c c c || c c|}
\hline
$\epsilon$ &
$\Ni{1} \times \Ni{2}$ &
$\rho$ &
$\delta$ &
$\sigma-1$ &
$d_{\B}$ &
$\hat \rho$ &
$\frac{\mathfrak{C}_1 \sub{b}{E}}{\gamma^4 \rho^{4 \tau}}  $ &
$\frac{\mathfrak{C}_2 \sub{b}{E}}{\gamma^2 \rho^{2 \tau} } $ \\ 
\hline \hline
0.005 & 128$\times$128 & 5.037106e-02 & 8.325794e-03 & 9.971678e-09 & 1.304797e-13 & 2.518553e-01 & 1.22e-11 & 1.65e-24 \\
0.010 & 128$\times$128 & 4.149000e-02 & 6.943934e-03 & 2.991465e-09 & 2.984918e-14 & 2.489400e-01 & 8.18e-11 & 2.53e-24 \\
0.015 & 128$\times$128 & 3.381873e-02 & 5.660038e-03 & 1.007737e-09 & 7.594011e-15 & 2.198217e-01 & 3.17e-09 & 2.49e-23 \\
0.020 & 128$\times$128 & 2.640907e-02 & 4.419929e-03 & 3.038407e-10 & 1.678977e-15 & 1.980681e-01 & 1.51e-05 & 2.62e-20 \\
0.025 & 256$\times$256 & 1.480798e-02 & 2.447599e-03 & 2.798552e-09 & 8.363941e-15 & 1.184638e-01 & 4.82e-11 & 4.18e-25 \\
0.030 & 512$\times$512 & 6.403442e-03 & 1.071706e-03 & 1.635121e-10 & 2.072109e-16 & 1.088585e-01 & 3.38e-41 & 7.25e-57 \\
0.035 & 512$\times$512 & 6.339408e-03 & 1.060989e-03 & 1.271988e-10 & 1.484302e-16 & 1.014305e-01 & 8.06e-38 & 1.24e-53 \\
0.040 & 512$\times$512 & 6.339408e-03 & 1.060989e-03 & 1.042153e-10 & 1.129674e-16 & 1.014305e-01 & 1.13e-37 & 1.32e-53 \\
0.045 & 512$\times$512 & 6.339408e-03 & 1.060989e-03 & 8.603873e-11 & 8.652147e-17 & 1.077699e-01 & 9.77e-42 & 8.75e-58 \\
0.050 & 512$\times$512 & 6.339408e-03 & 1.060989e-03 & 7.143414e-11 & 6.649623e-17 & 1.077699e-01 & 1.71e-41 & 1.17e-57 \\
0.055 & 512$\times$512 & 6.339408e-03 & 1.060989e-03 & 5.955329e-11 & 5.127525e-17 & 1.077699e-01 & 2.98e-41 & 1.58e-57 \\
0.060 & 512$\times$512 & 6.339408e-03 & 1.060989e-03 & 5.050749e-11 & 4.016518e-17 & 1.077699e-01 & 1.34e-40 & 5.57e-57 \\
0.065 & 512$\times$512 & 6.339408e-03 & 1.060989e-03 & 4.239981e-11 & 3.109553e-17 & 1.014305e-01 & 1.96e-36 & 6.31e-53 \\
0.070 & 512$\times$512 & 6.339408e-03 & 1.060989e-03 & 3.566570e-11 & 2.408118e-17 & 9.509112e-02 & 7.89e-32 & 1.97e-48 \\
0.075 &1024$\times$512 & 6.085832e-03 & 1.018549e-03 & 1.128252e-14 & 6.748840e-21 & 7.911581e-02 & 1.45e-16 & 1.01e-36 \\
\hline
\end{tabular}
\caption{ {\footnotesize 
Application of Theorem~\ref{theo:KAM} using
Algorithm~\ref{algo:validation} around the cubic invariant torus 
of the Froeschl\'e map~\eqref{eq:map:froeschle}
for different
values of $\eps$. We use the R\"ussmann estimates given in~\eqref{eq:russmann:mejor}.
}}
\label{tab:froeschle} }
\end{table}

For the Froeschl\'e map, given by Equation~\eqref{eq:map:froeschle}, we have $\A=\TT^2 \times \RR^2$,
$\aform=\aform_0= y_1\dif x_1+y_2 \dif x_2$ and $\sform= \sform_0= \dif y_1 \wedge \dif x_1+\dif y_2 \wedge \dif x_2$.
Hence, we take $\sub{c}{\Omega}=1$, $\sub{c}{\Dif \Omega}=0$, $\sub{c}{\Dif a}=1$ and $\sub{c}{\Dif^2 a}=0$.

Using the ambient structure, we select the transversal vectors $N_0 : \TT \rightarrow \RR^{4\times 2}$ as
\begin{equation}\label{eq:normalN0:froeschle}
N_0(\theta)= \Omega_0 \Dif K(\theta),
\end{equation}
and it is clear that $\Dif K(\theta)$ and $N_0$ form a basis of $\RR^4$
for every $\theta \in \TT^2$. 

In Step 0 of Algorithm~\ref{algo:validation}, we introduce the domain $\B$ as
\[
\B = \{ (x,y) \in \CC^2/\ZZ^2 \times \CC^2 \,:\, 
|{\im\,x_1}|\leq r_1 \, , 
|{\im\,x_2}|\leq r_2 \, , 
\, |y_1|\leq r_3 \, ,
\, |y_2|\leq r_4 
\},
\]
and the domain $\hat \B$ as
\[
\hat \B = \{ (x,y) \in \CC^2/\ZZ^2 \times \CC^2 \,:\, 
|{\im\,x_1}|\leq \hat r_1 \, , 
|{\im\,x_2}|\leq \hat r_2 \, , 
\, |y_1|\leq \hat r_3 \, , 
\, |y_2|\leq \hat r_4 
\},
\]
where
\[
r_1 = d_{\B}+\rho+\snorm{K^{x_1}_p}_{F,\rho}, \quad
r_2 = d_{\B}+\rho+\snorm{K^{x_2}_p}_{F,\rho}, \quad
r_3 = d_{\B}+\snorm{K^{y_1}_p}_{F,\rho}, \quad
r_4 = d_{\B}+\snorm{K^{y_2}_p}_{F,\rho},
\]
and
\[
\hat r_1 = \hat \rho+\snorm{K^{x_1}_p}_{F,\hat \rho}, \quad
\hat r_2 = \hat \rho+\snorm{K^{x_2}_p}_{F,\hat \rho}, \quad
\hat r_3 = \snorm{K^{y_1}_p}_{F,\hat \rho}, \quad
\hat r_4 = \snorm{K^{y_2}_p}_{F,\hat \rho}.
\]
Then, global estimates of the the symplectic map \eqref{eq:map:froeschle} in these domains are characterized as follows
\begin{align*}
& \snorm{\Dif F}_\B \leq  \sub{c}{\Dif F} := \max \bigg\{ 2 + \lambda_1 c_1+ 2 \eps c_3, 2 + \lambda_2 c_2 + 2 \eps c_3 \bigg\}, \\
& \snorm{\Dif^2 F}_\B \leq  \sub{c}{\Dif^2 F} := \max \bigg\{ 2 \pi \lambda_1 c_1 + 4 \pi \eps c_3 , 2\pi \lambda_2 c_2 + 4 \pi \eps c_3
\bigg\}, \\
& \snorm{F_p}_{\hat \B} \leq  \sub{\hat c}{F_p} := \max \bigg\{
\hat r_3 + \frac{\lambda_1}{2\pi} \hat c_1 + \frac{\eps}{2\pi} \hat c_3,
\hat r_4 + \frac{\lambda_3}{2\pi} \hat c_3 + \frac{\eps}{2\pi} \hat c_3,
\bigg\},
\end{align*}
where
\[
c_1 = \cosh(2\pi r_1), \qquad
c_2 = \cosh(2\pi r_2), \qquad
c_3 = \cosh(2\pi (r_1+r_2)),
\]
and
\[
\hat c_1 = \cosh(2\pi \hat r_1), \qquad
\hat c_2 = \cosh(2\pi \hat r_2).
\]

In Step 2 of Algorithm~\ref{algo:validation}, we observe that 
\[
N_0(\theta)^\top \Omega_0 N_0(\theta) = \Dif K(\theta)^\top \Omega_0 \Dif K(\theta),
\]
where we used $\Omega_0^\top \Omega_0=I_4$. Using Theorem~\ref{ADFT} we
take
\[
\sub{c}{N_0^\top (\Omega \circ K) N_0}= \snorm{\Dif K^\top \Omega_0 \Dif K}_{F,\rho}+ C_{\NF}(\rho,\hat \rho)
\snorm{\Dif K^\top}_{F,\hat \rho}
\snorm{\Dif K}_{F,\hat \rho}.
\]
Moreover, using Equation~\eqref{eq:normalN0:froeschle}, we control the remaining analytic norms related to $N_0$ with
the constants
\[
\sub{c}{N_0}=\snorm{\Dif K}_{F,\rho},\quad 
\sub{c}{N_0^\top}=\snorm{\Dif K^\top}_{F,\rho}, \quad
\sub{\hat c}{N_0}=\snorm{\Dif K}_{F,\hat \rho}, \quad 
\sub{\hat c}{N_0^\top}=\snorm{\Dif K^\top}_{F,\hat \rho}.
\]

\begin{table}[!t]
\centering
{\scriptsize
\begin{tabular}{|c || c c c|}
\hline
$\Ni{1} \times \Ni{2}$ &
Precision (bits) &
Tolerance & CPU time (sec)\\
\hline \hline
128$\times$128  & 267 & 1e-33 & 246 \\
256$\times$256  & 347 & 1e-38 & 751 \\
512$\times$512  & 533 & 1e-66 & 3628\\
1024$\times$512 & 533 & 1e-66 & 6992\\
\hline
\end{tabular}
\caption{ {\footnotesize
Auxiliary implementation parameters of the CAPs corresponding to Table~\ref{tab:froeschle}.
The CAPs are performed in a single processor Intel(R) Xeon(R) CPU at 2.40 GHz.
}}
\label{tab:froeschle2}}
\end{table}

\begin{figure}[!t]
\centering
\begin{tabular}{cc}
\includegraphics[scale=0.4]{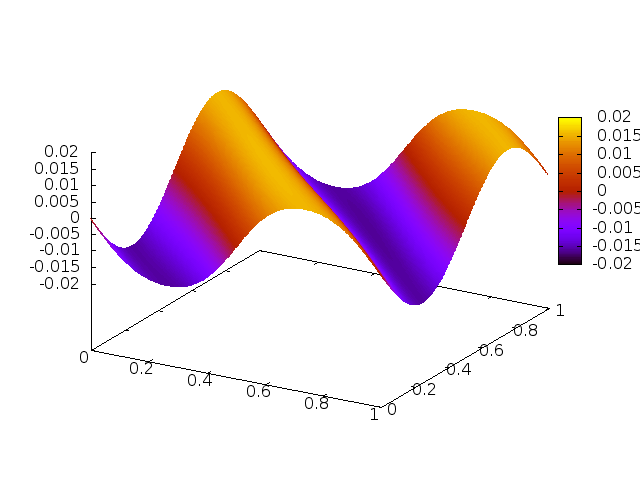}
&
\includegraphics[scale=0.4]{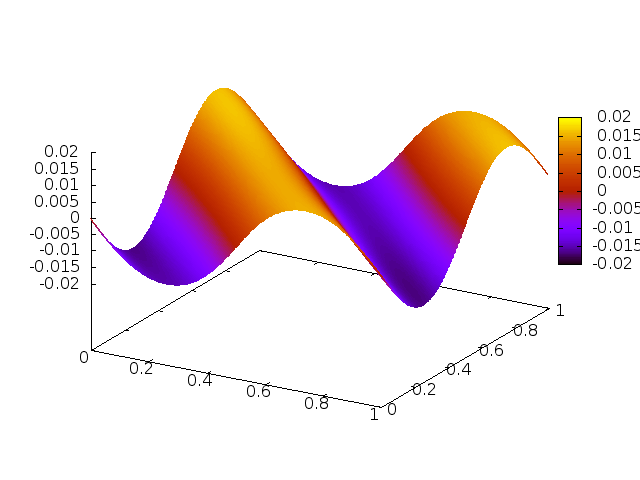}
\\
\includegraphics[scale=0.4]{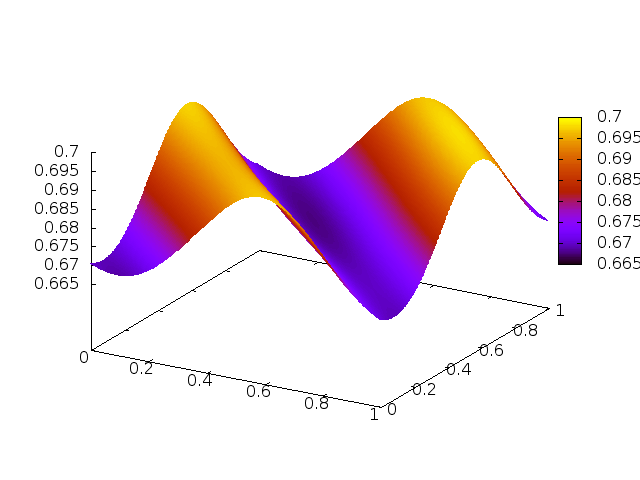}
&
\includegraphics[scale=0.4]{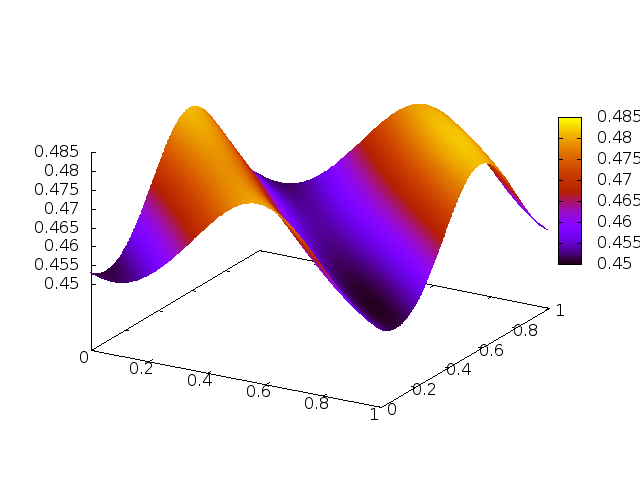}
\end{tabular}
\caption{{\footnotesize
Last validated invariant torus in Table~\ref{tab:froeschle}. 
Left-Top plot: $(\theta_1,\theta_2) \mapsto K_p^{x_1}(\theta_1,\theta_2)$.
Right-Top plot: $(\theta_1,\theta_2) \mapsto K_p^{x_2}(\theta_1,\theta_2)$.
Left-Bottom plot: $(\theta_1,\theta_2) \mapsto K_p^{y_1}(\theta_1,\theta_2)$.
Left-Bottom plot: $(\theta_1,\theta_2) \mapsto K_p^{y_2}(\theta_1,\theta_2)$.
}}\label{fig:torus2D}
\end{figure}

We consider the frequency vector $\omega=(\nu,\nu^2)$
where $\nu$ is the so-called \emph{cubic golden number} (the only real root of $x^3+x-1=0$).
We fix the parameters $\lambda_1=0.01$ and $\lambda_2=0.02$ and we validate
the invariant torus with frequency $\omega$ for different values of $\epsilon$.
Several validations are summarized in Table~\ref{tab:froeschle}
whose format is analogous as in previous examples.
In Table~\ref{tab:froeschle2} we provide some 
complementary information about the execution of the computer program
depending on the number of Fourier coefficients.
We enclose the components of $\omega$ with intervals of radius
$\mu^2$, where $\mu$ is the tolerance requested in the numerical
computation of the parameterization (see Table~\ref{tab:froeschle2}).
The last validated invariant torus, corresponding to $\eps=0.075$,
is shown in Figure~\ref{fig:torus2D}.

\section*{Acknowledgements}

The authors are very grateful to 
R. Calleja, 
R. de la Llave, 
and J. Villanueva
for useful and fruitful discussions along the last years.
We would like to acknowledge financial support from
the Spanish grant MTM2012-32541  and
the Catalan grant 2014-SGR-1145.
J.-Ll. Figueras acknowledges the partial support from Essen, L. and C.-G., for mathematical studies.
Moreover, A.L. acknowledges support from a postdoctoral
position in the ERC Starting grant 335079.
We acknowledge Albert Granados and the use of the UPC Applied Math cluster system for
research computing (see \url{http://www.pagines.ma1.upc.edu/~eixam/}).

\addcontentsline{toc}{section}{References}
\bibliographystyle{plain}
\bibliography{references}

\appendix

\section{A heuristic selection of parameters to validate invariant tori}\label{app:optimization}

In this appendix we describe a direct method to select the implementation
parameters $\rho$, $\delta$, $\sigma$, $\rho_{\infty}$, $d_{\B}$, and $\hat \rho$, required in
Input 4 of Algorithm~\ref{algo:validation} to rigorously validate an invariant torus.
The idea consists in using the structure of the constants that appear in 
Step 4 (c.f. Section~\ref{ssec:details:step4}).
We do not claim that the procedure described below is optimal,
but allows us to obtain suitable values of the parameters at a moderate computational
cost. 

We assume that the term $2(a_3)^{3\tau+1} \gamma^3 \rho^{3
\tau-1} C_3$ does not contribute to the constant $\mathfrak{C}_1$
and we look for parameters such that
the constants in~\eqref{eq:C3C4C5} satisfy $\mathfrak{C}_3=\mathfrak{C}_4=\mathfrak{C}_5$.
The neglected
term is in general very small 
(it stands for the control of the
approximate Lagrangian character of the approximated torus)
and 
this assumption allows us to
simplify the
heuristic analysis of the constants.

We proceed in analogy with the procedure described in~\cite{HaroCFLM}.
We observe that the dependence on $a_2$ is very simple: it appears
only in the expression $a_3 = 3 \tfrac{a_1}{a_1-1} \tfrac{a_2}{a_2-1}$ 
and in the final strip of analyticity $\rho_{\infty}=\rho/a_2$.
In order to look for the limit condition, we take $a_2=\infty$ (so $\rho_\infty=0$) in all
subsequent computations. 

Now we can describe a very simple algorithm to obtain suitable values of the parameters
$\rho$, $\delta$, $\sigma$, $d_\B$ and $\hat \rho$
for a given parameterization $K$ in a grid of size $\NF=(\Ni{1},\ldots,\Ni{n})$:
\begin{enumerate}
\item We take $\rho_0=- \log(\snorm{E}_{F,0})/(2 \pi N)$, where $E$ is the error of invariance
and $N=\max_{i} \{ \Ni{i}\}$. This will be the initial value of $\rho$.

\item For a given value $\rho$, we consider
values $\delta \in [\tfrac{\rho}{6.5},\tfrac{\rho}{4.5}]$ (recall that $a_3=\tfrac{\rho}{\delta}$ and $a_1= \tfrac{a_3}{a_3-3}$). For any of these
values $(\rho,\delta)$ we compute $\sigma$ and $d_\B$ solving
the equations $\mathfrak{C}_4=\mathfrak{C}_5$ and $\mathfrak{C}_3=\mathfrak{C}_4$,
which are respectively writen as follows:
\begin{align}
& \sigma_*(1-a_1^{-2\tau})
d_\B-(\sigma-1) \delta (1-a_1^{1-2\tau})=0,
\label{eq:cond:op1} \\
& \sigma_* (a_3)^{2\tau+1} \gamma^2 \rho^{2\tau-1} \hat C_2 - (\sigma-1)(1-a_1^{1-2\tau})(a_1a_3)^{4 \tau}\hat C_5 =0.
\label{eq:cond:op2}
\end{align}
Let us recall that $\sigma_*$, $\hat C_2$ and $\hat C_5$ depend on
$\rho$, $\delta$, $\sigma$, $d_\B$, and $\hat \rho$.
In order to avoid the dependence on $\hat \rho$ in the above expression,
we take
$C_\NF(\rho, \hat \rho)=0$ when computing
these constants. A suitable value of $\hat \rho$ is fixed later.
Then,  we select the value of
$\delta$ that minimizes the expression $\mathfrak{C}_1 \gamma^{-4} \rho^{-4\tau} \snorm{E}_\rho$.

\item 
If $\mathfrak{C}_1 \gamma^{-4} \rho^{-4\tau} \snorm{E}_\rho \geq 1$, 
we decrease the value of
$\rho$ and repeat step II, thus obtaining a new
value of $\mathfrak{C}_1$.
 We proceed
until we find that $\mathfrak{C}_1 \gamma^{-4} \rho^{-4\tau} \snorm{E}_\rho<1$.
If at any point we reach a minimum of the function $\mathfrak{C}_1
\gamma^{-4} \rho^{-4\tau} \snorm{E}_\rho$ then we stop the computations.
If this condition is not satisfied at the minimum,
then we need a better approximation of the invariant torus. 

\item
Assume that we have obtained values of $\rho$, $\delta$, $\sigma$, and $d_\B$
as above. Then, we take a sequence of increasing values of $\hat \rho$ (starting at a value slightly greater than $\rho$) and compute the constant $C_\NF(\rho,\hat \rho)$. Then we compute again the constant $\mathfrak{C}_1$ and the bound $\sub{b}{E}$ (see Section~\ref{ssec:details:step1}).
We select a value of $\hat \rho$ that minimizes the expression $\mathfrak{C}_1 \gamma^{-4} \rho^{-4\tau} \sub{b}{E}$.
\end{enumerate}
\end{document}